\documentclass[11pt]{amsart}
\pdfoutput=1
\usepackage[english]{babel}
\usepackage{graphicx,amsmath}
\def \R {{\mathbb R}}
\def \C {{\mathbb C}}
\def \GTNN {{Gr^{\mbox{\tiny TNN}} (k,n)}}
\def \S {{\mathcal S}_{\mathcal M}^{\mbox{\tiny TNN}}}
\numberwithin{equation}{section}

\newtheorem{theorem}{Theorem}[section]
\newtheorem{corollary}[theorem]{Corollary}
\newtheorem{lemma}[theorem]{Lemma}
\newtheorem{proposition}[theorem]{Proposition}
\newtheorem{remark}[theorem]{Remark}

\newtheorem{definition}[theorem]{Definition}
\newtheorem{notation}[theorem]{Notation}
\newtheorem{example}[theorem]{Example}
\newtheorem{conjecture}[theorem]{Conjecture}

\title[Geometric relations on plabic graphs]{Geometric nature of relations on plabic graphs and totally non-negative  Grassmannians}
\author{Simonetta Abenda}
\address{Dipartimento di Matematica and Alma Mater Research Center on Applied Mathematics, Universit\`a di Bologna, Italy\\ INFN, sez. di Bologna, Italy}
\email{simonetta.abenda@unibo.it}
\author{Petr G. Grinevich}
\address {Steklov Mathematical Institute of Russian Academy of Sciences, Moscow, Russia\\
L.D.Landau Institute for Theoretical Physics, Chernogolovka, Russia\\
Lomonosov Moscow State University, Faculty of Mechanics and Mathematics, Moscow, Russia}
\email{pgg@landau.ac.ru}

\thanks{This research has been partially supported by GNFM-INDAM and RFO University of Bologna.
The work of P.G. Grinevich was performed at the Steklov International Mathematical Center and supported by the Ministry of Science and Higher Education of the Russian Federation (agreement no. 075-15-2019-1614). It was also partially supported by the Russian Foundation for Basic Research, grant 20-01-00157.}

\begin{document}

\begin{abstract}
 The standard parametrization of totally non-negative Grassmannians was obtained by A. Postnikov \cite{Pos} introducing the boundary measurement map in terms of discrete path integration on planar bicoloured (plabic) graphs in the disc. An alternative parametrization was proposed by T. Lam \cite{Lam2} introducing systems of relations at the vertices of such graphs, depending on some signatures defined on their edges. The problem of characterizing the signatures corresponding to the totally non-negative cells, was left open in \cite{Lam2}. In our paper we provide an explicit construction of such signatures, satisfying both the full rank condition and the total non--negativity property on the full positroid cell. If the graph $\mathcal G$ satisfies the following natural constraint: each edge belongs to some oriented path from the boundary to the boundary, then such signature is unique up to a vertex gauge transformation. 

  Such signature is uniquely identified by geometric indices (local winding and intersection number) ruled by the orientation $\mathcal O$ and the gauge ray direction $\mathfrak l$ on $\mathcal G$. Moreover, we provide a combinatorial representation of the geometric signatures by showing that the total signature of every finite face just depends on the number of white vertices on it. The latter characterization is a Kasteleyn-type property \cite{AGPR,A3}, and we conjecture a mechanical-statistical interpretation of such relations. An explicit connection between the solution of Lam's system of relations and the value of Postnikov's boundary measurement map is established using the generalization of Talaska's formula \cite{Tal2} obtained in \cite{AG6}. In particular, the components of the edge vectors are rational in the edge weights with subtraction-free denominators.  

 Finally, we provide explicit formulas for the transformations of the signatures under Postnikov's moves and reductions, and  amalgamations of networks.
\end{abstract}

\maketitle

\section{Introduction}

Totally non--negative matrices and Grassmannians $\GTNN$  naturally arise in many different areas of mathematics. Totally non--negative matrices were first introduced in \cite{Sch}, and in \cite{GK,GK2} it was shown that they naturally arise in one-dimensional mechanical problems. The classical notion of total positivity was generalized by Lusztig  \cite{Lus1,Lus2} to reductive Lie groups, and a special case of this construction corresponds to totally non-negative Grassmanians. The literature dedicated to their applications is very wide, and it includes connections to the theory of cluster algebras of Fomin-Zelevinsky  \cite{FZ1,FZ2} in \cite{Sc,OPS}, the topological classification of real forms for isolated singularities of plane curves \cite{FPS}, the computation of scattering amplitudes in $N=4$ supersymmetric Yang--Mills theory \cite{AGP2,ADM}, the theory of Josephson's junctions \cite{BG}, and statistical mechanical models such as the asymmetric exclusion process \cite{CW} and dimer models in the disc \cite{Lam1}.

The restriction of the Gelfand-Serganova stratification of real Grassmannians to their totally non-negative part generates a cell decomposition of the totally non-negative Grassmannians \cite{PSW}. A rational parametrization of these cells was obtained in \cite{Pos} in terms of planar networks in the disc, where the parameters are positive edge weights, and the map from equivalence classes of weights to the points of a Grassmannian is Postnikov's boundary measurement map. The matrix elements of this map are expressed as sums over all weighted paths from a fixed boundary source to a fixed boundary sink.

An alternative approach was proposed by T. Lam \cite{Lam2}. In this approach the boundary measurement map is defined as the solution of a linear system of relations at the vertices of the graph. The total non-negativity property is encoded in the choice of the edge signatures. The existence of such a signature was proven in \cite{Lam2}, whereas the question of characterizing them explicitly was left open.

These relation spaces associated to graphs were originally proposed by T. Lam \cite{Lam2} to represent the amalgamation of totally non--negative Grassmannians, which is a special case of the amalgamation of cluster varieties originally introduced in \cite{FG1}. The latter have relevant applications in cluster algebras and relativistic quantum field theory \cite{AGP2,Kap,MS}. In particular, if the projected graphs represent positroid cells, the amalgamation of adjacent boundary vertices preserves the total non--negativity property and plays a relevant role also in real algebraic geometric problems such as polyhedral subdivisions \cite{Pos2}. 

Our interest in this subject was essentially motivated by applications of total non-negativity to the theory of the Kadomtsev--Petviashvili 2 (KP2) equation.  A connection between total non-negativity and the regularity of real KP 2 soliton solutions was first pointed out in \cite{Mal}. In fact, any such solution corresponds to a point of a totally non-negative Grassmannian. The structure and asymptotic behaviour of these solutions turned out to be very non-trivial and was described in terms of the combinatorial structure of the Grassmann cells (see \cite{BPPP,BK, CK,KW1,KW2} and references therein). In \cite{AG1,AG2,AG3} we proved that real regular soliton solutions can be obtained degenerating real regular finite-gap solutions, where the latter correspond to real divisors on $\mathtt M$--curves \cite{DN, Kr4}. In \cite{AG3,AG5} the double points of the degenerate spectral curve correspond to the edges of a plabic graph in the family associated by Postnikov \cite{Pos} to the positroid cell representing the given family of soliton solutions.  In particular, in \cite{AG5}, it was shown that the values of the KP wave functions at the double points solve Lam's relations \cite{Lam2} on the graph.

\smallskip

In this paper we provide an explicit solution to the problem posed by Lam: how to characterize all the admissible edge signatures. Admissible here means that, for every choice of positive edge weights 1) Lam's relations have full rank, 2) the solution at the boundary vertices generates a point in the totally non-negative Grassmannian. We provide an explicit formula for such a signature - we call it geometric -, and show that this signature is unique up to the vertex gauge transformation. The map from the positive weights to the totally non-negative Grassmannian for the geometric signature coincides with Postnikov's boundary measurement map \cite{Pos}. We also show that the solution of Lam's system at the internal edges coincides with the one obtained in \cite{AG6} using generalized Talaska flows. 

Moreover, we provide a completeness result. If a signature has the following properties: for any collection of positive weights:
\begin{enumerate}
\item Lam's system of relation has full rank;
\item The image at the boundary vertices is a point of the totally non-negative part of the Grassmannian,  
\end{enumerate}
then this signature is gauge equivalent to the geometric one.

In particular, we prove that the total signature of each face depends just on the number of white vertices bounding it. We claim that this property is equivalent to a variant of Kasteleyn theorem proven in \cite{Sp}. Indeed, this is true in the case of reduced bipartite graphs \cite{A3,AGPR}.

Finally, we describe the action of both Postnikov's moves and reductions and of the amalgamation of graphs on the signature.

\smallskip

We are convinced that our construction will turn out useful also for other applications connected to totally non--negative Grassmannians.
In \cite{GSV1} it is proven that the boundary measurement map possesses a natural Poisson-Lie structure, compatible with the natural cluster algebra structure on such Grassmannians. An interesting open question is how to use such Poisson--Lie structure in association with our geometric approach. Another open problem is to extend our construction to planar graphs on orientable surfaces with boundaries. We expect that a proper modification of Talaska formula \cite{Tal2} may again be used. Moreover, in such case it is necessary to modify the present construction using the procedure established in \cite{GSV2} to extend the boundary measurement map to planar networks in the annulus. An extension of the present construction to planar graphs in geometries different from the disc would also open the possibility of investigating the generalization of geometric relations in the framework of discrete integrable systems in cluster varieties \cite{KG, FG1}, amplitude scatterings on non-planar on-shell diagrams \cite{BFGW}, dimer models \cite{KO}, and possible relations to the Deodhar decomposition of the Grassmannian \cite{MR, TW}, which has already proven relevant for KP soliton theory \cite{KW1}.

\smallskip

\paragraph{\textbf{Main results and plan of the paper}}
In Section \ref{sec:plabic_graphs} we recall some useful properties of totally non--negative Grassmannians $\GTNN$, set up the class of the networks $\mathcal N$ used throughout the paper and recall the definition of Lam's relations for half-edge vectors. In the following $\S \subset Gr^{\mbox{\tiny TNN}}(k,n)$ is a positroid cell of dimension $|D|$, and $\mathcal G$ is a planar bicoloured directed trivalent perfect (plabic) graph in the disc representing $\S$ (Definition \ref{def:graph}). In our setting boundary vertices are all univalent, internal sources or sinks are not allowed, internal vertices may be either bivalent or trivalent and $\mathcal G$ may be either reducible or irreducible in Postnikov's sense \cite{Pos}. $\mathcal G$ has $g+1$ faces where $g=|D|$ if the graph is reduced, otherwise $g>|D|$. We also introduce a more restricted class of graphs (PBDTP graphs) with the following additional property: for each edge there exists at least one directed path from the boundary to the boundary passing through this edge.

In Section \ref{sec:lam_vectors} we summarize the results from \cite{AG6} necessary for us. In particular, we define the edge vectors $E_e$ as the formal sums of the signed contributions over all directed walks from the given edge $e$ to the boundary sinks. By definition, edge vectors satisfy a full rank system of linear relations at the vertices, and their components are explicit rational expressions in the edge weights with subtraction--free denominators, expressed in terms of the generalized Talaska's formula in \cite{AG6}.

In Section~\ref{sec:signatures} we introduce the geometric signatures, and show that the corresponding Lam's relations have full rank and that their solutions coincide with the edge vectors defined in Section~\ref{sec:lam_vectors}. In particular, if the vector space is $\R^n$,  the solution of Lam's relations at the boundary sources coincides with the value of Postnikov's boundary measurement map for any choice of positive weights. Moreover, we show that the combinatorial signature introduced in \cite{AG3} is geometric.

In Section~\ref{sec:vector_changes} we prove that the following transformations of the network: (1) changes of perfect orientations; (2) changes of gauge directions; (3) changes of positions of the vertices respecting the topology of the graph, are all gauge transformations; therefore  there is a unique equivalence class of geometric signatures on each given graph. The proofs are in the Appendix. 

In Section \ref{sec:comb} we prove that, for PBDTP networks, the only signatures such that Lam's relations have full rank and the image is a point of the totally non-negative part of the Grassmannian for every choice of positive weights, are geometric. Therefore we provide a geometric solution of Lam's problem of characterizing such signatures \cite{Lam2}.

In Section~\ref{sec:kasteleyn} we provide a topological characterization of the geometric signature: the total signature of each face just depends on the number of white vertices bounding it. This result is a Kasteleyn-type theorem since in the case of bipartite graphs this value coincides with that of the Kasteleyn's signature (see \cite{AGPR,A3}). Using this topological characterization, we compute the effect of Postnikov's moves and reductions, and of the amalgamation of graphs on the geometric signature.

\section{Plabic networks and totally non--negative Grassmannians}\label{sec:plabic_graphs}

In this Section we recall some basic properties of totally non--negative Grassmannians, and define the class of graphs $\mathcal G$ representing a given positroid cell which we use throughout the text. 
We use the following notations throughout the paper:
\begin{enumerate}
\item $k$ and $n$ are positive integers such that $k<n$;
\item  For $s\in {\mathbb N}$  $[s] =\{ 1,2,\dots, s\}$; if $s,j \in {\mathbb N}$, $s<j$, then
$[s,j] =\{ s, s+1, s+2,\dots, j-1,j\}$;
\end{enumerate}

\begin{definition}\textbf{Totally non-negative Grassmannian \cite{Pos}.}
Let $Mat^{\mbox{\tiny TNN}}_{k,n}$ denote the set of real $k\times n$ matrices of maximal rank $k$ with non--negative maximal minors $\Delta_I (A)$. Let $GL_k^+$ be the group of $k\times k$ matrices with positive determinants. Then the totally non-negative Grassmannian for these data is 
\[
\GTNN = GL_k^+ \backslash Mat^{\mbox{\tiny TNN}}_{k,n}.
\]
\end{definition}
In the theory of totally non-negative Grassmannians an important role is played by the positroid stratification. Each cell in this stratification is defined as the intersection of a Gelfand-Serganova stratum \cite{GS,GGMS} with the totally non-negative part of the Grassmannian. More precisely:
\begin{definition}\textbf{Positroid stratification \cite{Pos}.} Let $\mathcal M$ be a matroid i.e. a collection of $k$-element ordered subsets $I$ in $[n]$, satisfying the exchange axiom (see, for example \cite{GS,GGMS}). Then the positroid cell $\S$ is defined as
$$
\S=\{[A]\in \GTNN\ | \ \Delta_{I}(A) >0 \ \mbox{if}\ I\in{\mathcal M} \ \mbox{and} \  \Delta_{I}(A) = 0 \ \mbox{if} \ I\not\in{\mathcal M}  \}.
$$
A positroid cell is irreducible if, for any $j\in [n]$, there exist $I, J\in \mathcal M$ such that $j\in I$ and $j\not\in J$.
\end{definition}
The combinatorial classification of the non-empty positroid cells and their rational parametrizations were obtained in \cite{Pos}, \cite{Tal2}. In our construction we use the classification of positroid cells via directed planar networks in the disc in \cite{Pos}. More precisely, we use the following class of graphs ${\mathcal G}$ introduced by Postnikov \cite{Pos}:
\begin{definition}\label{def:graph} \textbf{Planar bicoloured directed trivalent perfect graphs in the disc (plabic graphs).} A graph ${\mathcal G}$ is called plabic if:
\begin{enumerate}
\item  ${\mathcal G}$ is planar, directed and lies inside a disc. Moreover ${\mathcal G}$ is connected in the sense it does not possess components isolated from the boundary;
\item It has finitely many vertices and edges;
\item It has $n$ boundary vertices on the boundary of the disc labeled $b_1,\cdots,b_n$ clockwise. Each boundary vertex has degree 1. We call a boundary vertex $b_i$ a source (respectively sink) if its edge is outgoing (respectively incoming);
\item The remaining vertices are called internal and are located strictly inside the disc. They are either bivalent or trivalent;
\item ${\mathcal G}$ is a perfect graph, that is each internal vertex in  ${\mathcal G}$ is incident to exactly one incoming edge or to one outgoing edge. In the first case the vertex is coloured white, in the second case black. Bivalent vertices are assigned either white or black colour.
\end{enumerate}
Moreover, to simplify the overall construction we further assume that the boundary vertices $b_j$, $j\in [n]$, lie on a common interval of the boundary of the disc.
\end{definition}

\begin{remark}
\begin{enumerate}
\item The trivalency assumption is not restrictive, since any perfect plabic graph can be transformed into a trivalent
  one;
\item The assumption that the boundary vertices $b_j$, $j\in [n]$ lie on a common interval of the boundary of the disc considerably simplifies the use of the gauge ray directions for defining the geometric signature in terms of the local winding and intersections' numbers (see \cite{AG6} and Section~\ref{sec:def_edge_vectors}). Indeed, one may assume that the graph lies inside the upper half-plane, all edges are straight intervals, and the infinite face contains the infinite point. 
\end{enumerate}
\end{remark}

\begin{definition}\textbf{Acyclic orientation}\label{def:acyclic}
A plabic graph ${\mathcal G}$ is called \textbf{acyclically oriented} if it does not have closed directed paths \cite{Pos}.
\end{definition}
\begin{figure}
\centering{\includegraphics[width=0.4\textwidth]{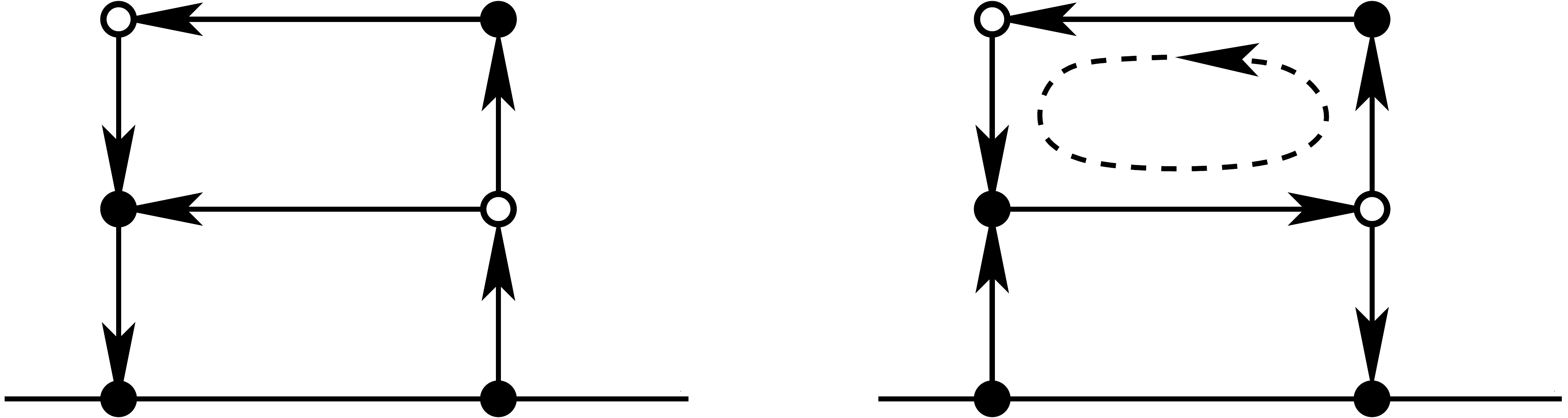}}
\caption{\small{\sl The graph on the left is acyclic, whereas the one on the right has a non-trivial cycle.}\label{fig:acycl}}
\end{figure}

\begin{definition}\textbf{Terminology for faces}\label{def:faces}
A face $\Omega$ is internal if its boundary has empty intersection with the boundary of the disc, otherwise it is a boundary face. There is a unique boundary face including the boundary segment from $b_n$ to $b_1$ clockwise and we call it the infinite face. All other faces are finite faces.
\end{definition}

The class of perfect orientations of the plabic graph ${\mathcal G}$ are those which are compatible with the colouring of the vertices. The graph is of type $(k,n)$ if it has $n$ boundary vertices and $k$ of them are boundary sources. Any choice of perfect orientation preserves the type of ${\mathcal G}$. To any perfect orientation $\mathcal O$ of ${\mathcal G}$ one assigns the base $I_{\mathcal O}\subset [n]$ of the $k$-element source set for $\mathcal O$. Following \cite{Pos} the matroid of ${\mathcal G}$ is the set of $k$-subsets  $I_{\mathcal O}$ for all perfect orientations:
$$
\mathcal M_{\mathcal G}:=\{I_{\mathcal O}|{\mathcal O}\ \mbox{is a perfect orientation of}\ \mathcal G \}.
$$
In \cite{Pos} it is proven that $\mathcal M_{\mathcal G}$ is a totally non-negative matroid $\mathcal S^{\mbox{\tiny TNN}}_{\mathcal M_{\mathcal G}}\subset \GTNN$. 
In Figure \ref{fig:Rules0} we present an example of a plabic graph satisfying Definition~\ref{def:graph}, and representing a 10-dimensional positroid cell in $Gr^{\mbox{\tiny TNN}}(4,9)$.

\begin{figure}
  \centering{\includegraphics[width=0.5\textwidth]{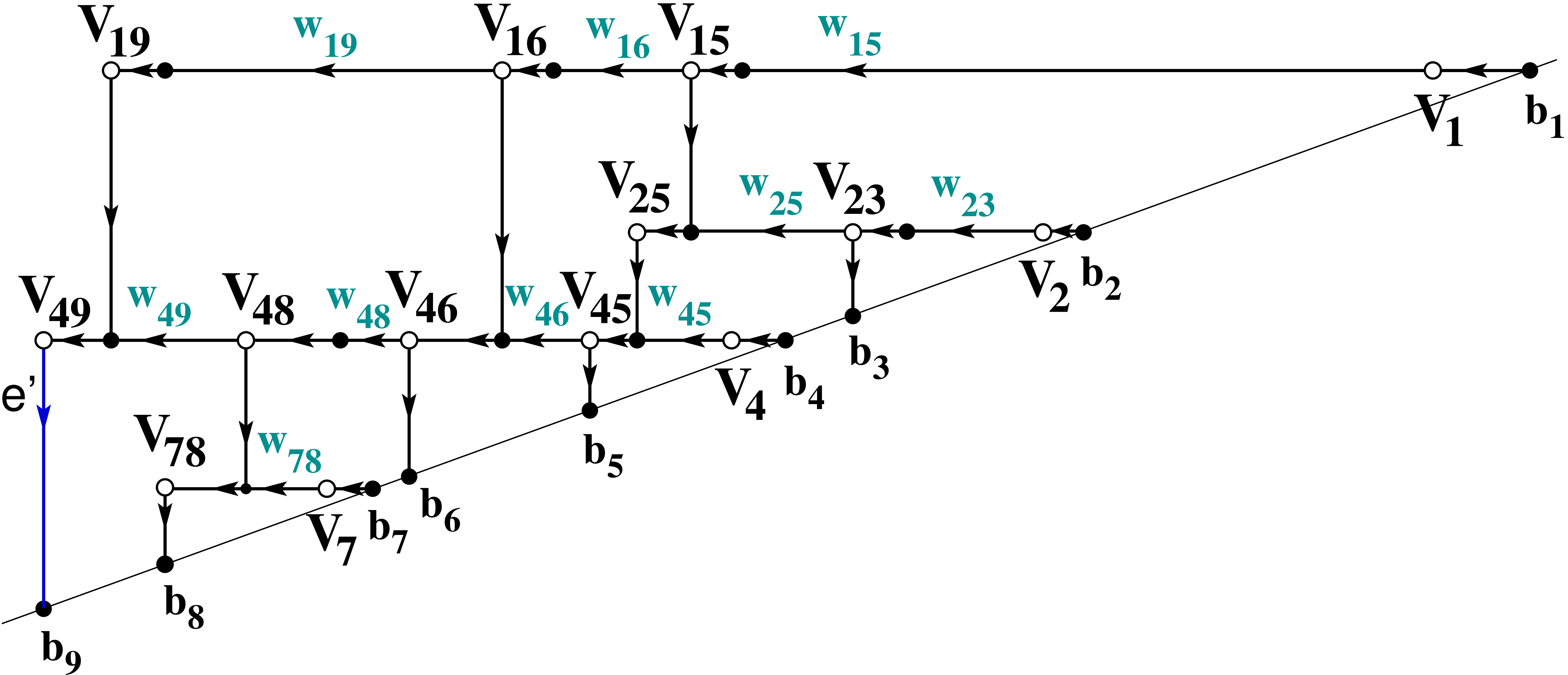}
      \caption{\small{\sl The plabic network represents a point in a 10--dimensional positroid cell in $Gr^{\mbox{\tiny TNN}} (4,9)$.
					}}\label{fig:Rules0}}
\end{figure}

\begin{theorem}\cite{Pos}
A plabic graph $\mathcal G$ can be transformed into a plabic graph $\mathcal G'$ via a finite sequence of Postnikov's moves and reductions if and only if $\mathcal M_{\mathcal G}=\mathcal M_{\mathcal G'}$.
\end{theorem}

A graph  $\mathcal G$ is reduced if there is no other graph in its move reduction equivalence class which can be obtained from $\mathcal G$ applying a sequence of transformations containing at least one reduction \cite{Pos}. Each positroid cell $\S$ is represented by at least one reduced graph, the so called Le--graph (see \cite{Pos} and Definition \ref{def:can_Le}), associated to the Le--diagram representing $\S$. The plabic graph in Figure \ref{fig:Rules0} is a Le-graph.

If $\mathcal G$ is a reduced  plabic graph, then the dimension of  $\mathcal S^{\mbox{\tiny TNN}}_{\mathcal M_{\mathcal G}}$ is equal to the number of faces of $\mathcal G$ minus 1.

\begin{definition}\label{def:plabic_network} \textbf{Plabic network}
  A plabic network  $\mathcal N$ is a plabic graph $\mathcal G$ together with a collection of positive weights assigned to its edges.  
\end{definition}

\begin{definition}\label{def:weight_gauge} \textbf{Weight gauge equivalence}
  A pair of plabic networks  $\mathcal N_1$,  $\mathcal N_2$ are called weight gauge equivalent if
\begin{enumerate}
\item They share the same graph $\mathcal G$;  
\item There exists a positive function  $t(U)$ on the vertices of $\mathcal G$ such that $t(U)=1$ at all boundary vertices and  at each directed edge
$e=(U,V)\in\mathcal G$:
\begin{equation}
\label{eq:gauge}
w^{(2)}_e =  w^{(1)}_e t_U \left(t_V\right)^{-1}.
\end{equation}
Here $w^{(1)}_e$ and  $w^{(2)}_e$ are the weights assigned to the edge $e$ in networks  $\mathcal N_1$,  $\mathcal N_2$ respectively.  
\end{enumerate}
\end{definition}

Plabic networks provide a global parametrization of the positroid cells $\S$ \cite{Pos}.

\begin{definition}\label{def:boundary_map} \textbf{Boundary measurement map}
In \cite{Pos}, for any given oriented planar network in the disc it is defined the formal boundary measurement map 
\begin{equation}\label{eq:bmm}
M_{ij} := \sum\limits_{P \, : \, b_i\mapsto b_j} (-1)^{\mbox{\scriptsize
 Wind}(P)} w(P),
\end{equation}
where the sum is over all directed walks from the source $b_i$ to the sink $b_j$, $w(P)$ is the product of the edge weights of $P$ and $\mbox{Wind}(P)$ is its topological winding index.
\end{definition}
The image of the boundary measurement map is the same for all gauge weight equivalent networks.

These formal power series sum up to subtraction--free rational expressions in the weights \cite{Pos} and explicit expressions of the $M_{ij}$ in function of flows and conservative flows in the network are obtained in \cite{Tal2}.

\begin{definition}\label{def:boundary_matrix} \textbf{Boundary measurement matrix \cite{Pos}}
Let $I$ be the base inducing the orientation of $\mathcal N$ used in the computation of the boundary measurement map. Then the point $Meas(\mathcal N)\in Gr(k,n)$ is represented by the boundary measurement matrix $A$ such that:
\begin{itemize}
\item The submatrix $A_I$ in the column set $I$ is the identity matrix;
\item The remaining entries $A^r_j = (-1)^{\sigma(i_r,j)} M_{ij}$, $r\in [k]$, $j\in \bar I$, where $\sigma(i_r,j)$ is the number of elements of $I$ strictly between $i_r$ and $j$.
\end{itemize}
\end{definition}

\begin{remark}\label{rem:reciprocal}
If one changes the perfect orientation on the network and simultaneously changes the weights on the edges with the following rule:
\begin{enumerate}
\item If the edge $e=(U,V)$ does not change the orientation, the weight $w_e$ remains unchanged; 
\item If the edge $e=(U,V)$ changes the orientation, the weight is replaced by its reciprocal,   
\end{enumerate}
then the new  boundary measurement matrix $\tilde A$ represents the same point of the Grassmannian $Gr(k,n)$.
\end{remark}

In the following we also consider a more restrictive class of plabic graphs.
\begin{definition}\label{def:PBDTP} \textbf{PBDTP graph}
A plabic graph $\mathcal G$ is called PBDTP graph if it satisfies the following additional condition: for any edge of $\mathcal G$  there exists a directed path from the boundary to the boundary containing it.
\end{definition}
\begin{remark}
This additional requirement means that the weights at all edges contribute to the boundary measurement map.   
\end{remark}

We have the following elementary Lemma.
\begin{lemma}
  \begin{enumerate}
  \item A PBDTP graph always represents an irreducible positroid  cell;
  \item In a  PBDTP graph all internal faces contain vertices of both colours;  
  \item If a plabic graph is PBDTP in one orientation, it is PBDTP in all other perfect orientations. 
\end{enumerate}
\end{lemma}

\subsection{Lam's relations for half-edge vectors}\label{sec:lam_relations}
The boundary measurement map, defined in the previous Section, admits an alternative representation in terms of linear relations in the half-edge vectors, proposed by Lam \cite{Lam2}. Let us recall this construction. 

Let $\mathcal V$ be a vector space and $\mathcal N$ be a plabic network. Assign a signature $\epsilon_{U,V}\in\{0,1\}$ to any edge $e=(U,V)$. The signature is independent on the orientation of the network: $\epsilon_{U,V}=\epsilon_{V,U}$. 

Then a collection of half-edge vectors $z_{U,e}\in\mathcal V$ enumerated by pairs $(U,e)$, where $e$ is an edge of $\mathcal N$ and $U$ is one of the ends of $e$, is a solution of Lam's system of relations if it fulfils the conditions of the following Definition:  

\begin{definition}\textbf{Lam's system of relations}\label{def:lam} \cite{Lam2}
\begin{enumerate}
\item For any edge $e=(U,V)$, $z_{U,e} = (-1)^{\epsilon_{U,V}}w_{U,V} z_{V,e}$, where $U$ is the starting vertex of $e$ and $V$ is the final one, and $w_{U,V}$ is the weight in this orientation;
\item If $e_i$, $i\in [m]$, are the edges at an $m$-valent white vertex $V$, then $\sum_{i=1}^m z_{V,e_i} =0$;
\item If $e_i$, $i\in [m]$, are the edges at an $m$-valent black vertex $V$, then $z_{V,e_i} =z_{V,e_j}$ for all $i,j\in[m]$.
\end{enumerate}
\end{definition}

\begin{figure}
  \centering
	{\includegraphics[width=0.8\textwidth]{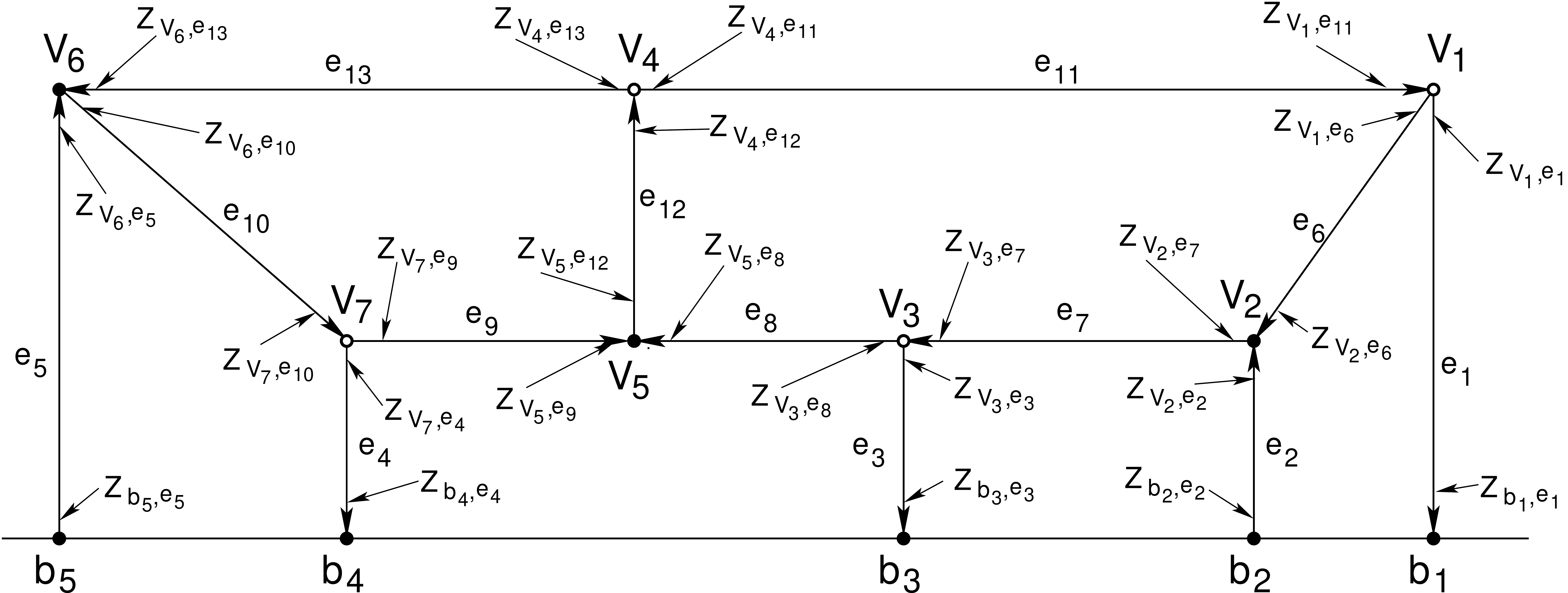}}
  \caption{\small{\sl Example of Lam's relations. }\label{fig:lam_rel}}
\end{figure}

  \begin{example}\label{ex:lam_rel} Consider the graph in Figure~\ref{fig:lam_rel}. To each edge we associate a positive weight, a sign and a pair of vectors. Lam's relations for this graph are the following:
\begin{align}\label{eq:lam_example}
  & Z_{V_1,e_1}+ Z_{V_1,e_6}+ Z_{V_1,e_{11}}= 0, & & Z_{V_3,e_3}+ Z_{V_3,e_7}+ Z_{V_3,e_8}= 0, \nonumber \\
  & Z_{V_4,e_{11}}+ Z_{V_4,e_{12}}+ Z_{V_4,e_{13}}= 0, & & Z_{V_7,e_4}+ Z_{V_7,e_9}+ Z_{V_7,e_{10}}= 0, \nonumber \\
  & Z_{V_2,e_2}= Z_{V_2,e_6}= Z_{V_2,e_7},  & & Z_{V_5,e_8}= Z_{V_5,e_9}= Z_{V_5,e_{12}}, \nonumber \\
  & Z_{V_6,e_5}= Z_{V_6,e_{10}}= Z_{V_6,e_{13}},  & &  \nonumber \\
   &  Z_{V_1,e_6}= (-1)^{s(e_6)} w(V_1,V_2) Z_{V_2,e_6}, & &  Z_{V_2,e_7}= (-1)^{s(e_7)} w(V_2,V_3) Z_{V_3,e_7},  \nonumber \\
  &  Z_{V_3,e_8}= (-1)^{s(e_8)} w(V_3,V_5) Z_{V_5,e_8},  & & Z_{V_5,e_9}= (-1)^{s(e_9)} w(V_5,V_7) Z_{V_7,e_9},    \nonumber \\
   &  Z_{V_7,e_{10}}= (-1)^{s(e_{10})} w(V_7,V_6) Z_{V_6,e_{10}}, & &  Z_{V_1,e_{11}}= (-1)^{s(e_{11})} w(V_1,V_4) Z_{V_4,e_{11}},  \nonumber \\
  &  Z_{V_4,e_{12}}= (-1)^{s(e_{12})} w(V_4,V_5) Z_{V_5,e_{12}},  & & Z_{V_4,e_{13}}= (-1)^{s(e_{13})} w(V_4,V_6) Z_{V_6,e_{13}}.    \nonumber  \\
	&  Z_{b_1,e_1}= (-1)^{s(e_1)} w(b_1,V_1) Z_{V_1,e_1}, & & Z_{b_2,e_2}= (-1)^{s(e_2)} w(b_2,V_2) Z_{V_2,e_2},  \nonumber \\
  &  Z_{b_3,e_3}= (-1)^{s(e_3)} w(b_3,V_3) Z_{V_3,e_3}, & & Z_{b_4,e_4}= (-1)^{s(e_4)} w(b_4,V_4) Z_{V_4,e_4}, \\
  &  Z_{b_5,e_5}= (-1)^{s(e_5)} w(b_5,V_5) Z_{V_5,e_5}, & &   \nonumber \\																												
\end{align}  
We have 23 equations and 26 half-edge vectors. If this system has maximal rank, we can solve it for all of the half-edge vectors at the internal vertices, and we end up with 2 equations in just the 5 half-edge vectors at the boundary $Z_{b_1,e_1}$,  $Z_{b_2,e_2}$, $Z_{b_3,e_3}$,  $Z_{b_4,e_4}$, $Z_{b_5,e_5}$.  
\end{example}

\begin{remark}
If all internal vertices are trivalent, then the linear system can be interpreted as the amalgamation of the little positive Grassmannians $Gr^{\mbox{\tiny TP}}(1,3)$, $Gr^{\mbox{\tiny TP}}(2,3)$ \cite{Lam2}.
\end{remark}

In \cite{Lam2} it is conjectured that there exist simple rules for the assignment of signatures to the edges so that the system in Definition \ref{def:lam} has full rank for any choice of positive weights, and the image of this weighted space of relations is the positroid cell $\S\subset\GTNN$ corresponding to the graph.

Of course, if such signature exists, it is not unique, because of the following gauge freedom:

\begin{definition}\textbf{Equivalence between edge signatures}.\label{def:admiss_sign_gauge}
Let $\epsilon^{(1)}_{U,V}$ and $\epsilon^{(2)}_{U,V}$ be two signatures on all the edges $e=(U,V)$ of the plabic graph $\mathcal G$, edges at the boundary included. We say that the two signatures are equivalent if there exists an index $\eta(U) \in \{ 0,1\}$ at each internal vertex $U$ such that
\begin{equation}\label{eq:equiv_sign}
\epsilon^{(2)}_{U,V} = \left\{\begin{array}{ll}
\epsilon^{(1)}_{U,V} +\eta(U)+\eta(V), \mod(2) & \mbox{ if } e=(U,V) \mbox{ is an internal edge},\\
\epsilon^{(1)}_{U,V} +\eta(U), \mod(2) & \mbox{ if } e=(U,V) \mbox{ is the edge at some boundary vertex } V.
\end{array}\right.
\end{equation} 
\end{definition}

This gauge freedom corresponds to the following transformation of the system of half-edge vectors at the internal vertices $U$:
\begin{equation}\label{eq:equiv_sign_vectors}
z^{(2)}_{U,e} =  (-1)^{\eta(U)}\ z^{(1)}_{U,e}.
\end{equation}
The gauge transformations (\ref{eq:equiv_sign_vectors}) and the weight gauge transformations (\ref{eq:gauge}) leave invariant the half-edge vectors at the boundary vertices, as well as the relations at the internal white and black vertices. Therefore:
\begin{proposition}
  If the system of Lam's relations has full rank for a given collection of weights and  a given signature $\epsilon^{(1)}_{U,V}$, then it has full rank for any signature $\epsilon^{(2)}_{U,V}$ equivalent to $\epsilon^{(1)}_{U,V}$ and  for any weight gauge equivalent collection of weights. Moreover, in such case, the solutions of all such systems of relations coincide at the boundary vertices.

Finally, the system of Lam's relations remains invariant if one changes the perfect orientation of the network $\mathcal N$ and simultaneously replaces the weights as in Remark~\ref{rem:reciprocal}.
\end{proposition}

\section{Geometric systems of relations on plabic networks}\label{sec:lam_vectors}

In \cite{AG6} we constructed systems of edge vectors on plabic networks $\mathcal N$ by extending Postnikov's construction in \cite{Pos} to the internal vertices. These systems of edge vectors satisfy full rank linear systems of relations whose solutions are expressed using a generalization of Talaska flows. In this Section we recall such construction and then, in the following Section, we prove that the associated geometric signature solves Lam's problem.

\subsection{Systems of edge vectors on plabic networks, edge flows and conservative flows}\label{sec:def_edge_vectors}

In this Section, we recall the construction of the edge vectors in \cite{AG6}, and the definitions of the edge flows and of the conservative flows.

To obtain a well-defined system of edge vectors it is necessary to fix a gauge by introducing a gauge ray direction. Using this direction we define  two local indices -- the \textbf{local winding number} and the \textbf{local intersection number}. These indices extend Postnikov's \textbf{topological winding number} of a path and \textbf{the number of sources} between its starting and ending boundary vertices, to paths starting at an internal edge of the graph and ending at a boundary sink. 

Let us remark that earlier in \cite{GSV}, gauge directions were introduced to compute the winding number of a path joining boundary vertices.

\begin{definition}\label{def:gauge_ray}\textbf{The gauge ray direction $\mathfrak{l}$.}
A gauge ray direction is an oriented direction ${\mathfrak l}$ with the following properties:
\begin{enumerate}
\item The ray ${\mathfrak l}$ starting at a boundary vertex points inside the disc; 
\item No edge is parallel to this direction;
\item All rays starting at boundary vertices do not cross internal vertices of the network.
\end{enumerate}
\end{definition}
We remark that the first property may be always satisfied since we assume that all boundary vertices lie on a common straight interval of the boundary of $\mathcal N$. We then define the local winding number between a pair of consecutive edges $e_k,e_{k+1}$ as follows.

\begin{figure}
  \centering
  {\includegraphics[width=0.4\textwidth]{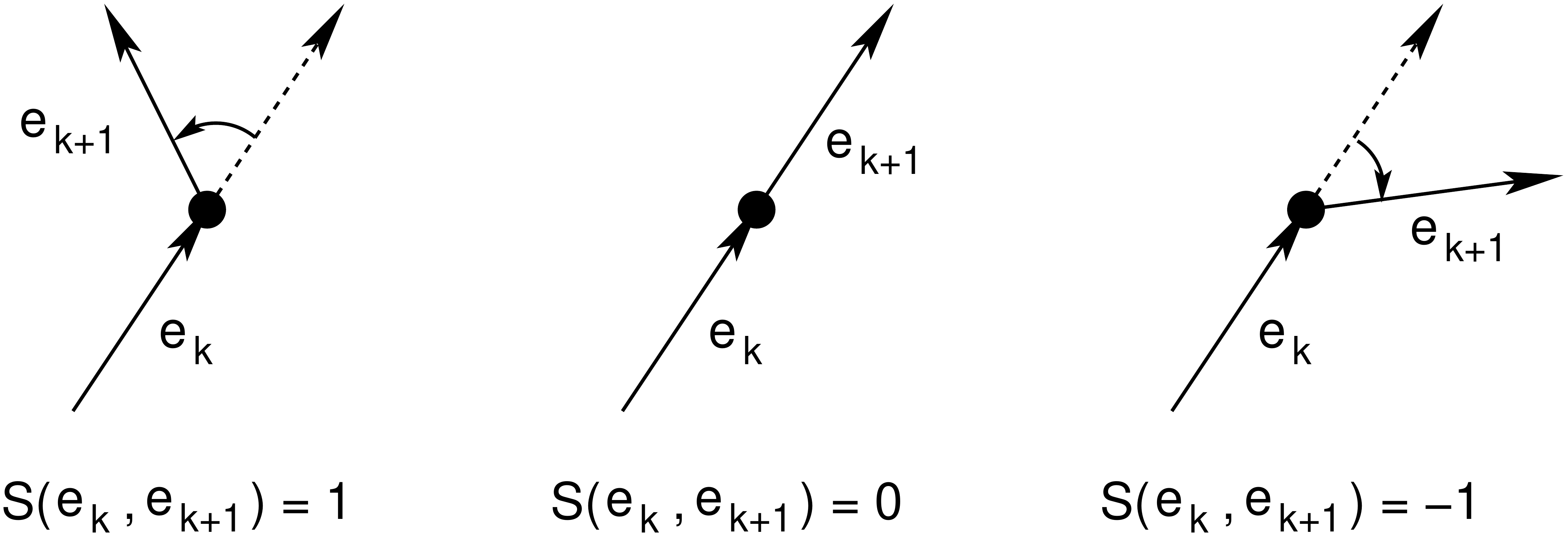}
	\caption{\footnotesize{\sl The index $s(e_k,e_{k+1})$.}}
  \label{fig:orient}}
\end{figure}
\begin{definition}\label{def:winding_pair}\textbf{The local winding number at an ordered pair of oriented edges}
Let  $(e_k,e_{k+1})$ be an ordered pair of oriented edges. Let us define
\begin{equation}\label{eq:def_s}
s(e_k,e_{k+1}) = \left\{
\begin{array}{ll}
+1 & \mbox{ if the ordered pair is positively oriented }  \\
0  & \mbox{ if } e_k \mbox{ and } e_{k+1} \mbox{ are parallel }\\
-1 & \mbox{ if the ordered pair is negatively oriented }
\end{array}
\right.
\end{equation}
Then the winding number of the ordered pair $(e_k,e_{k+1})$ with respect to the gauge ray direction $\mathfrak{l}$ is
\begin{equation}\label{eq:def_wind}
\mbox{wind}(e_k,e_{k+1}) = \left\{
\begin{array}{ll}
+1 & \mbox{ if } s(e_k,e_{k+1}) = s(e_k,\mathfrak{l}) = s(\mathfrak{l},e_{k+1}) = 1\\
-1 & \mbox{ if } s(e_k,e_{k+1}) = s(e_k,\mathfrak{l}) = s(\mathfrak{l},e_{k+1}) = -1\\
0  & \mbox{otherwise}.
\end{array}
\right.
\end{equation}
\end{definition}

Let $b_{i_r}$, $r\in[k]$, $b_{j_l}$, $l\in[n-k]$, respectively  be the set of the boundary sources and of the boundary sinks for the given orientation. Then draw the rays ${\mathfrak l}_{i_r}$ starting at $b_{i_r}$, $r\in[k]$ (see Figure \ref{fig:indexes} for an example).

\begin{figure}
  \centering
  {\includegraphics[width=0.55\textwidth]{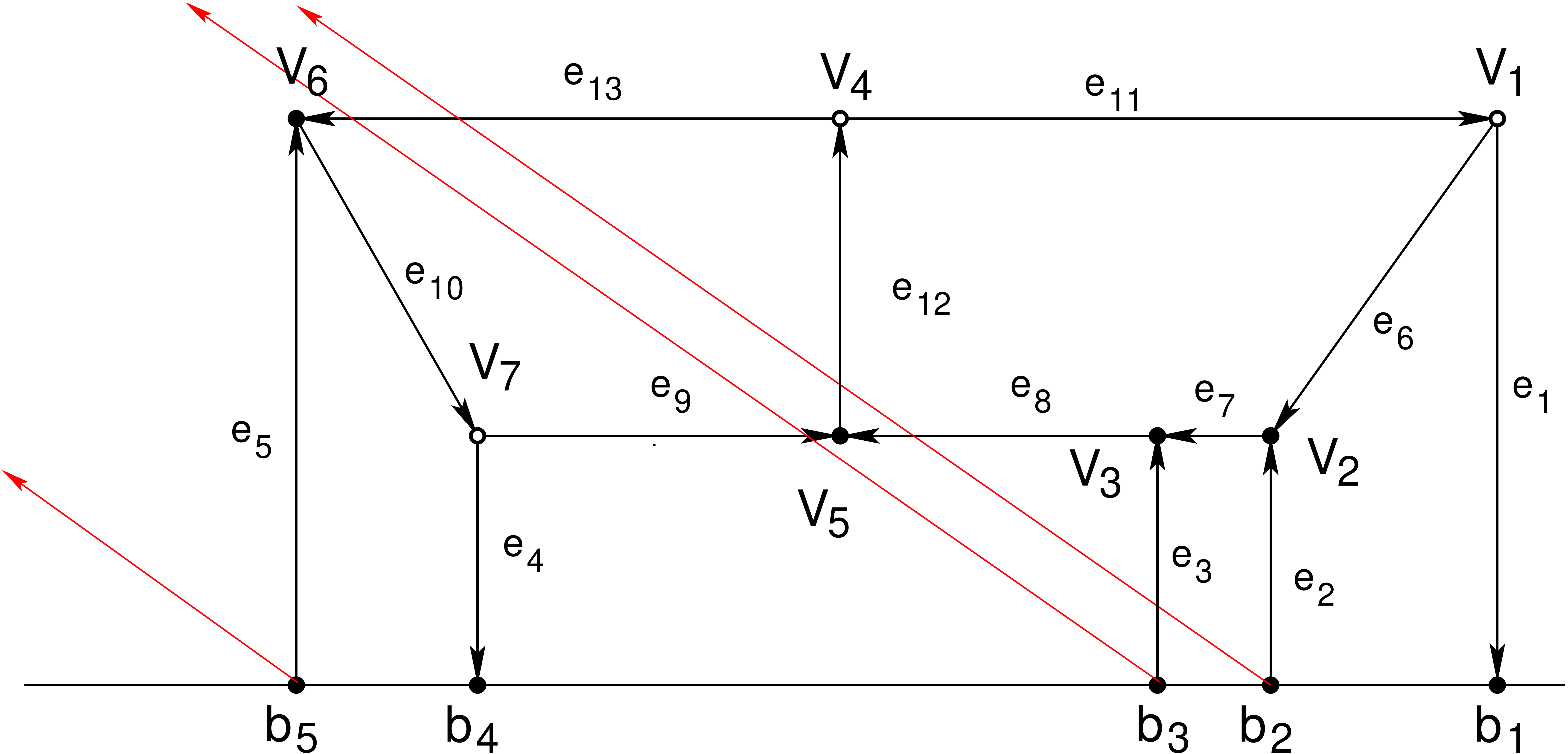}
	\caption{\footnotesize{\sl The gauge ray direction fixes the winding number of ordered pairs of edges, whereas the gauge rays starting at the boundary sources fix the intersection number of each edge.}}
        \label{fig:indexes}}
      \end{figure}
    
\begin{definition}\label{def:local_int}\textbf{The intersection number of an edge} The intersection number of the edge $e$ is the number of intersections between $e$ and the rays ${\mathfrak l}_{i_r}$, $r\in[k]$ times $s({\mathfrak l},e)$, and it is denoted by $\mbox{int}(e)$.
\end{definition}

\begin{example} Let us compute the winding and the intersection numbers for the example in Figure~\ref{fig:indexes}:
  \begin{align} \label{eq:indexes}
    & \mbox{wind}(e_2,e_7) =1, & &  \mbox{wind}(e_8,e_{12}) =-1, & &  \mbox{wind}(e_6,e_7) =0, & & \mbox{wind}(e_{10},e_9) =0,
    & & \mbox{wind}(e_{13},e_{10}) =0, \\
    & \mbox{int}(e_2) = 0, & & \mbox{int}(e_3) = -1, & & \mbox{int}(e_8) = 1, & & \mbox{int}(e_9) = -1,
    & & \mbox{int}(e_{12}) = -1, & & \mbox{int}(e_{13}) = 2. \nonumber
\end{align}  

\end{example}

Next in \cite{AG6} we define the $j$-th component of the edge vector at $e$ as the following signed summation over all directed paths starting at the edge $e$ and ending at the boundary sink $b_j$.

\begin{definition}\label{def:edge_vector}\textbf{The edge vector $E_e$ }\cite{AG6} Assume that we have a set of vectors $E_j$ assigned to the boundary sinks $b_j$ (we call these vectors \textbf{boundary conditions}). For any edge $e$, let us consider all possible directed paths ${\mathcal P}:e\rightarrow b_{j}$ in $({\mathcal N},{\mathcal O},{\mathfrak l})$ such that the first edge is $e$ and the end point is a boundary sink $b_{j}$, $j\in\bar I$, where $I$ is the base associated with the orientation ${\mathcal O}$, i.e. the subset of indices enumerating the boundary sources.
Then we consider the following formal sum:
\begin{equation}\label{eq:sum}
E_{e} = \sum\limits_{j\in \bar I}\ \sum\limits_{{\mathcal P}\, :\, e\rightarrow b_{j}} (-1)^{\mbox{wind}({\mathcal P})+ \mbox{int}({\mathcal P})} 
w({\mathcal P}) E_j,
\end{equation}
where
\begin{enumerate}
\item The \textbf{weight $w({\mathcal P})$} is the product of the weights $w_l$ of all edges $e_l$ in ${\mathcal P}$: $w({\mathcal P})=\prod_{l=1}^m w_l$. If we pass the 
same edge $e$ of weight $w_e$ $r$ times, the weight is counted as $w_e^r$;
\item The \textbf{generalized winding number} $\mbox{wind}({\mathcal P})$ is the sum of the local winding numbers at each ordered pair of its edges:
$\mbox{wind}(\mathcal P) = \sum_{k=1}^{m-1} \mbox{wind}(e_k,e_{k+1}),$ 
with $\mbox{wind}(e_k,e_{k+1})$ as in Definition \ref{def:winding_pair}; 
\item The \textbf{intersection number} $\mbox{int}(\mathcal P) = \sum\limits_{s=1}^m \mbox{int}(e_s)$, where $\mbox{int}(e_s)$ is as in Definition~\ref{def:local_int}.
\end{enumerate}

By definition, at the edge $e$ at the boundary sink $b_j$, the edge vector $E_{e}$ is
\begin{equation}\label{eq:vec_bou_sink}
E_{e} =  (-1)^{\mbox{int}(e)} w(e) E_j.
\end{equation}
\end{definition}

It is easy to check that the formal sum in (\ref{eq:sum}) converges if the weights are sufficiently small. In \cite{AG6} (see Section \ref{sec:linear}), we adapt Theorem~3.2 in \cite{Tal2} to our purposes: the sum over all paths $e\rightarrow b_j$ for a fixed pair $e,b_j$ in (\ref{eq:sum}) is a rational expression in the weights with subtraction-free denominator, and it may be explicitly computed in terms of flows.
Let us recall the necessary definitions. The following is a straightforward adaptation of Fomin's notion of loop--erased walk to our purposes (see \cite{Fom}).

\begin{definition}
\label{def:loop-erased-walk}
\textbf{Edge loop-erased walks.}  Let ${\mathcal P}$ be a walk (directed path) given by
$$
V_e \stackrel{e}{\rightarrow} V_1 \stackrel{e_1}{\rightarrow} V_2 \rightarrow \ldots \rightarrow b_j,
$$
where $V_e$ is the initial vertex of the edge $e$. The edge loop-erased part of  ${\mathcal P}$, denoted $LE({\mathcal P})$, is defined recursively as 
follows. If ${\mathcal P}$ does not pass any edge twice (i.e. all edges $e_i$ are distinct), then $LE({\mathcal P})={\mathcal P}$.
Otherwise, set $LE({\mathcal P})=LE({\mathcal P}_0)$, where ${\mathcal P}_0$ is obtained from ${\mathcal P}$ removing the first
edge loop it makes; more precisely, given all pairs $l,s$ with $s>l$ and $e_l = e_s$, one chooses the one with the smallest value of $s$ and removes the cycle
$$
V_l \stackrel{e_l}{\rightarrow} V_{l+1} \stackrel{e_{l+1}}\rightarrow V_{l+2} \rightarrow \ldots \stackrel{e_{s-1}}\rightarrow V_{s} ,
$$
from ${\mathcal P}$.
\end{definition}

With this procedure, to each path starting at $e$ and ending at $b_j$ one associates a unique edge loop-erased walk $LE({\mathcal P})$, where the latter path is either acyclic or possesses one simple cycle passing through its initial vertex, since each vertex contains either one incoming or one outgoing edge.

Next, in \cite{AG6} the definitions of flows and conservative flows in \cite{Tal2} are adapted to the present case.

\begin{definition}\label{def:cons_flow}\textbf{Conservative flow \cite{Tal2}}. A collection $C$ of distinct edges in a plabic graph $\mathcal G$ is called a conservative flow if, for each vertex $V$ in $\mathcal G$, the number of edges of $C$ that arrive at $V$ is equal to the number of edges of $C$ that leave from $V$. In particular $C$ does not contain edges incident to the boundary.

The set of all conservative flows $C$ in $\mathcal G$ is denoted by ${\mathcal C}(\mathcal G)$. 

The \textbf{weight $w(C)$} of the conservative flow $C$ is the product of the weights of all edges in $C$. In particular, ${\mathcal C}(\mathcal G)$ contains the trivial flow with no edges, which has weight 1.
\end{definition}

\begin{definition}\label{def:edge_flow}\textbf{Edge flow at $e$}. \cite{AG6} A collection $F_e$ of distinct edges in a plabic graph $\mathcal G$ is called an edge flow starting at the edge $e$ if: 
\begin{enumerate}
\item It contains the edge $e$;
\item If $V$ is the starting vertex of $e$, the number of edges of $F_e$ that arrive at $V$ is equal to the number of edges of $F_e$ that leave from $V$ minus 1;
\item For each other interior vertex $V_d$ in $\mathcal G$, the number of edges of $F_e$ that arrive at $V_d$ is equal to the number of edges of $F_e$ that leave from $V_d$;
\item The intersection of $F_e$ with the set of the edges at the boundary sources is either empty or contains just $e$ itself.  
\end{enumerate}
We  denote by   ${\mathcal F}_{ej}(\mathcal G)$ the collection of the edge flows starting at $e$ and containing the boundary sink $b_j$.

An edge flow $F_{e,b_j}\in{\mathcal F}_{e,b_j}(\mathcal G)$ is the union of an edge loop-erased walk $P_{e,b_j}$ with a conservative flow with no common edges with $P_{e,b_j}$. Three numbers are assigned to $F_{e,b_j}$:
\begin{enumerate} 
\item The \textbf{weight of the flow $w(F_e)$} is the product of the weights of all edges in $F_e$.
\item The \textbf{winding number of the flow} $\mbox{wind}(F_{e,b_j})$:
$\mbox{wind}(F_{e,b_j}) = \mbox{wind}(P_{e,b_j})$;
\item The \textbf{intersection number of the flow} $\mbox{int}(F_{e,b_j})$:
$\mbox{int}(F_{e,b_j}) = \mbox{int}(P_{e,b_j})$.
\end{enumerate}
\end{definition}

\subsection{The linear system on $({\mathcal N},\mathcal O,\mathfrak l)$ and the rational representation of the edge vectors}\label{sec:linear}

In this Section we recall the main results in \cite{AG6}: edge vectors satisfy a maximal rank linear system of relations at the vertices of ${\mathcal N}$, and the solution of such relations is obtained using flows; in particular, the components of the edge vectors are rational in the weights with subtraction-free denominators.  This representation extends Talaska's formula \cite{Tal2} for Postnikov's boundary measurement map to the internal edges.

From the definition of the components of an edge vector as summations over all paths starting at this edge, it immediately follows that
the edge vectors $E_e$ on $({\mathcal N},\mathcal O,\mathfrak l)$ satisfy the following linear equation at each vertex:  
\begin{enumerate}
\item  At each bivalent vertex with incoming edge $e$ and outgoing edge $f$:
\begin{equation}\label{eq:lineq_biv}
E_e =  (-1)^{\mbox{int}(e)+\mbox{wind}(e,f)} w_e E_f;
\end{equation}
\item At each trivalent black vertex with incoming edges $e_2$, $e_3$ and outgoing edge $e_1$ we have two relations:
\begin{equation}\label{eq:lineq_black}
E_2 =  (-1)^{\mbox{int}(e_2)+\mbox{wind}(e_2, e_1)}\ w_2 E_1,\quad\quad
E_3 =  (-1)^{\mbox{int}(e_3)+\mbox{wind}(e_3, e_1)}\  w_3 E_1;
\end{equation}
\item At each trivalent white vertex with incoming edge $e_3$ and outgoing edges $e_1$, $e_2$:
\begin{equation}\label{eq:lineq_white}
E_3 =  (-1)^{\mbox{int}(e_3)+\mbox{wind}(e_3, e_1)}\ w_3 E_1 + (-1)^{\mbox{int}(e_3)+\mbox{wind}(e_3, e_2)}\ w_3 E_2,
\end{equation}
\end{enumerate}
where $E_k$ denotes the vector associated to the edge $e_k$.

The linear system defined by equations (\ref{eq:lineq_biv}), (\ref{eq:lineq_black}), and (\ref{eq:lineq_white}) at the internal vertices of $(\mathcal N, \mathcal O, \mathfrak{l})$ possesses a unique solution, and, moreover, this solutions can be written as a rational function of the weights with subtraction-free denominator:

\begin{theorem}\textbf{Full rank of the geometric system of equations for edge vectors on $(\mathcal N, \mathcal O, \mathfrak{l})$  and rational representation for the components of vectors $E_e$}\label{theo:consist} \cite{AG6}
Let $(\mathcal N, \mathcal O, \mathfrak{l})$ be a given plabic network with orientation ${\mathcal O}={\mathcal O}(I) $ with respect to the base $I=\{1\le _1 < i_2 < \cdots < i_k \le n\}$, and gauge ray direction $\mathfrak{l}$.

Given a set $\{B_j\,|\,j\in\bar I\}$ of $n-k$ linearly independent vectors assigned to the boundary sinks $b_j$, let the edge vectors at the boundary sinks be defined as in (\ref{eq:vec_bou_sink}): $E_{e_j} = (-1)^{\mbox{int}(e_j)} w_{e_j} B_j$, $j\in\bar I$.  Then
\begin{enumerate}
\item The linear system of equations (\ref{eq:lineq_biv})--(\ref{eq:lineq_white}) at the internal vertices of $(\mathcal N, \mathcal O, \mathfrak{l})$ has full rank. Since the number of equations coincides with the number of unknowns, it provides a unique system of edge vectors on 
$(\mathcal N, \mathcal O, \mathfrak{l})$;
\item The edge vector $E_{e}$ at the edge $e$ defined in (\ref{eq:sum}), is a rational expression in the edge weights with subtraction-free denominator: 
\begin{equation}
\label{eq:tal_formula}
E_{e} =\sum\limits_{j\in\bar I}\ \ \left[ \frac{\displaystyle\sum\limits_{F\in {\mathcal F}_{e,b_j}(\mathcal G)} \big(-1\big)^{\mbox{wind}(F)+\mbox{int}(F)}\ w(F)}{\displaystyle \sum\limits_{C\in {\mathcal C}(\mathcal G)} \ w(C)}\right]\  B_j,
\end{equation}
where notations are as in Definitions~\ref{def:cons_flow} and~\ref{def:edge_flow};
\item If $e$ is the edge starting at the boundary source $b_{i_r}$, then $\mbox{wind}(F)+\mbox{int}(F) = \sigma(i_r,j)\mod{2}$ for any $F\in {\mathcal F}_{e,b_j}(\mathcal G)$, where $\sigma(i_r,j)$ is the number of boundary sources between $i_r$ and $j$ in the orientation $\mathcal O$. Therefore (\ref{eq:tal_formula}) simplifies to
\begin{equation}
\label{eq:tal_formula_source}
\frac{\displaystyle\sum\limits_{F\in {\mathcal F}_{e,b_j}(\mathcal G)} \big(-1\big)^{\mbox{wind}(F)+\mbox{int}(F)}\ w(F)}{\displaystyle\sum\limits_{C\in {\mathcal C}(\mathcal G)} \ w(C)} \, = \,\big(-1\big)^{\sigma(i_r,j)}\, \frac{\displaystyle\sum\limits_{F\in {\mathcal F}_{e,b_j}(\mathcal G)} \ w(F)}{\displaystyle\sum\limits_{C\in {\mathcal C}(\mathcal G)} \ w(C)} \,=	\, A^r_{j}, \ \ j\in \bar I,
\end{equation}
where $A^r_j$ is the the entry of the reduced row echelon matrix $A$ with respect to the base $I$;
\item Finally, if $\mathcal N$ is acyclically oriented (see Definition~\ref{def:acyclic}), then all edge vectors $E_e$ are not-null. Indeed, in such case there exists $\sigma(e, b_j)\in \{0,1\}$ such that $\mbox{wind}(F)+\mbox{int}(F)=\sigma(e, b_j)\mod{2}$ for any $F\in {\mathcal F}_{e,b_j}(\mathcal G)$, and (\ref{eq:tal_formula}) simplifies to
\begin{equation}\label{eq:edge_parity}
E_{e}= \displaystyle\sum\limits_{j\in \bar I} \Big(\sum\limits_{F\in {\mathcal F}_{e,b_j}(\mathcal G)} \big(-1\big)^{\mbox{wind}(F)+\mbox{int}(F)}\ w(F)\Big) B_j
=\sum\limits_{j\in \bar I} \Big((-1)^{\sigma(e, b_j)}\sum\limits_{F\in {\mathcal F}_{e,b_j}(\mathcal G)} \ w(F)\Big) B_j.
\end{equation}
\end{enumerate}
\end{theorem}

\begin{figure}
\centering{\includegraphics[width=0.7\textwidth]{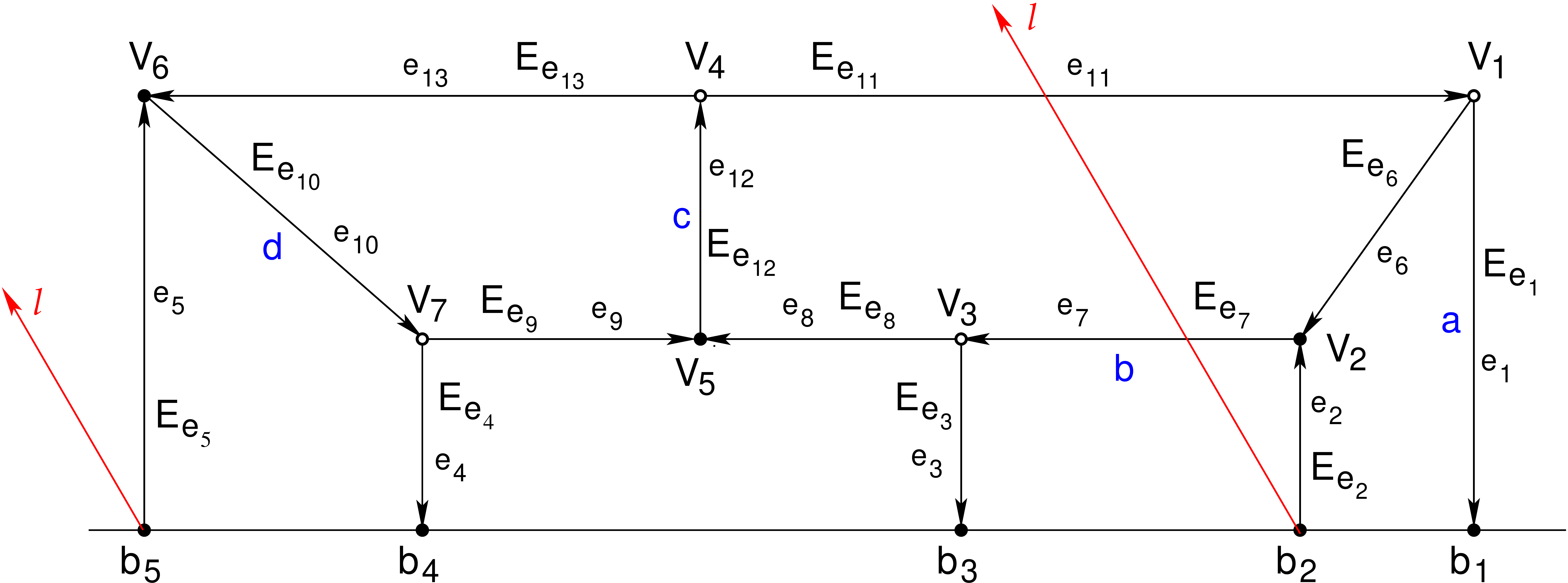}}
\caption{\small{\sl The network of Example~\ref{ex:talaska} represents a point in $Gr^{\mbox{\tiny TNN}}(2,5)$.}\label{fig:gr2_4_edge}}
\end{figure}

\begin{example}\label{ex:talaska}
Consider the network in Figure~\ref{fig:gr2_4_edge}. We assume that all weights are 1 except
\begin{equation}\label{eq:weights}
  w(V_1,b_1) =a, \ \ w(V_2,V_3) =b, \ \ w(V_5,V_4) =c, \ \  w(V_6,V_7) =d. 
\end{equation}
The graph has two oriented cycles $C_1=\{e_{11},e_6,e_7,e_8,e_{12}\}$ and $C_2=\{e_{13},e_{10},e_9,e_{12}\}$.
Our aim is to compare three equivalent approaches:
\begin{enumerate}
\item The direct computation of the edge vectors' components as sums over all paths from the edge $e$ to the corresponding boundary sinks, see (\ref{eq:sum});
\item The computation of the edge vectors' components using the generalized Talaska's formula (\ref{eq:tal_formula}). We also test that they coincide with the original Talaska' formula \cite{Tal2} if the edge is a boundary source edge (see \ref{eq:tal_formula_source});
\item The computation of the edge vectors as solutions of the linear system (\ref{eq:lineq_biv}), (\ref{eq:lineq_black}), (\ref{eq:lineq_white}).  
\end{enumerate}

1. Let us compute the edge vector $E_{e_7}$ using the direct summation. The paths from $e_7$ to $b_1$ are:
\begin{equation*}
{\mathcal P}_{71} = e_7,e_8,e_{12}, C_1^{k_1},C_2^{l_1},\ldots, C_1^{k_s}, C_2^{l_s},e_{11},e_1,
\end{equation*}
for some choice of non--negative indices $s,k_1,	\dots, l_s$, and
\begin{equation*}
\mbox{wind}({\mathcal P}_{71}) = -1 -k_1-\ldots - k_s+ l_1+\ldots+l_s, \ \ \mbox{int}({\mathcal P}_{71}) = 0, \ \ 
w({\mathcal P}_{71}) = abc (bc)^{k_1+\ldots+k_s}(dc)^{l_1+\ldots+l_s}.
\end{equation*}
Therefore the first component of $E_{e_7}$ is
\begin{equation}\label{eq:ex_tal1}
  \left(E_{e_7}\right)_{1} = -acb \sum_{k,l\ge 0}{k+l \choose k} (-1)^{k+l} (bc)^k(cd)^l= \frac{-abc}{1+bc + cd}.
\end{equation}
Similarly, the paths from $e_7$ to $b_3$ are either ${\mathcal P}^{(0)}_{73}=e_7,e_3$ or:
\begin{equation*}
{\mathcal P}^{(1)}_{73} = e_7,e_8,e_{12}, C_1^{k_1},C_2^{l_1},\ldots, C_1^{k_s}, C_2^{l_s},e_{11},e_6,e_7,e_8,
\end{equation*}
\begin{align*}
& \mbox{wind}({\mathcal P}^{(0)}_{73}) = 0 && \mbox{int}({\mathcal P}^{(0)}_{73}) = 1, && w({\mathcal P}^{(0)}_{73})_0 =b \\ 
& \mbox{wind}({\mathcal P}^{(1)}_{73}) = -1 -k_1-\ldots - k_s+ l_1+\ldots+l_s, && \mbox{int}({\mathcal P}^{(1)}_{73}) = 1, \ \ 
&& w({\mathcal P}^{(1)}_{73}) = b^2c (bc)^{k_1+\ldots+k_s}(dc)^{l_1+\ldots+l_s}.
\end{align*}
Therefore
\begin{equation}\label{eq:ex_tal2}
  \left(E_{e_7}\right)_{3} = - b\left[1 - bc \sum_{k,l\ge 0}{k+l \choose k} (-1)^{k+l} (bc)^k(cd)^l\right]  = -b\left[1-
    \frac{bc}{1+bc + cd} \right]=  \frac{-b - bcd }{1+bc + cd}.
\end{equation}
The paths from $e_7$ to $b_4$ have the following form:
\begin{equation*}
{\mathcal P}_{74} = e_7,e_8,e_{12}, C_1^{k_1},C_2^{l_1},\ldots, C_1^{k_s}, C_2^{l_s},e_{13},e_{10}e_4, 
\end{equation*}
\begin{equation*}
\mbox{wind}({\mathcal P}_{74}) = -k_1-\ldots - k_s+ l_1+\ldots+l_s, \ \ \mbox{int}({\mathcal P}_{74}) = 1, \ \ 
w({\mathcal P}_{74}) = bcd (bc)^{k_1+\ldots+k_s}(dc)^{l_1+\ldots+l_s}.
\end{equation*}
Therefore
\begin{equation}\label{eq:ex_tal3}
  \left(E_{e_7}\right)_{4} = -bcd \sum_{k,l\ge 0}{k+l \choose k} (-1)^{k+l} (bc)^k(cd)^l= \frac{-bcd}{1+bc + cd}.
\end{equation}
Finally
\begin{equation}\label{eq:ex_tal4}
E_{e_7} = -\frac{1}{1+bc + cd} \left(abc,0,b+bcd, bcd,0 \right).
\end{equation}
The computation of $E_{e_2}$ follows the same route, with all windings differing by 1 from those of $E_{e_2}$; therefore
\begin{equation}\label{eq:ex_tal5}
E_{e_2} = \frac{1}{1+bc + cd} \left(abc,0,b+bcd, bcd,0 \right).
\end{equation}

2. Let us compute the edge vector  $E_{e_7}$ using (\ref{eq:tal_formula}). We have 3 conservative flows -- the trivial one, $C_1$ and $C_2$. Therefore the denominator in (\ref{eq:tal_formula}) is exactly $1+bc+cd$. From $e_7$ to $e_1$ we have only one flow which coincides with the loop-erased walk $LE({\mathcal P}_{71})= e_7,e_8,e_{12},e_{11},e_1$  For this path 
\begin{equation*}
\mbox{wind}(LE({\mathcal P}_{71})) = -1, \ \ \mbox{int}(LE({\mathcal P}_{71})) = 0, \ \ 
w(LE({\mathcal P}_{71})) = abc,
\end{equation*} 
therefore we get (\ref{eq:ex_tal1}). From $e_7$ to $e_3$ we have one loop-erased walk: $LE({\mathcal P}^{(0)}_{73})=LE({\mathcal P}^{(1)}_{73}) = e_7,e_3$ and two flows:  $F_1=LE({\mathcal P}^{(0)}_{73})$ and  $F_2=LE({\mathcal P}^{(0)}_{73}))\cup C_2. 
$
\begin{equation*}
\mbox{wind}(F_1) = \mbox{wind}(F_2)  = 0 \ \  \mbox{int}(F_1) = \mbox{int}(F_2)  = 1, \ \ 
 w(F_1) =b, \ \  w(F_2) = bcd, 
\end{equation*}  
therefore we get (\ref{eq:ex_tal2}). From $e_7$ to $e_4$ we have only one flow which coincides with the loop-erased walk $LE({\mathcal P}_{74})= e_7,e_8,e_{12},e_{13},e_{10},e_4$.  For this path 
\begin{equation*}
\mbox{wind}(LE({\mathcal P}_{74})) = 0, \ \ \mbox{int}(LE({\mathcal P}_{74})) = 1, \ \ 
w(LE({\mathcal P}_{74})) = bcd,
\end{equation*} 
therefore we get (\ref{eq:ex_tal3}).

We can proceed in a similar way to compute $E_{e_2}$ using either  (\ref{eq:tal_formula}) or the original Talaska formula (\ref{eq:tal_formula_source}), and it is easy to check that both of them coincide with (\ref{eq:ex_tal5}).

3. The linear system of equations is
\begin{align*}
  & E_{e_{11}} = - E_{e_1}- E_{e_6}, &&  E_{e_7} =  E_{e_6} = - E_{e_2},  && E_{e_7} = -b [ E_{e_3} + E_{e_8} ],  &&   E_{e_{12}} = c[ E_{e_{11}} -  E_{e_{13}} ],\\
  &  E_{e_9} =  E_{e_{12}} = - E_{e_8}, &&   E_{e_{13}} =  E_{e_{10}} = E_{e_5}, &&  E_{e_{10}} = d  [ E_{e_4} + E_{e_9} ], && \\
  & E_{e_1} = a [1,0,0,0,0], &&   E_{e_3} =  [0,0,1,0,0], &&  E_{e_4} =  [0,0,0,1,0].  
\end{align*}
Solving it with respect to $E_{e_7}$, one gets
\begin{align*}
  &  E_{e_6} =  E_{e_7} &&  E_{e_2} = - E_{e_7},  && E_{e_{11}} = - E_{e_1}- E_{e_7}, \\
  & E_{e_8} =  - E_{e_3} - \frac{E_{e_7}}{b}, &&  E_{e_9} =  E_{e_3} + \frac{E_{e_7}}{b} , &&  E_{e_{12}} =  E_{e_3} + \frac{E_{e_7}}{b}, \\
  &   E_{e_{10}} = d \left[E_{e_{4}} + E_{e_{3}} + \frac{E_{e_7}}{b} \right], &&  E_{e_{5}} = d \left[E_{e_{4}} + E_{e_{3}} + \frac{E_{e_7}}{b} \right],
   &&  E_{e_{13}} = d \left[E_{e_{4}} + E_{e_{3}} + \frac{E_{e_7}}{b} \right],
\end{align*}
$$
E_{e_3} + \frac{E_{e_7}}{b} = - c\left[ E_{e_1} +  E_{e_7} + d\left[ E_{e_3} +E_{e_4}  + \frac{E_{e_7}}{b} \right] \right],
$$
from which one may easily conclude the $E_{e_7}$ satisfies (\ref{eq:ex_tal4}).
\end{example}

\section{Geometric signatures on plabic networks}\label{sec:signatures}

In this Section we show that, if we identify the edge vector constructed in \cite{AG6} at the edge $e=(U,V)$ with the half-edge vector $z_{U,e}$ up to a sign defined below, then it is possible to define an edge signature, which we call geometric, and to reformulate the linear system for the edge vectors in terms of Lam's relations for the corresponding half-edge vectors. Therefore, by Theorem \ref{theo:consist}, such geometric signature solves Lam's problem. Indeed, if ${\mathcal V}=\R^n$, the signature is geometric and the weights are positive, then the solution of Lam's system of relations at the boundary vertices coincides with the image of Postnikov's boundary measurement map in $\GTNN$ for the same collection of weights.

Moreover, we show that the signature defined combinatorially in \cite{AG3} on the Le-graph using the Le--diagram is geometric.

\subsection{Geometric signatures, half-edge vectors and geometric relations}\label{sec:complete}
   
The system of relations in the edge vectors has full rank for all positive edge weights and its solution is naturally related to the Postnikov's boundary measurement map, therefore it is a natural candidate for providing an explicit solution of Lam's problem of constructing a ``good'' signature. In this Section we define such a signature using the linear system of relations in the edge vectors, and we call it geometric.

\begin{figure}
\centering
{\includegraphics[width=0.48\textwidth]{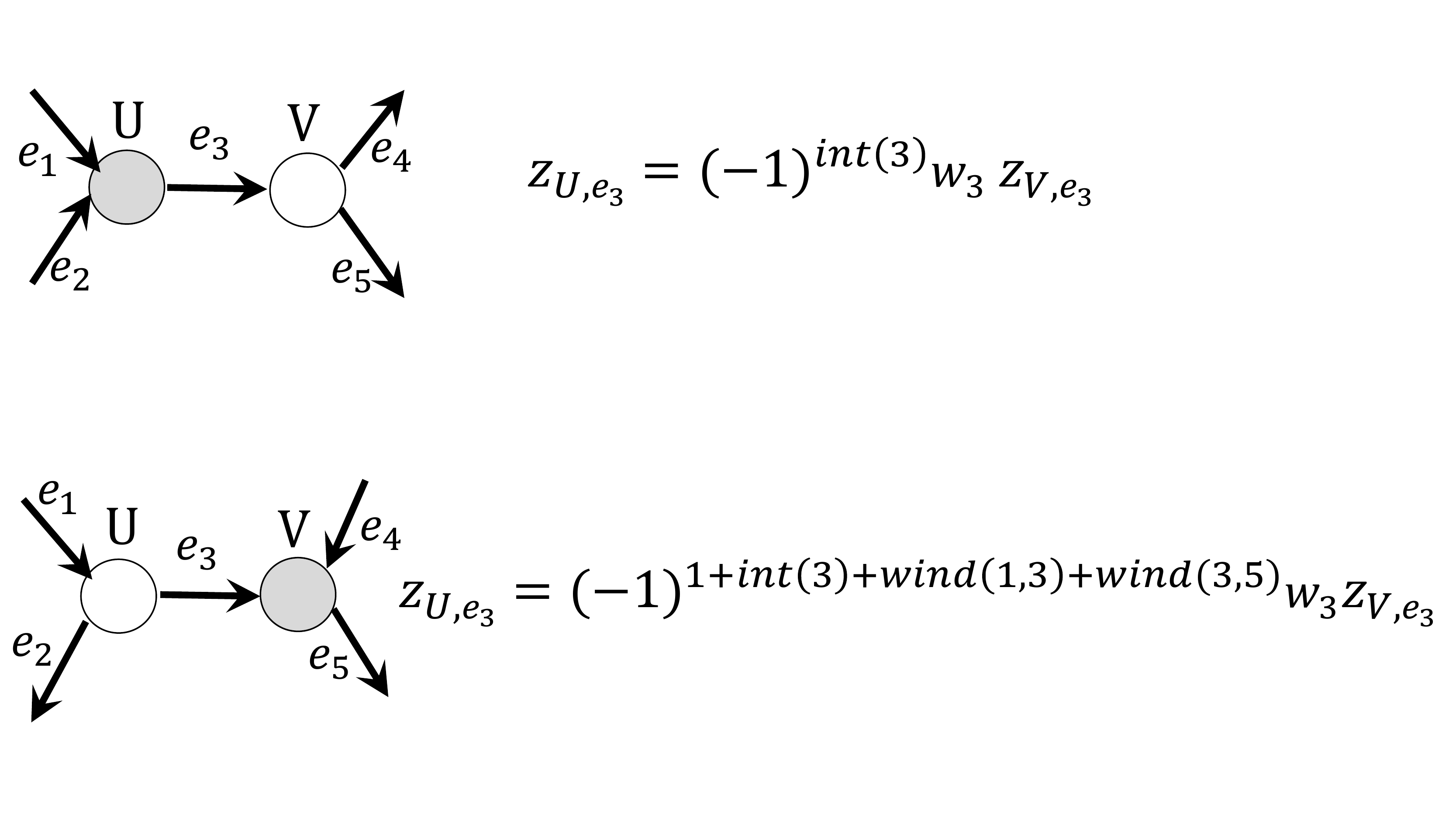}}
\hfill
{\includegraphics[width=0.48\textwidth]{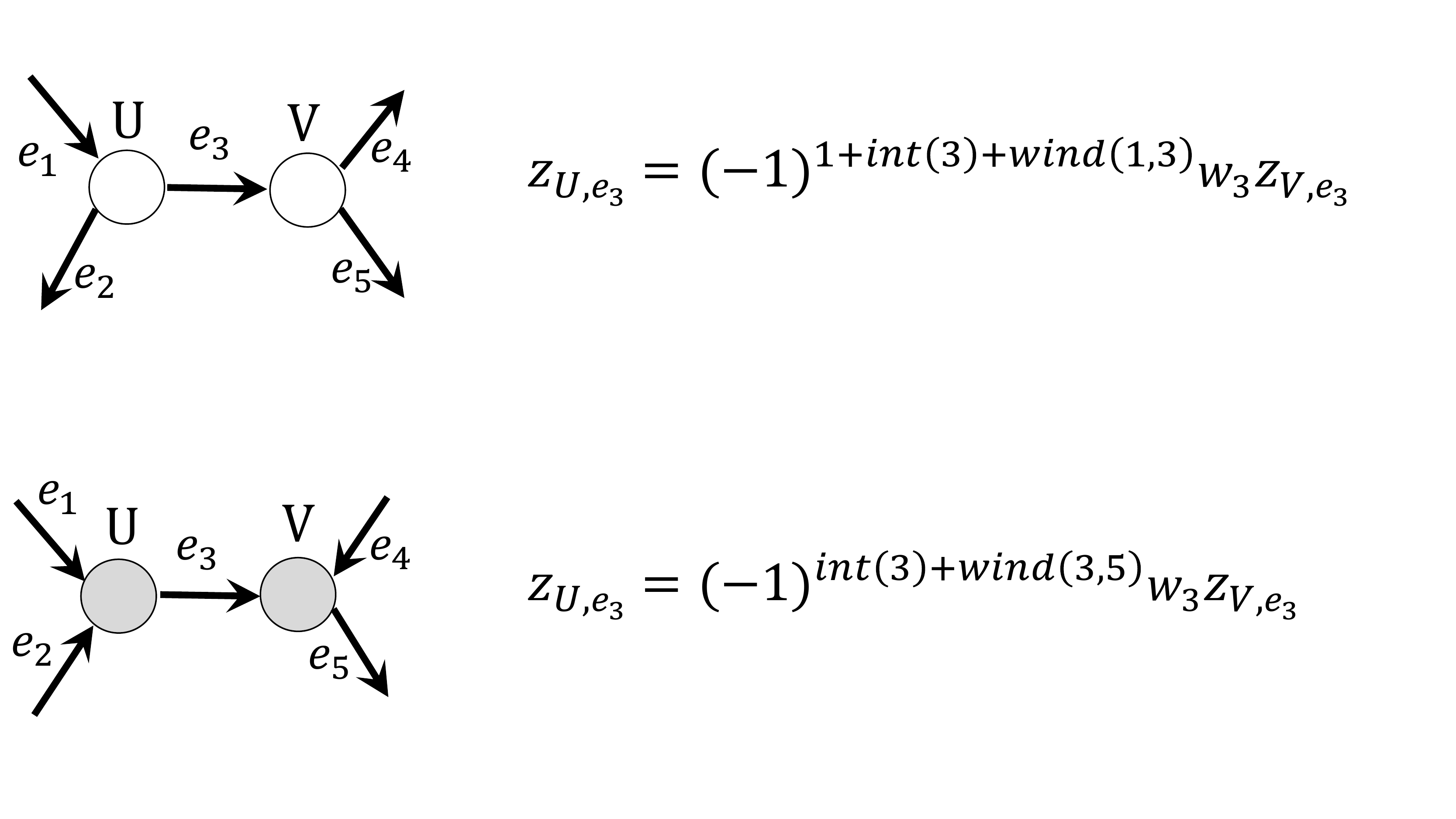}}
{\includegraphics[width=0.48\textwidth]{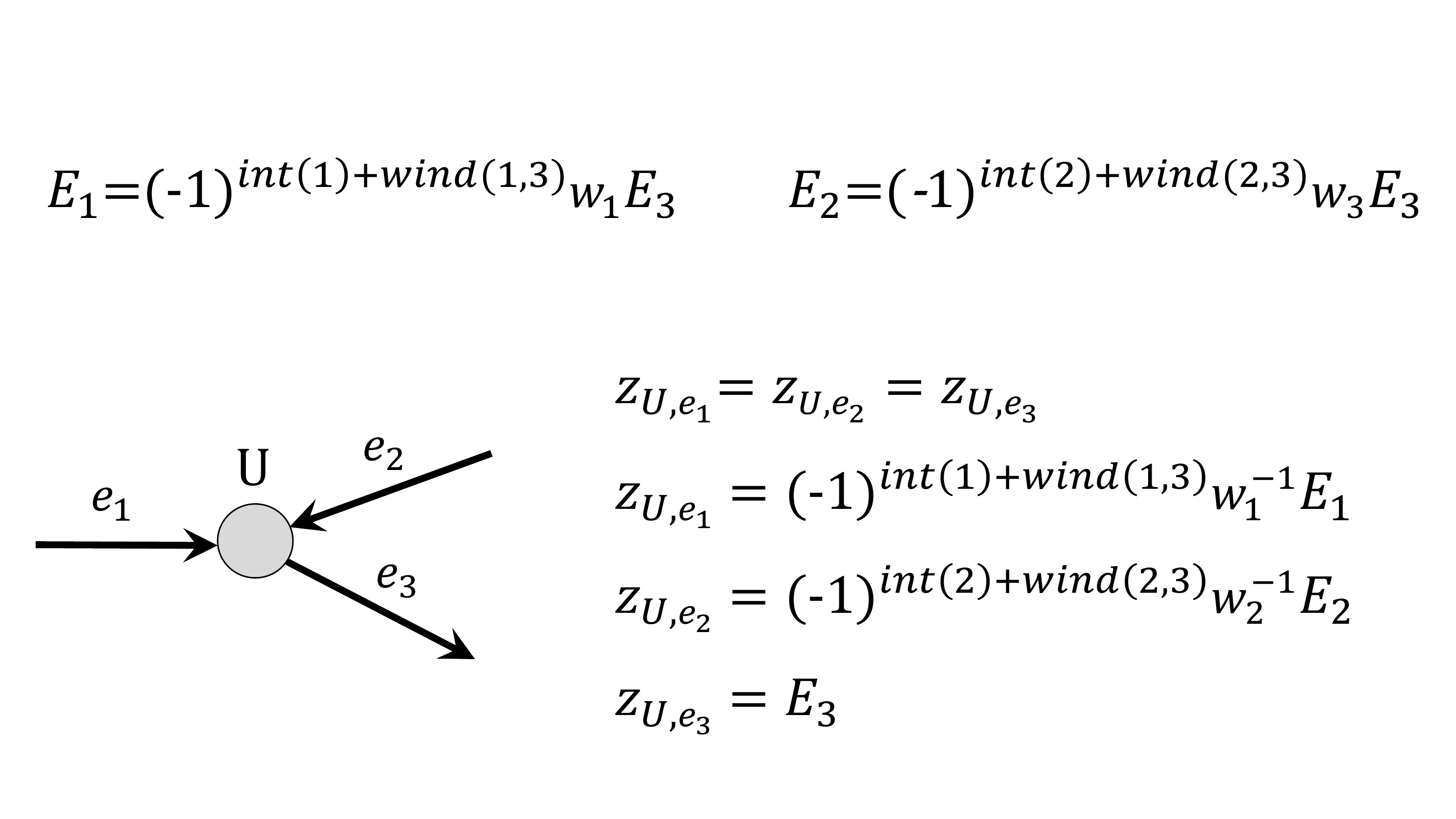}}
\hfill
{\includegraphics[width=0.48\textwidth]{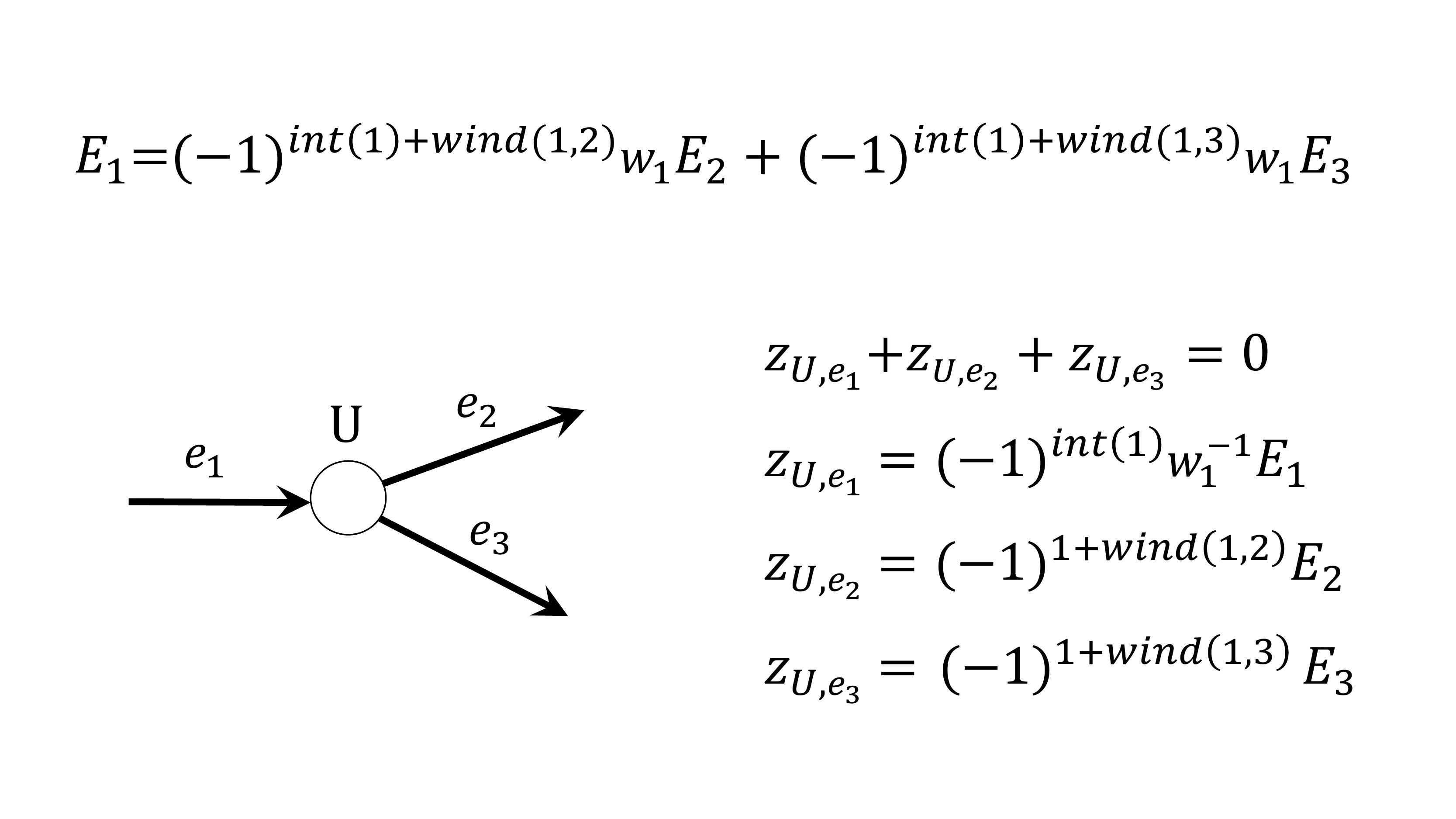}}
\caption{\small{\sl The geometric signature of Definition \ref{def:geometric-signature} [top] and the reformulation of the linear relations in the edge vectors as geometric relations in the half-edge vectors compatible with Lam's \cite{Lam2} approach [bottom]. We use the abridged notations $\mbox{wind}(i,j)\equiv \mbox{wind}(e_i,e_j)$ and $\mbox{int}(i)\equiv \mbox{int}(e_i)$.}\label{fig:lin_lam1}
}
\end{figure}

\begin{definition}\textbf{Geometric signature.} \label{def:geometric-signature}
Let $(\mathcal G, \mathcal O, \mathfrak l)$ be a plabic graph representing a $|D|$--dimensional positroid cell $\S \subset \GTNN$ with perfect orientation $\mathcal O$ associated to the base $I$ and gauge ray direction $\mathfrak l$. We call \textbf{geometric signature} on  $(\mathcal G, \mathcal O, \mathfrak l)$ any signature equivalent to the following one (the right-hand sides of all equations are taken  $\mod(2)$):
\begin{enumerate}
\item If $e=(V,b_j)$ is the edge at the boundary sink $b_j$, $j\in \bar{I}$, then
\begin{equation}
\label{eq:lin_lam_1.0.5}
\epsilon_{V,b_j}=\left\{\begin{array}{ll} \mbox{int}(e), & \mbox{ if } V \mbox{ is black,} \\ 
1+ \mbox{int}(e) + \mbox{wind}(e_1,e), & \mbox{ if } V  \mbox{ is white and } e_1 \mbox{ is incoming at } V;
                        \end{array} \right.          
\end{equation}
\item If $e=(b_i,V)$ is the edge at the boundary source $b_i$, $i\in I$, then
\begin{equation}
\label{eq:lin_lam_1.1}
\epsilon_{b_i,U}=\left\{\begin{array}{ll} 1+\mbox{int}(e)+\mbox{wind}(e,e_3), & \mbox{ if } V \mbox{ is black and }  e_3 \mbox{ is outgoing at } V,\\ 
1+ \mbox{int}(e), & \mbox{ if } V  \mbox{ is white};
      \end{array} \right.
\end{equation}
\item If $e_3=(U,V)$ is an internal edge, then 
\begin{equation}\label{eq:lam_corr_edge}
\resizebox{\textwidth}{!}{$ 
\epsilon_{U,V}= \left\{ \begin{array}{ll}
\mbox{int}(e_3), & \mbox{ if } U \mbox{ black and } V { white};\\
1+\mbox{int}(e_3)+\mbox{wind}(e_1,e_3), & \mbox{ if } U, V \mbox{ white and } e_1 \mbox{ incoming at } U;\\
1+ \mbox{int}(e_3)+\mbox{wind}(e_1,e_3)+\mbox{wind}(e_3,e_5), & \mbox{ if } e_1 \mbox{ incoming at } U \mbox{ white and } e_5 \mbox{ outgoing at } V\mbox{ black;}\\
\mbox{int}(e_3) +\mbox{wind}(e_3,e_5), & \mbox{ if } U, V \mbox{ black and } e_5 \mbox{ outgoing at } V.
\end{array} \right.
$}
\end{equation}
\end{enumerate}
\end{definition}
We illustrate Definition \ref{def:geometric-signature} in Figure \ref{fig:lin_lam1}[top].
Next we re-express the linear relations satisfied by the edge vectors (\ref{eq:lineq_biv})--(\ref{eq:lineq_white}) on the network $(\mathcal N, \mathcal O, \mathfrak l)$ of graph $\mathcal G$ as linear relations in the half-edge vectors $z_{U,e}$ using the following identification (see Figure \ref{fig:lin_lam1}[bottom]):
\begin{enumerate}
\item If $U$ is a black vertex of valency $m$, and $e_m$ denotes the unique outgoing edge at $U$, then we define:
  \begin{equation}\label{eq:z_vs_E_b}
  z_{U,e_j}=\left\{\begin{array}{ll} (-1)^{\mbox{int}(e_j) + \mbox{wind}(e_j,e_m)} w_{e_j}^{-1} E_{e_j}, & \ \ \mbox{if} \ \   j\ne m ;\\
 E_{e_m}, & \ \ \mbox{if} \ \   j= m;                   
  \end{array}\right.  
\end{equation} 
\item If $U$ is a white vertex of valency $m$, and $e_1$ denotes the unique incoming edge at $U$, then we define:
\begin{equation}\label{eq:z_vs_E_w} 
  z_{U,e_j}=\left\{\begin{array}{ll} (-1)^{1+ \mbox{wind}(e_j,e_1)}  E_{e_j}, & \ \ \mbox{if} \ \   j\ne 1 ;\\
 (-1)^{\mbox{int}(e_1)} w_{e_1}^{-1}  E_{e_1}, & \ \ \mbox{if} \ \   j= 1,                   
  \end{array}\right.
\end{equation}  
\end{enumerate}
where, in the formulas above, if $e_j$ connects the vertices $U,V$,  $w_{e_j}$ corresponds to $w_{U,V}$ if $U$ is the initial vertex of $e_j$ in the orientation $\mathcal O$, otherwise it corresponds to $w_{V,U}$.

\begin{remark}
We remark that with our definition of geometric signature, at the boundary sources the half-edge vectors take opposite values to the corresponding edge vectors. The reason of this choice is to keep the total geometric signature of each face invariant with respect to the changes of the perfect orientation of the graph (see Section \ref{sec:vector_changes}).
\end{remark}  

\begin{remark}
When passing from the system for the edge vectors (\ref{eq:lineq_biv})--(\ref{eq:lineq_white}) to the system of relations in the half--edge vectors, there is a gauge freedom at each vertex which is evident from Figure \ref{fig:lin_lam1}[top]. Indeed, at any internal black vertex $U$ we may use 
\begin{equation}\label{eq:corr_black_2}
 z_{U,e_j}=\left\{\begin{array}{ll} (-1)^{1+\mbox{int}(e_j) + \mbox{wind}(e_j,e_m)} w_{e_j}^{-1} E_{e_j}, & \ \ \mbox{if} \ \   j\ne m ;\\
 -E_{e_m}, & \ \ \mbox{if} \ \   j= m,                   
  \end{array}\right.    
\end{equation}
instead of (\ref{eq:z_vs_E_b}).

Similarly, at any white vertex $U$, we may use
\begin{equation}\label{eq:corr_white_2}
 z_{U,e_j}=\left\{\begin{array}{ll} (-1)^{\mbox{wind}(e_j,e_1)}  E_{e_j}, & \ \ \mbox{if} \ \   j\ne 1 ;\\
 (-1)^{1+\mbox{int}(e_1)} w_{e_1}^{-1}  E_{e_1}, & \ \ \mbox{if} \ \   j= 1,              
\end{array}\right.                
\end{equation}
instead of (\ref{eq:z_vs_E_w}).

These alternative choices of the correspondence between half-edge vectors and edge vectors are equivalent up to a vertex gauge transformation; therefore they correspond to gauge-equivalent geometric signatures.
\end{remark}

From now on we enumerate edges from 1 to $m$ counterclockwise at any $m$-valent vertex $V$ without reference to their orientation, and we use cyclical order in summations.  
 We remark that the following Theorem also holds in a perfectly oriented bicoloured network of valency bigger than three. 

\begin{theorem}\textbf{Lam's relations for the geometric signature parametrize $\S$.}\label{lemma:lam_rela} 
Let $(\mathcal G, \mathcal O, \mathfrak l)$ be a plabic graph representing a positroid cell $\S \subset \GTNN$ with perfect orientation $\mathcal O$ associated to the base $I$ and gauge ray direction $\mathfrak l$. 
Let $w_{U,V}$ be positive weights at the oriented edges $e=(U,V)$, let $\mathcal N$ be the associated network of graph $\mathcal G$, and let $\epsilon_{U,V}$ be the geometric signature defined in equations (\ref{eq:lin_lam_1.0.5})--(\ref{eq:lam_corr_edge}). Then the linear system of relations (\ref{eq:lineq_biv})--(\ref{eq:lineq_white}) on $(\mathcal N, \mathcal O, \mathfrak l)$ may be re-expressed in Lam's form:
\begin{enumerate}
\item If $e=(U,V)$, then $z_{U,e} = (-1)^{\epsilon_{U,V}} w_{U,V} z_{V, e}$;
\item At each $m$-valent black vertex $V$ and for any pair of edges $(e_i,e_{i+1})$  at $V$,
$z_{V,e_i} = z_{V,e_{i+1}}$, $i\in [m]$;
\item  For any $m$-valent white vertex $V$,
$\displaystyle \sum_{i=1}^m z_{V,e_i} =0$,
where $e_i$, $i\in [m]$, are the edges at $V$.
\end{enumerate}
Such geometric system of relations has full rank for any choice of positive weights. If the space for half-edge vector is $\R^n$ or $\C^n$, and we assign the canonical basis vectors to the boundary sink half-edge vectors $ z_{b_j, e}$, then the solution at the boundary sources $b_{i_r}$ is $E_{i_r}-A[r]$, where $E_{i_r}$ is the $i_r$-th vector of the canonical base, and $A[r]$ is the $r$-th row of the matrix in reduced row echelon form for the base $I=\{i_1,i_2,\ldots,i_k\}$. Therefore the solution produces the boundary measurement matrix for $\mathcal N$ with respect to the base $I$, and the half-edge vectors $z_{V,e}$ may be expressed as in Theorem~\ref{theo:consist}.

Moreover, if ${\hat\epsilon}_{U,V}$ is the geometric signature equivalent to  $\epsilon_{U,V}$ via the gauge transformation $\eta(U)$, then the system of half--edge vectors ${\hat z}_{U,e}$ on $\mathcal N$ associated to ${\hat\epsilon}_{U,V}$ and satisfying the same boundary conditions at the boundary sinks is given by:
$$
{\hat z}_{U,e} = \left\{\begin{array}{ll} (-1)^{\eta(U)} z_{U,e} & \mbox{if $U$ is an internal vertex,} \\
                        z_{U,e}  & \mbox{if $U$ is a boundary source vertex}. 
  \end{array}\right.
$$
\end{theorem}

\begin{figure}
  \centering
	{\includegraphics[width=0.49\textwidth]{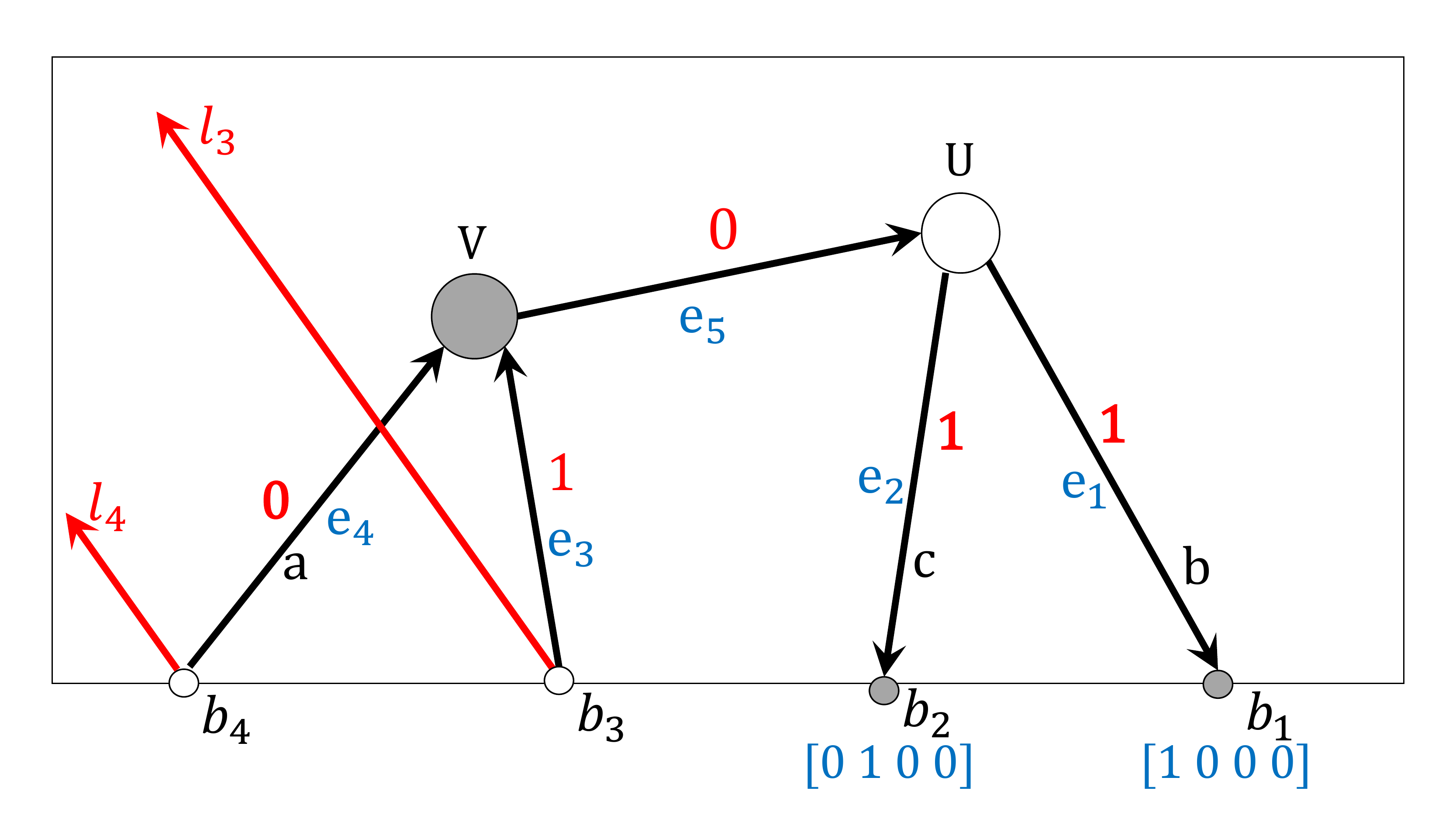}}
  \caption{\small{\sl The geometric signature for Example~\ref{ex:glick}. }\label{fig:glick}}
\end{figure}

\begin{example}\label{ex:glick}
Let us compare the construction of the edge vectors and of the half--edge vectors for the directed network in Figure \ref{fig:glick}.
All edges,  except $e_1,e_2$ and $e_4$, carry unit weight, and $a,b,c$ are assumed to be positive constants. The boundary measurement matrix is
\[
A = \left( \begin{array}{cccc}
b & c & 1& 0\\
-ab & -ac & 0 & 1
\end{array}
\right).
\]
Let $E_i$, $i\in [4]$, be the canonical vectors in $\R^4$.
Choosing $z_{b_1,e_1}=E_1$ and  $z_{b_2,e_2}=E_2$, we obtain the following system of half-edge vectors for the geometric signature associated to the given orientation and gauge ray direction:
\[
\resizebox{\textwidth}{!}{$ 
\begin{array}{lll}
z_{U,e_1} = -b z_{b_1,e_1}, & \, z_{U,e_2} = -c z_{b_2,e_2}, & \, z_{U,e_5} = -z_{U,e_1} - z_{U,e_2} =(b,c,0,0), \\
z_{V,e_3} = z_{V,e_4} = z_{V,e_5} = z_{U,e_5},   &\, z_{b_3,e_3} =-z_{V,e_3} =(-b,-c,0,0) = E_3 -A[1],& \, z_{b_4,e_4} =az_{V,e_4} =(ab,ac,0,0) = E_4 -A[2].\\
\end{array}
$}
\]
The solution to the linear system for the edge vectors $E_{e_j}$, $j\in [5]$, defined in Section \ref{sec:linear} for the same boundary conditions $E_{b_i}=E_i$, $i\in[2]$, is
\[
\begin{array}{lll}
E_{e_1} =bE_1, &\; E_{e_2} =cE_2, &\; E_{e_5} = E_{e_1} +E_{e_2},\\
E_{e_3} = E_{e_5} = (b,c,0,0) = A[1] -E_3, & \; E_{e_4} = -aE_{e_5} = (-ab,-ac,0,0) = A[2] - E_4. &
\end{array}
\]
Notice that the components coincide with (\ref{eq:tal_formula}) in Theorem \ref{theo:consist} and satisfy (\ref{eq:z_vs_E_b}) and 
(\ref{eq:z_vs_E_w}):
\[
z_{U,e_1} = -E_{e_1}, \quad z_{U,e_2} = -E_{e_2}, \quad z_{b_3,e_3} = -E_{e_3}, \quad z_{b_4,e_4} = -E_{e_4}, \quad z_{V,e_5} = E_{e_5}.
\]
\end{example}

\subsection{The master signature on Le--networks is geometric}\label{sec:master} 
In \cite{AG3}, we have used combinatorics on the Le--diagram to associate a master signature to the edges of each canonically oriented Le--graph. The results of Section~\ref{sec:complete} imply that this signature is geometric. Moreover, it is easy to provide a gauge ray direction generating this signature. Let us recall its construction from \cite{AG3}.

\begin{figure}
  \centering
  {\includegraphics[width=0.46\textwidth]{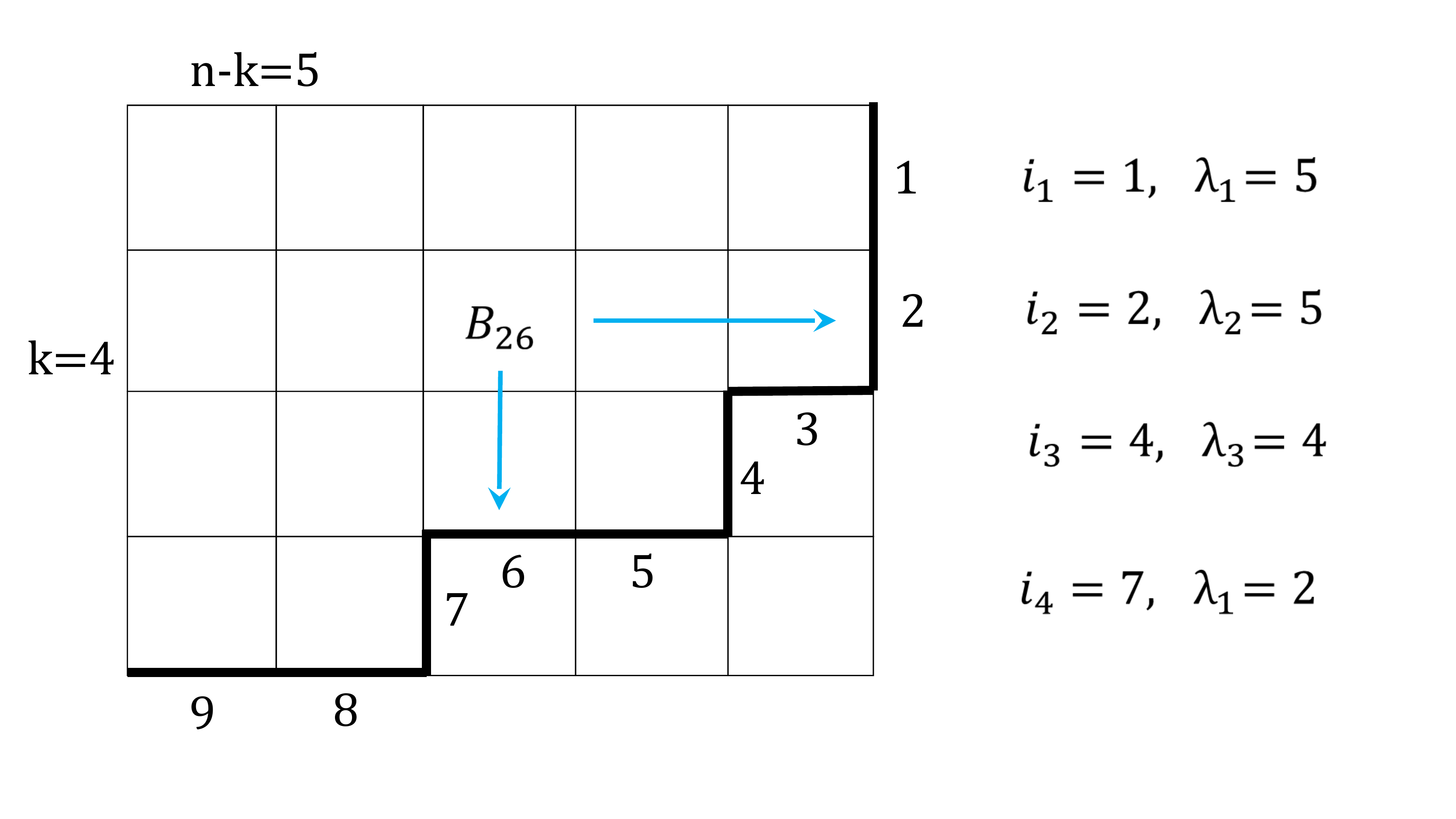}}
  \caption{\label{fig:young}\footnotesize{\sl The Young diagram associated to the partition $(5,5,4,2)$, $k=4$, $n=9$.}}
\end{figure}

\begin{remark}
As in \cite{AG3} we use the ship battle rule to enumerate the boxes of the  Young diagram of a given partition $\lambda$. Let $I=I(\lambda)$ be the pivot set of the $k$ vertical steps in the path along the SE boundary of the Young diagram proceeding from the NE vertex to the SW vertex of the $k\times(n-k)$ bound box, and let ${\bar I} = [n]\backslash I$ be the non--pivot set. Then the box $B_{ij}$ corresponds to the pivot element $i\in I$ and the non--pivot element $j\in {\bar I}$ (see Figure~\ref{fig:young} for an example).
\end{remark}

\begin{definition}\textbf{Le--diagram and Le--tableau.\cite{Pos}}
For a partition $\lambda$, a Le--diagram $L$ of shape $\lambda$ is a filling of the boxes of its Young diagram with $0$'s and $1$'s such that, for any three boxes indexed $(i,k)$, $(l, k)$, $(l, j)$, where $i<l$ and $k<j$, filled correspondingly with $a,b,c$, if $a,c\not =0$, then $b\not =0$. For such a diagram 
denote by $d$ the number of boxes of $D$ filled with $1$s. 

The Le--tableau $T$  is obtained from a Le--diagram $L$ of shape $\lambda$, by replacing all 1s in $L$ by positive 
numbers $w_{ij}$ (weights). We show an example in Figure~\ref{fig:lediag}[left].
\end{definition}

\begin{figure}
  \centering
  {\includegraphics[width=0.49\textwidth]{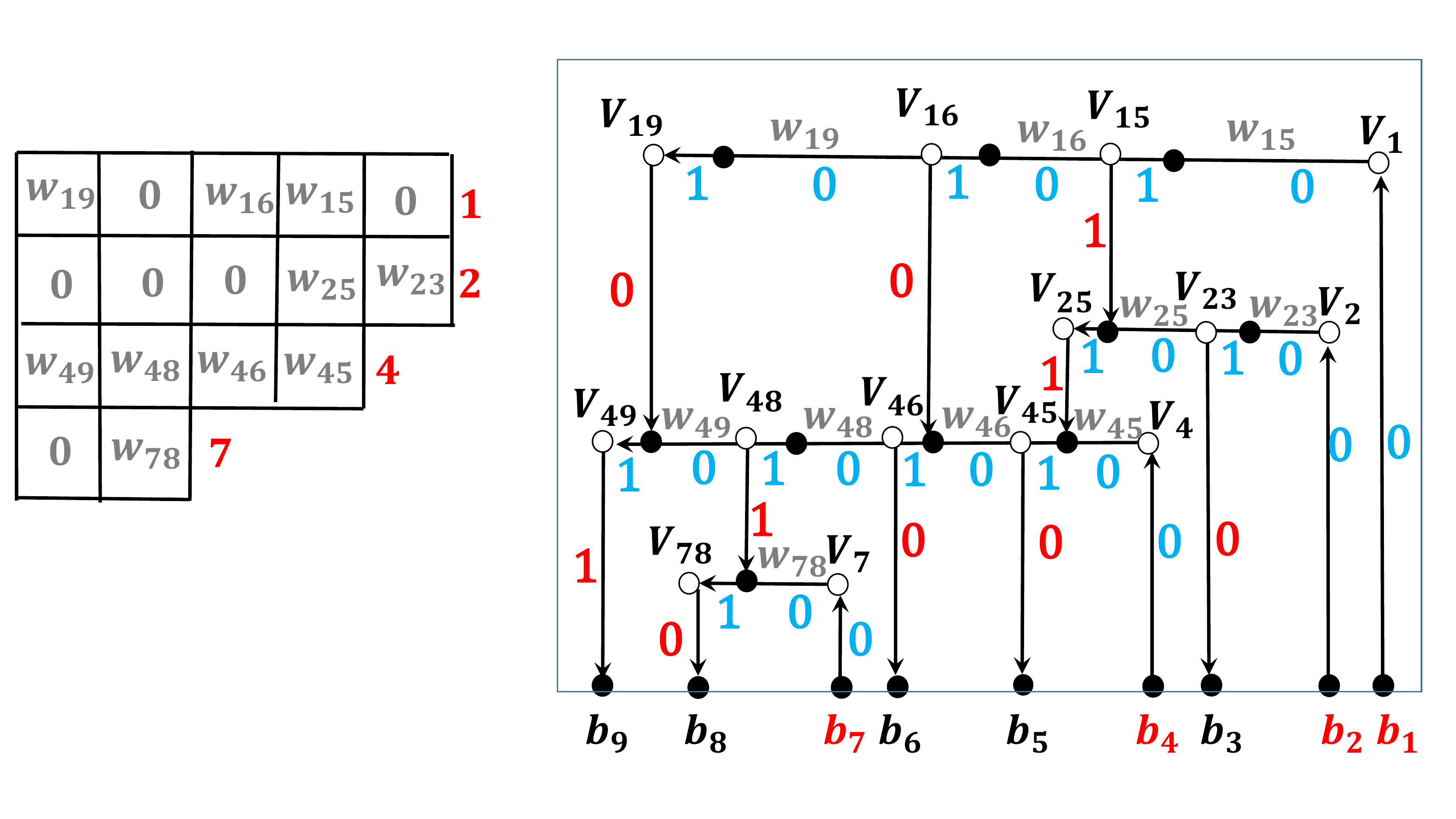}
	\includegraphics[width=0.49\textwidth]{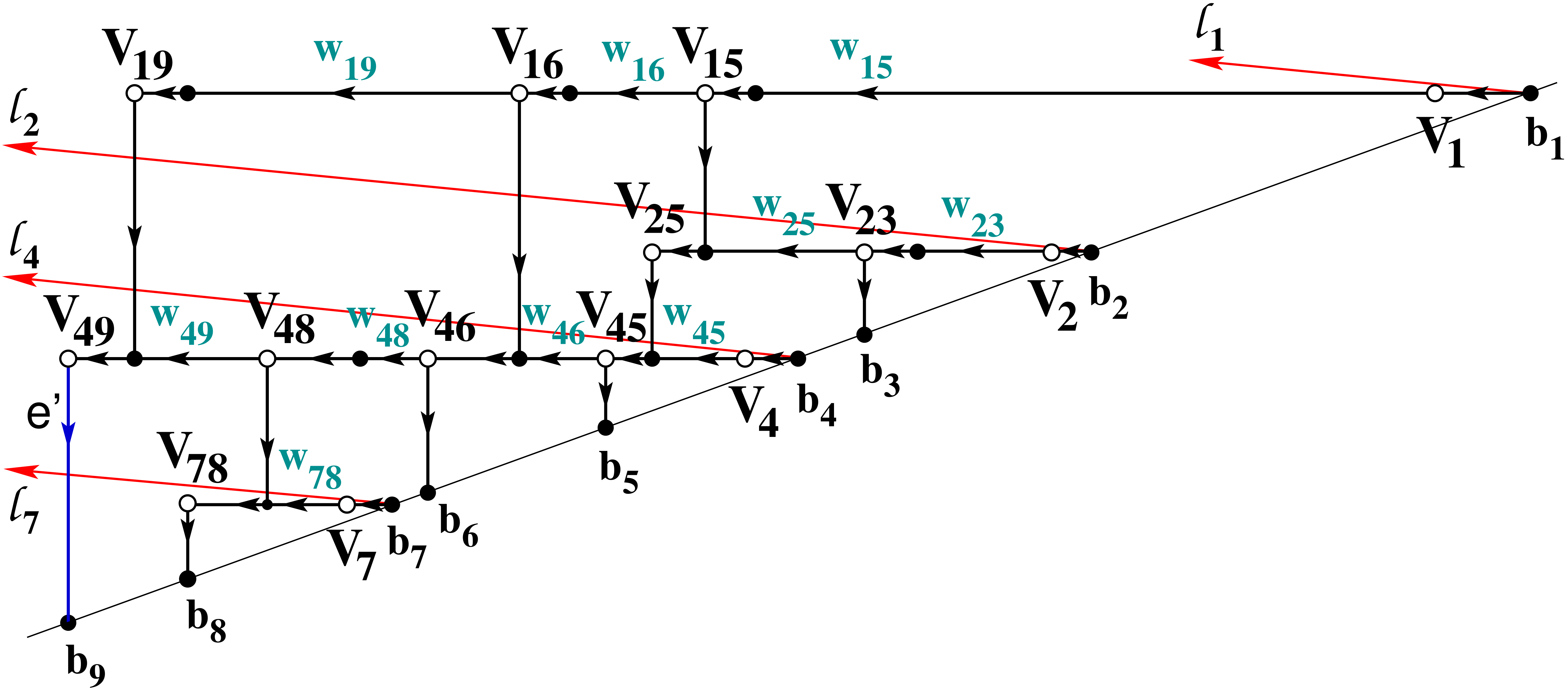}}
	\caption{\footnotesize{\sl Left: a Le--tableau $T$ for the Young diagram of Figure \ref{fig:young} and its Le--network. We mark the master signature of Definition \ref{def:master} with blue colour on the horizontal edges and the edges on the boundary sources, and with red colour on the remaining edges. Right: the choice of gauge ray direction corresponding to the master signature for the same graph.}}
  \label{fig:lediag}
\end{figure}

To any Le--tableau there is associated a canonically oriented bipartite trivalent network defined as follows (see Figure \ref{fig:lediag}[right] for an example of a 10--dimensional positroid cell in $Gr^{\mbox{\tiny TNN}} (4,9)$).

\begin{definition}\label{def:can_Le}\textbf{The trivalent bipartite Le--network \cite{Pos}}
The acyclically oriented perfect trivalent bipartite network  in the disc, ${\mathcal N}$, associated to the Le--tableau $T$ is recursively defined as follows. On the tableau place a white vertex $V_i$ and a boundary source vertex labelled $b_i$ in correspondence of each $i$-th vertical boundary segment, $i\in I$, and a boundary sink vertex $b_j$ for each $j$--th horizontal boundary segment, $j\in \bar I$. colour black all boundary vertices; deform the contour containing all boundary vertices to a horizontal line.
Then starting from the bottom row of the tableau and moving right to left on the given row, do the following 
for any $i\in I$, $j\in \bar I$:
\begin{enumerate}
\item Add a vertical edge of unit weight from $b_i$ to $V_i$. 
\item Ignore all empty boxes;
\item For each box $B_{ij}$ of $T$ filled with $w_{ij}>0$, place a couple of vertices inside the box: $V^{\prime}_{ij}$ coloured black and $V_{ij}$ coloured white to its left in such a way that all internal vertices corresponding to the same row, included $V_{i}$ lie on a common horizontal line.
Then add:
\begin{enumerate}
\item A horizontal edge of unit weight directed from $V_{ij}^{\prime} $ to $V_{ij}$;
\item A horizontal edge of weight $w_{ij}$ from $V_{i, l}$ to $V_{ij}^{\prime}$ , where $V_{i, l}$ is the first white vertex on the same horizontal line to the right of $V_{il}$ (here $V_{i0}$ means $V_i$);
\item A vertical edge of unit weight and oriented downwards from $V_{ij}$ to either the boundary sink $b_j$ if all boxes in the same column and below $B_{ij}$ are empty, or to the black vertex $V^{\prime}_{lj}$ if $B_{lj}$ is the first filled box below $B_{ij}$.
\end{enumerate}
\end{enumerate}
\end{definition}

In \cite{AG3}, we define the following signature on the Le--graph using the Le--diagram.

\begin{definition}\textbf{The master signature for canonically oriented Le--graphs \cite{AG3}}\label{def:master}
Let $\mathcal G$ be the Le--bipartite graph associated to the positroid cell $\S \subset Gr^{\mbox{\tiny TNN}} (k,n)$ acyclically oriented with respect to the lexicographically minimal base $I$ and
let $L$ be its Le--diagram. We define the master signature $\epsilon^{\mbox{\tiny mas}}$ on $\mathcal G$ as follows
\begin{equation}\label{eq:master}
\resizebox{\textwidth}{!}{$
\epsilon^{\mbox{\tiny mas}}_{e} = \left\{ \begin{array}{ll}
0 & \mbox{ if } e \mbox{ is a horizontal edge starting at a white vertex and ending at a black one},\\
1 & \mbox{ if } e \mbox{ is a horizontal edge starting at a black vertex and ending at a white vertex},\\
0 & \mbox{ if } e \mbox{ is an edge starting at a boundary source}, \\
\#(I\cap(i,l]) \mod(2) &  \mbox{ if } e=(V_{ij},V^{\prime}_{lj}) \mbox{ is vertical and incident at internal vertices},\\
\#(I\cap(i,j)) \mod(2) &  \mbox{ if } e=(V_{ij},b_{j}) \mbox{ is vertical and ends at a boundary sink}.
\end{array}
\right.$}
\end{equation}
\end{definition}

\begin{proposition}\textbf{The master edge signature is of geometric type}
Let $\mathcal G$ be Le--graph associated to the positroid cell $\S \subset Gr^{\mbox{\tiny TNN}} (k,n)$. Then its master signature $\epsilon^{\mbox{\tiny mas}}$ is of geometric type.
\end{proposition}
\begin{proof}
  To produce the geometric signature for a canonically oriented  Le-network, it is sufficient to draw the gauge rays almost horizontally (see Figure \ref{fig:lediag} right). Indeed, all windings for this choice are 0, and the intersections with the vertical edges have the same parity as the master signature. Moreover, it is clear that this signature is gauge-equivalent to (\ref{eq:master}) with
  $$
  \eta(U)=\left\{\begin{array}{ll} 0 & \mbox{if  } U \mbox{  is black},  \\ 1 & \mbox{if  } U \mbox{  is white.}
                 \end{array}  \right.
  $$ 
\end{proof}

In \cite{AG3}, we have proven that the system of relations associated to the master signature has full rank and explicitly constructed its solution recursively, verifying that the vectors at the boundary sources satisfy Theorem \ref{theo:consist} for any choice of positive weights.

\section{Invariance of the geometric signature with respect to the network gauge freedoms}\label{sec:vector_changes}

In this Section we prove that the following transformations of the network:
\begin{enumerate}
\item Changes of perfect orientation;
\item Changes of gauge direction;
\item Changes of graph vertices positions respecting the topology,
\end{enumerate}
do not change the equivalence class of the geometric signature. Therefore, for any graph the geometric signature is uniquely defined modulo gauge transformations.

The proof of this statement is rather straightforward: for any such transformation we explicitly calculate the corresponding gauge change.

Following \cite{Pos}, any change of perfect orientation can be represented as a finite composition of elementary changes of orientation, each one consisting in a change of orientation either along a simple cycle ${\mathcal Q}_0$ or along a non-self-intersecting oriented path ${\mathcal P}$ from a boundary source $i_0$ to a boundary sink $j_0$.

\subsection{Some auxiliary indices}\label{sec:auxiliary}
Following \cite{AG6} let us introduce the following auxiliary indices to study the effect of changes of orientation.

\begin{definition}\textbf{Cyclic order.} Generic triples of vectors in the plane have a natural cyclic order. We write $[f,g,h]=0$ if the triple $f$, $g$, $h$ is ordered counterclockwise, and $[f,g,h]=1$ if the triple $f$, $g$, $h$ is ordered clockwise (see Fig~\ref{fig:cyclic_order}).
\end{definition}  
\begin{figure}
  \centering
  \includegraphics[width=0.3\textwidth]{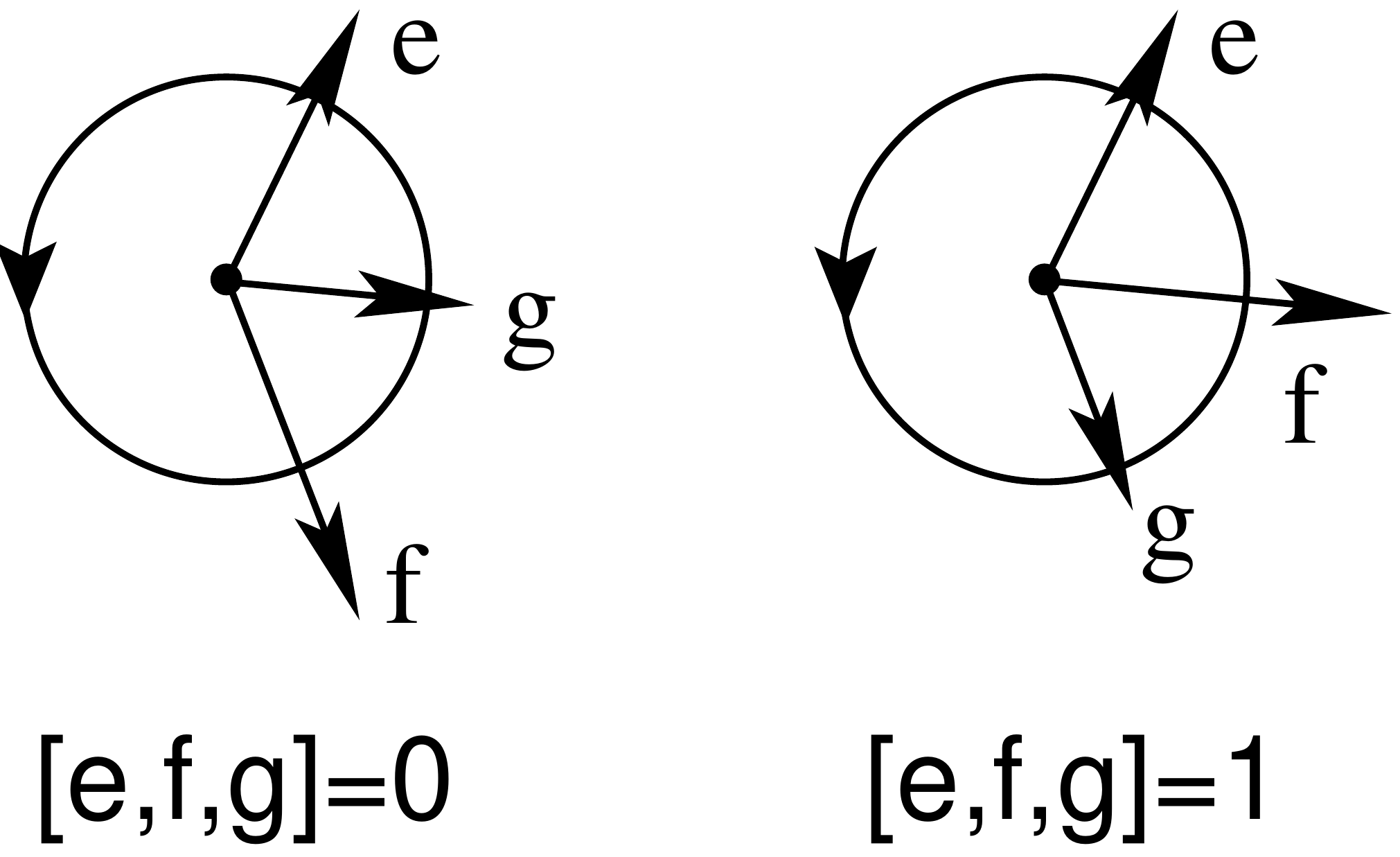}
  \caption{Cyclic order on triples of vectors. By definition $[e,f,g]=[f,g,e]=[g,e,f]=1-[e,g,f]=1-[g,f,e]=1-[f,e,g]$.}
	\label{fig:cyclic_order}
\end{figure}

Assume that we change the perfect orientation either along a non-self-intersecting oriented path from a boundary source to a boundary sink or along a simple cycle. Then mark all regions of the disc by either $+$ or $-$ using the following rule.
\begin{enumerate}
\item If $\mathcal P_0$ is a non-self-intersecting oriented path from a boundary source $i_0$ to a boundary sink $j_0$ in the initial orientation of ${\mathcal N}$, one divides the interior of the disc into a finite number of regions bounded by the gauge ray ${\mathfrak l}_{i_0}$ oriented 
upwards, the gauge ray ${\mathfrak l}_{j_0}$ oriented downwards, the path $\mathcal P_0$ in the initial orientation and the boundary of the disc divided into two arcs, each oriented from $j_0$ to $i_0$. Then a region is marked with a $+$ if its boundary is
oriented, otherwise with $-$.
\item If $\mathcal Q_0$ is a closed oriented simple path, it  divides the interior of the disc into two regions: we mark the region external to $\mathcal Q_0$ with a $+$ and the internal region with $-$. 
\end{enumerate}
\begin{figure}
  \centering
	{\includegraphics[width=0.4\textwidth]{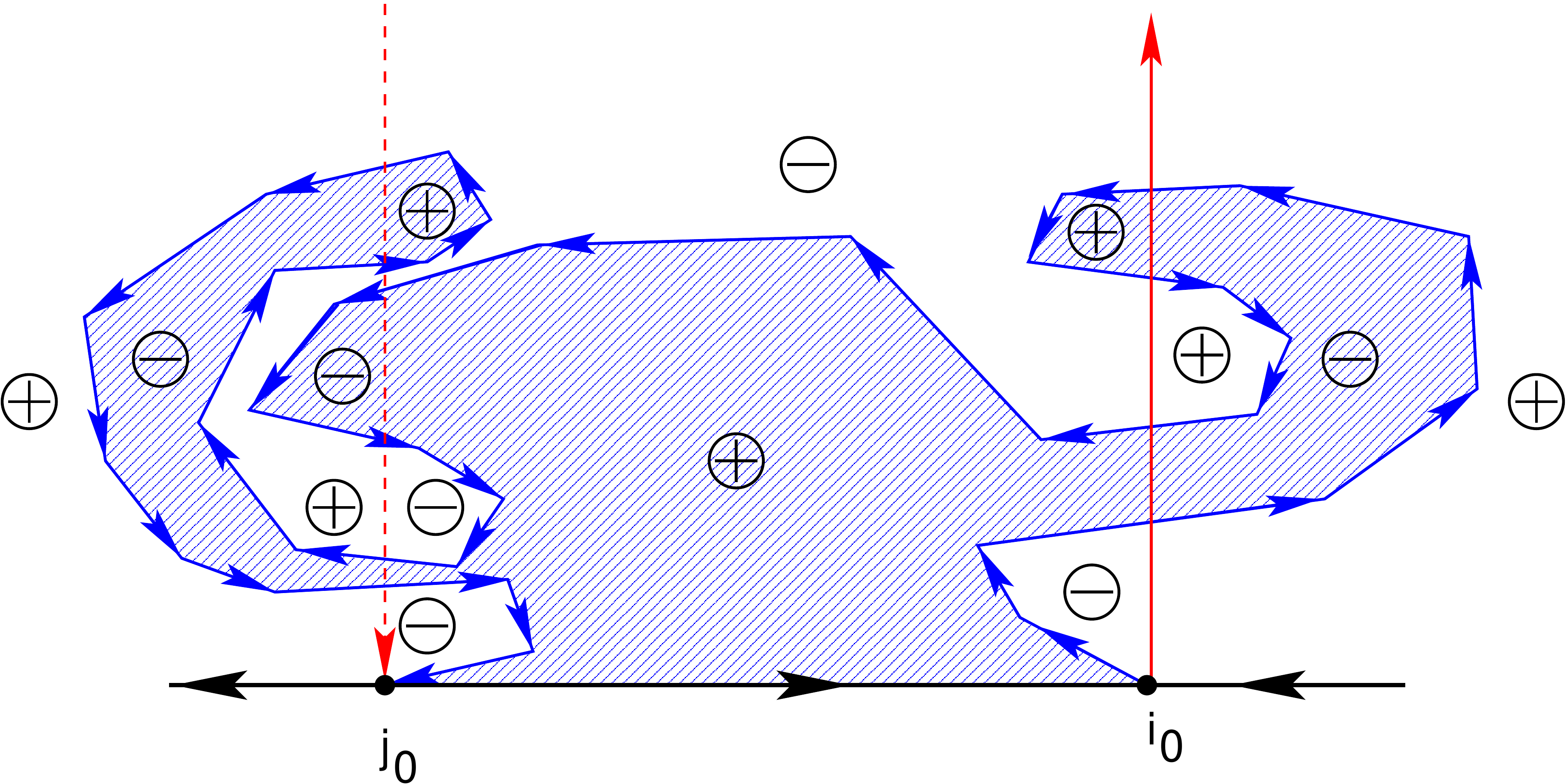}}
	\hfill
	{\includegraphics[width=0.4\textwidth]{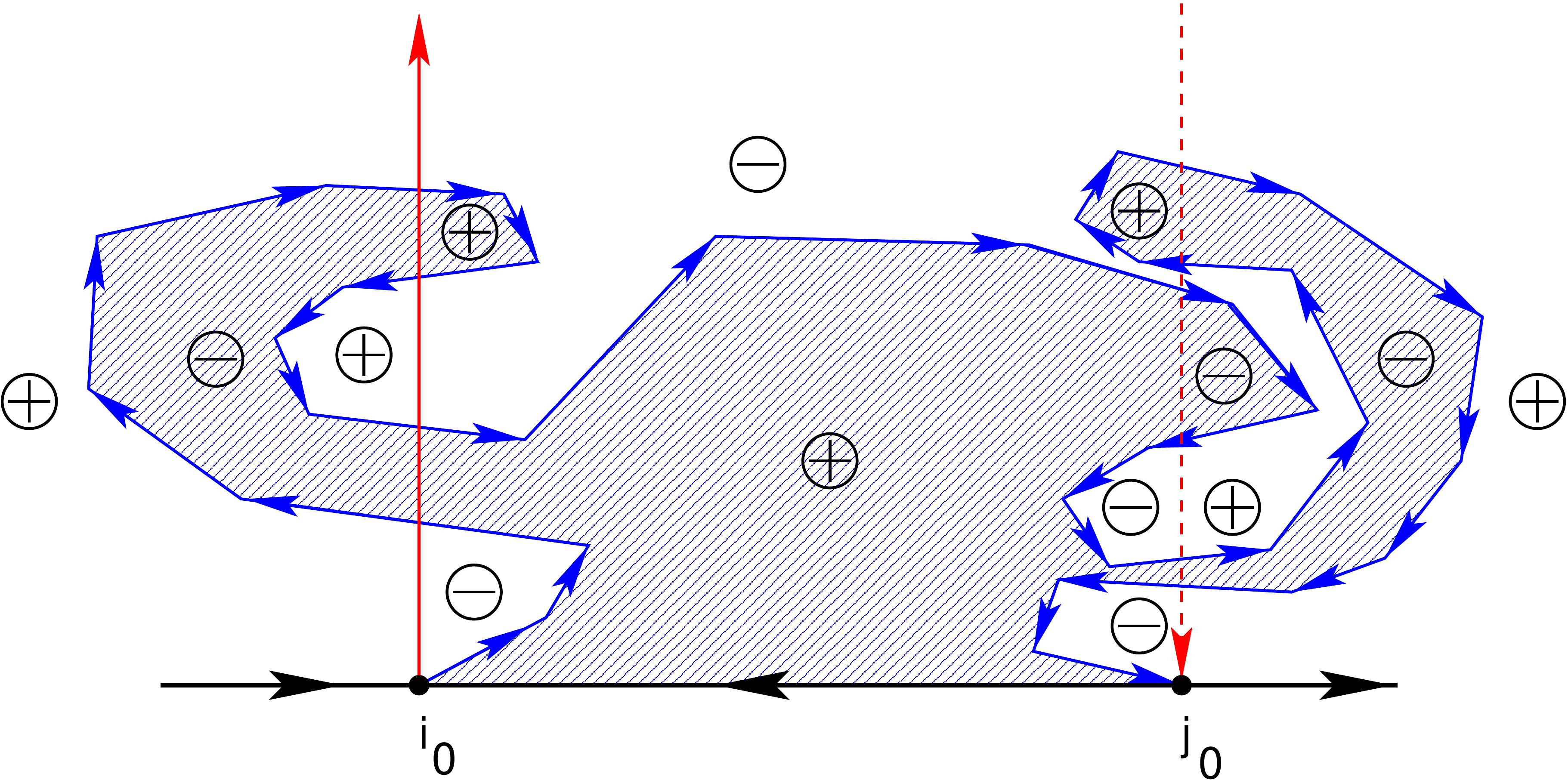}}
  \caption{\small{\sl We illustrate the marking of the regions.}\label{fig:inv_symb}}
\end{figure}

This marking remains invariant after the change of orientation.
\begin{enumerate}
\item If the edge $e\not \in \mathcal P_0$ (respectively $e\not \in \mathcal Q_0$), then we assign it the following index:
\begin{equation}\label{eq:eps_not_path}
\gamma (e) =
\left\{ 
\begin{array}{ll} 
0 & \mbox{ if the starting vertex of } e \mbox{ belongs to a } + \mbox{ region, }\\
1 & \mbox{ if the starting vertex of } e \mbox{ belongs to a } - \mbox{ region, }\\
\end{array}
\right.
\end{equation}
where, in case the initial vertex of $e$ belongs to $\mathcal P_0$ or $\mathcal Q_0$, we make an infinitesimal shift of the starting vertex in the direction of $e$ before assigning the edge to a region;
\item If the edge $e\in \mathcal P_0$ (respectively $e\in \mathcal Q_0$), we assign it the following two indices $\gamma_1(e)$,  $\gamma_2(e)$ using the initial orientation $\mathcal O$:
\begin{enumerate}
\item We look at the region to the left and near the ending point of $e$, and assign the index 
\begin{equation}\label{eq:gamma1}
\gamma_1 (e) =
\left\{ 
\begin{array}{ll} 
0 & \mbox{ if the region is marked with } + ,\\
1 & \mbox{ if the region is marked with } - ;\\
\end{array}
\right.
\end{equation}
\item We consider the ordered pair $(e, \mathfrak{l})$ and assign the index 
\begin{equation}\label{eq:gamma2}
\gamma_2 (e) = \frac{1-s(e,\mathfrak{l})}{2}
\end{equation}
with $s(\cdot,\cdot)$ as in (\ref{eq:def_s}).
\end{enumerate}
\end{enumerate}
It is easy to check that $\gamma_1(e)+\gamma_2(e)+\mbox{int}(e) $ does not change after the change of orientation.

\subsection{Invariance of the geometric signature}

By definition, geometric signatures depend on the orientation of the graph, on the position of the vertices inside the disc and on the gauge ray direction. In this Section we show that all these transformations generate gauge transformations of the signature, therefore geometric signatures are well-defined.

\begin{remark}\label{rem:gauge_vertices}\textbf{Vertex gauge freedom of the graph} The boundary map is the same if we move vertices in $\mathcal G$ without changing their relative positions in the graph. Such transformation acts on edges via rotations, translations and contractions/dilations of their lenghts. Any such transformation may be decomposed in a sequence of elementary transformations in which a single vertex is moved whereas all other vertices remain fixed (see Figure \ref{fig:vertex_gauge_vectors}).
\end{remark}
\begin{figure}
  \centering{\includegraphics[width=0.6\textwidth]{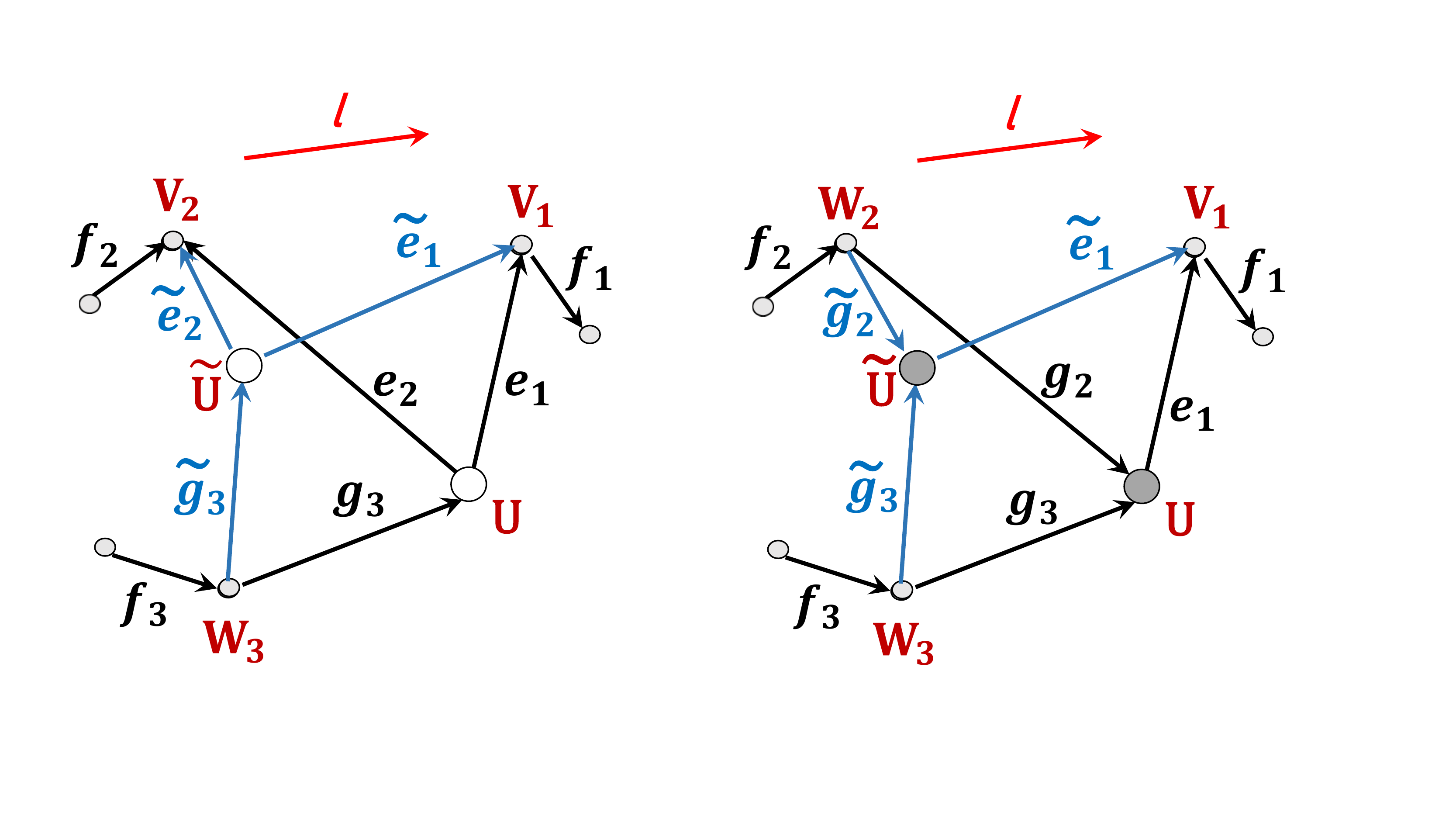}}
	\vspace{-.9 truecm} 
  \caption{\small{\sl The vertex gauge transformation at a white [left] and at a black [right] vertex consists in moving an internal vertex from position $U$ to $\tilde U$.}\label{fig:vertex_gauge_vectors}}
\end{figure}

\begin{theorem}\textbf{The effect of the network transformations on signatures and the uniqueness of the geometric signature}\label{thm:z_orient_gauge}
For any given plabic graph representing a positroid cell $\S$ there is a unique geometric signature up to gauge equivalence.
Indeed, let  $(\mathcal N, \mathcal O, \mathfrak l)$ and $(\hat{\mathcal N}, \hat{ \mathcal O},\hat{\mathfrak l})$ be a pair of plabic networks representing the same point $[A]\in \S \subset \GTNN$, where $\hat{\mathcal N}$ is related to $\mathcal N$ by a composition of changes of the perfect orientation, of the gauge ray direction and of the positions of the vertices. Let $\epsilon_{e}$, ${\hat\epsilon}_{e}$ be the geometric signatures on them defined by equations (\ref{eq:lin_lam_1.0.5}), (\ref{eq:lin_lam_1.1}), (\ref{eq:lam_corr_edge}).

Then $\epsilon_{e}$ and ${\hat\epsilon}_{e}$ are gauge equivalent in the sense of Definition~\ref{def:admiss_sign_gauge}, that is there exists a gauge function $\eta(U)\in\{0,1\}$ defined at the internal vertices of the initial graph such that  
\begin{equation}\label{eq:gauge_gen}
{\hat \epsilon}_{U,V} = \left\{\begin{array}{ll}
\epsilon_{U,V} +\eta(U)+\eta(V), \mod(2) & \mbox{ if } e=(U,V) \mbox{ is an internal edge},\\
\epsilon_{U,V} +\eta(U), \mod(2) & \mbox{ if } e=(U,V) \mbox{ is the edge at some boundary vertex } V.
\end{array}
\right.
\end{equation}
More precisely, the gauge function may be explicitly computed by composing the following elementary transformations:
\begin{enumerate}
\item If $(\mathcal N, \hat{ \mathcal O},\hat{\mathfrak l})$ is obtained from $\mathcal N$ reversing the orientation either along an edge loop-erased path $\mathcal P$ from the boundary source $i_0$ to the boundary sink $j_0$ or along a closed cycle $\mathcal Q$, and $\hat{\mathfrak l}= \mathfrak l$, then the signatures $\epsilon_{e}$ and ${\hat\epsilon}_{e}$ are related by the gauge transformation (\ref{eq:equiv_sign}) with the following gauge function $\eta(U)$:
\begin{equation}\label{eq:zeta_orient}
\eta(U) = \left\{ \begin{array}{ll} 
                    0 , & \mbox{ if } U \not \in \mathcal Q\ (\mathcal P) \mbox{ and belongs to a } + \mbox{ region};\\
                    1 , & \mbox{ if } U \not \in \mathcal Q\ (\mathcal P)\mbox{ and belongs to a } - \mbox{ region};\\
\mbox{wind}(e_1,e_2)+\gamma_1 (e_1) +\gamma_2(e_2) +1, & \mbox{ if } U \in \mathcal Q\ (\mathcal P) \ \ \mbox{is white};\\
\mbox{wind}(e_1,e_2)+\gamma_1 (e_1) +\gamma_2(e_1),    & \mbox{ if } U \in \mathcal Q\ (\mathcal P) \ \ \mbox{is black}. 
\end{array} \right.
\end{equation}
Here $e_1$, $e_2$ is the pair of edges changing orientation at $U$ and $e_1$ is an incoming edge at $U$ in the initial orientation. The indices in the right hand side of (\ref{eq:zeta_orient}) are as in (\ref{eq:gamma1}),  (\ref{eq:gamma2}).

\item If $(\hat{ \mathcal N}, \mathcal O,\hat{\mathfrak l})$ is obtained from $\mathcal N$ changing the gauge ray direction from $\mathfrak l$ to $\hat{ \mathfrak l}$, then the signatures $\epsilon_{e}$ and ${\hat\epsilon}_{e}$ are related by the gauge transformation  (\ref{eq:equiv_sign}) with the following gauge function $\eta(U)$:
\begin{equation}\label{eq:zeta_gauge}
\eta(U) = \left\{ \begin{array}{ll} 
\mbox{cr}(U)+\mbox{par}(e_m),    & \mbox{ if } U  \ \ m-\mbox{valent white};\\
\mbox{cr}(U)+\mbox{par}(e_1),    & \mbox{ if } U  \ \ m-\mbox{valent black}.
\end{array} \right.
\end{equation}
In the right hand side of (\ref{eq:zeta_gauge}), $\mbox{par(e)}$ is 1 if $e$ belongs to the angle $\widehat{{\mathfrak l},{\mathfrak l}^{\prime}}$ , and 0 otherwise; $\mbox{cr}(U)$ denotes the number of gauge rays passing the vertex $U$ during the rotation from  ${\mathfrak l}$ to ${\mathfrak l}^{\prime}$ inside the disc; $e_m$ (respectively $e_1$) denotes the unique incoming (respectively outgoing) edge at $U$;
\item If $(\hat{ \mathcal N}, \mathcal O,\hat{\mathfrak l})$ is obtained from $\mathcal N$ moving an internal vertex $U$ where notations are as in Figure \ref{fig:vertex_gauge_vectors}, then the signatures $\epsilon_{e}$ and ${\hat\epsilon}_{e}$ are related by the gauge transformation  (\ref{eq:equiv_sign}) with the following gauge function:
  \begin{enumerate}  
  \item If an internal vertex $V$ is not connected to $U$ by an edge, then $\eta(V)=0$;
  \item $\eta (U)$ is given by (\ref{eq:zeta_gauge});  
  \item If the oriented edge $e_i=(U,V_i)$ connects $U$ with $V_i$ ($i=1$ if $U$ is black and $i\in[m-1]$ if $U$ is white), then     
\begin{equation}\label{eq:zeta_vertex_gauge_V}
      \eta(V_i) =  \left\{ \begin{array}{ll}    
\mbox{par}(e_i), & \mbox{ if } {V_i} \mbox{ is white};\\
0,    & \mbox{ if } V_i \mbox{ is black}.
\end{array} \right.
\end{equation}
\item If the oriented edge $g=(W_j,U)$ connects $W_j$ with $U$  ($j\in[2,m]$ if $U$ is black and $j=m$ if $U$ is white), then     
\begin{equation}\label{eq:zeta_vertex_gauge_W}
      \eta(W_j) =  \left\{ \begin{array}{ll}    
\mbox{par}(e_j), & \mbox{ if } {W_j} \mbox{ is black};\\
0,    & \mbox{ if } W_j \mbox{ is white}.
\end{array} \right.
\end{equation}
\end{enumerate}

\end{enumerate}
\end{theorem}

To prove Theorem~\ref{thm:z_orient_gauge} it is sufficient to check the effect of transformations case by case. We do these calculations in Appendix~\ref{app:invariance}. In the following examples we apply Theorem~\ref{thm:z_orient_gauge} to check that the initial and transformed geometric signatures are related by a gauge transformation at the vertices.

\begin{example}
\begin{figure}
\centering{\includegraphics[width=0.5\textwidth]{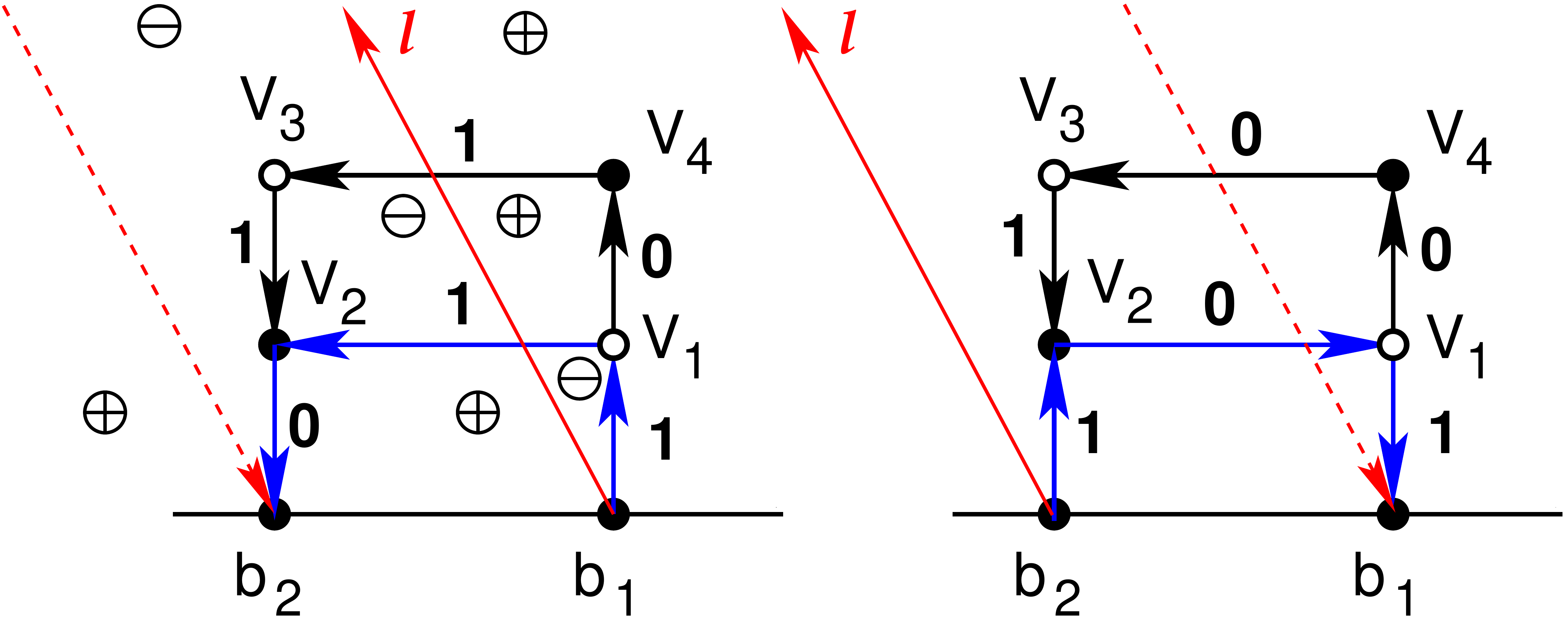}}
\caption{\small{\sl The computation of the gauge function in the case of a change of orientation. The bold numbers are the geometric edge signatures.}\label{fig:acycl_2}}
\end{figure}
Consider the graphs in Figure~\ref{fig:acycl_2}. We mark the regions in the initial graph following the rules of Section~\ref{sec:auxiliary}, and the edge signatures are calculated using Definition~\ref{def:geometric-signature}. Let us compute the gauge function $\eta(U)$ using Formula (\ref{eq:zeta_orient}). Since
\begin{equation*}
 \gamma_1(e_{b_1,V_1}) =1, \ \  \gamma_1(e_{V_1,V_2}) =0, \ \  \gamma_1(e_{V_2,b_2}) =0, \ \
 \gamma_2(e_{b_1,V_1}) =0, \ \  \gamma_2(e_{V_1,V_2}) =1, \ \  \gamma_2(e_{V_2,b_2}) =1, 
\end{equation*}
\begin{align*}
  & \eta(V_1) = \mbox{wind}(e_{b_1,V_1},e_{V_1,V_2}) +  \gamma_1(e_{b_1,V_1}) + \gamma_2(e_{V_1,V_2})+1 = 0, \\
  & \eta(V_2) = \mbox{wind}(e_{V_1,V_2},e_{V_2,b_2}) +  \gamma_1(e_{V_1,V_2}) + \gamma_2(e_{V_2,b_2}) = 1, \\
  & \eta(V_3) = 1,  \ \  \eta(V_4) = 0, 
\end{align*}
It is easy to check that that the signatures are related by this gauge function.
\end{example}
\begin{example}
\begin{figure}
\centering{\includegraphics[width=0.6\textwidth]{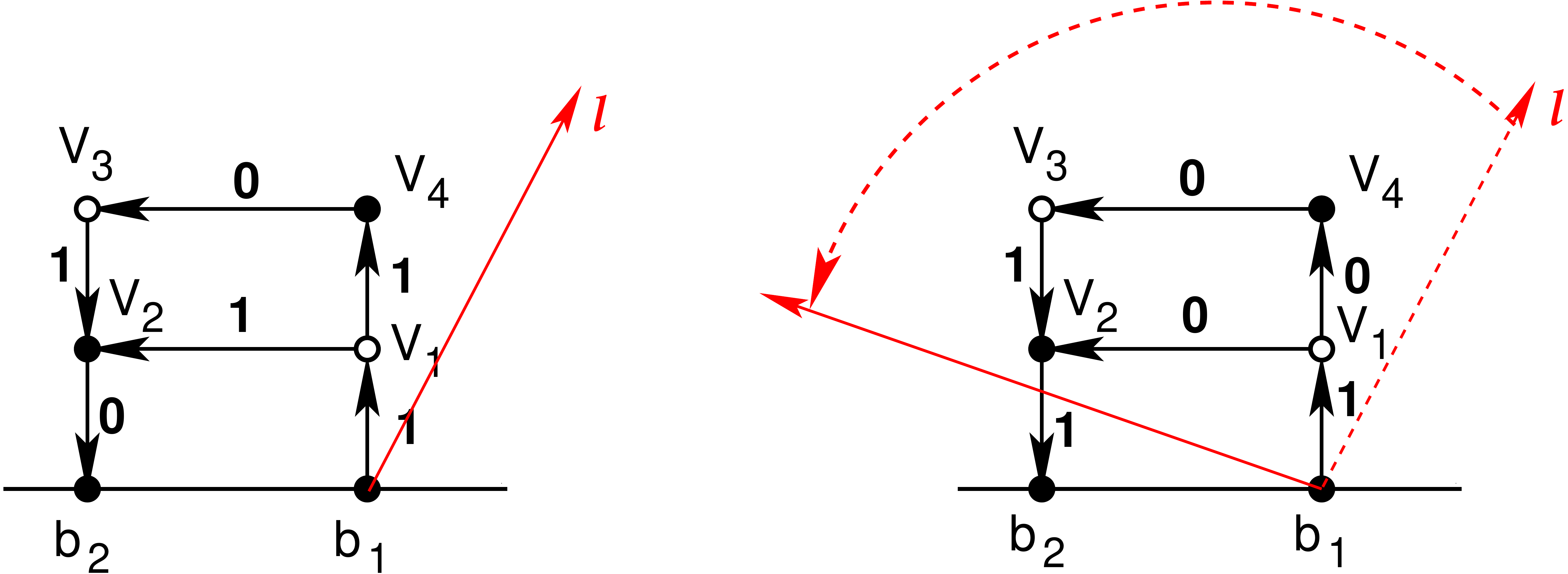}}
\caption{\small{\sl The computation of the gauge function in the case of a change of gauge ray direction. The bold numbers are the geometric edge signatures.}\label{fig:acycl_3}}
\end{figure}
Consider the graphs in Figure~\ref{fig:acycl_3}. The edge signatures are calculated using Definition~\ref{def:geometric-signature}. Let us compute the gauge function $\eta(U)$ using Formula (\ref{eq:zeta_gauge}):
\begin{align*}
  & \eta(V_1) =   \mbox{cr}(V_1)+  \mbox{par}(e_{b_1,V_1})  = 0, &&   \eta(V_2) =   \mbox{cr}(V_2)+  \mbox{par}(e_{V_2,b_2}) =1, \\
  &  \eta(V_3) =  \mbox{cr}(V_3)+  \mbox{par}(e_{V_4,V_3})  = 1  &&   \eta(V_4) =  \mbox{cr}(V_4)+  \mbox{par}(e_{V_4,V_3})  = 1  . 
\end{align*}
Again, it is easy to check that that the signatures are related by this gauge function.
\end{example}  
\begin{example}
\begin{figure}
\centering{\includegraphics[width=0.6\textwidth]{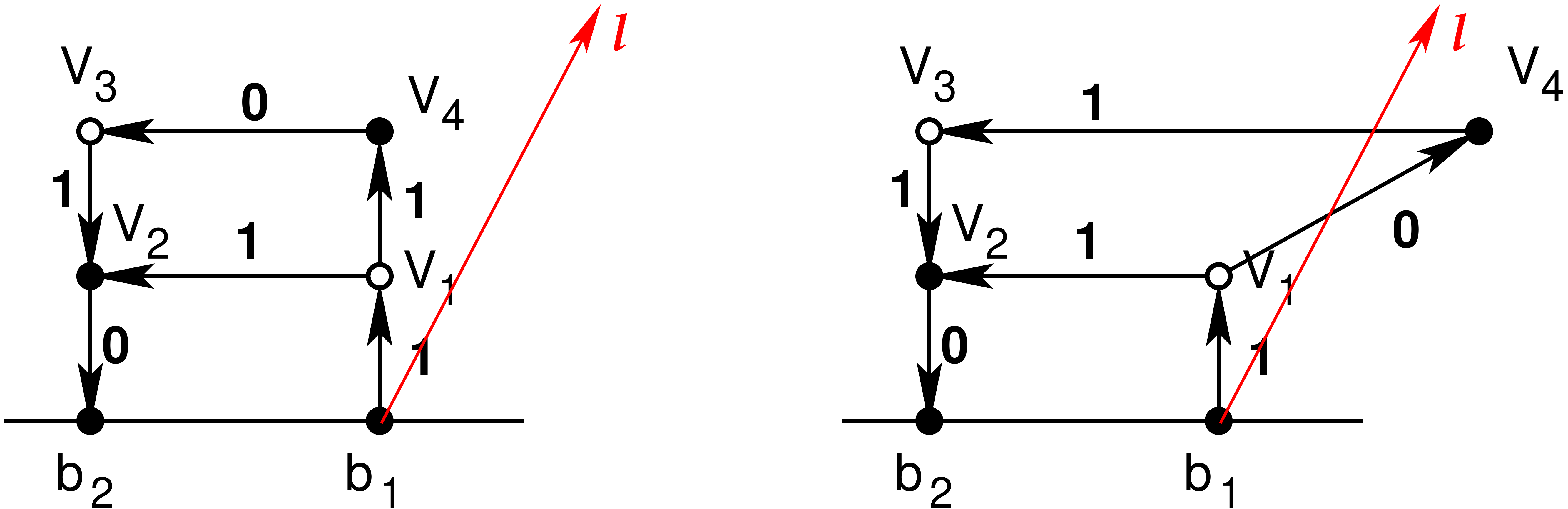}}
\caption{\small{\sl The computation of the gauge function in the case of a change of a vertex position. The bold numbers are the geometric edge signatures.}\label{fig:acycl_4}}
\end{figure}
Consider the graph in Figure~\ref{fig:acycl_4}. The edge signatures are calculated using Definition~\ref{def:geometric-signature}. Let us compute the gauge function $\eta(U)$ using Formulas (\ref{eq:zeta_gauge}), (\ref{eq:zeta_vertex_gauge_V}),
(\ref{eq:zeta_vertex_gauge_W}).
\begin{align*}
  & \eta(V_1) = 0, &&   \eta(V_2) = 0, \\
  & \eta(V_3) = \mbox{par}(e_{V_4,V_3}) = 0  &&   \eta(V_4) =  \mbox{cr}(V_4)+  \mbox{par}(e_{V_4,V_3})  = 1 .
\end{align*}
It is easy to check that that the signatures are related by this gauge function.
\end{example}

\begin{corollary}\textbf{The effect of the network transformations on half--edge vectors}\label{lem:z_orient_gauge} Let  $(\mathcal N, \mathcal O, \mathfrak l)$ and $(\mathcal N, \hat{ \mathcal O},\hat{\mathfrak l})$ be a pair of plabic networks as in Theorem~\ref{thm:z_orient_gauge}, and $z_{U,e}$,  $\hat z_{U,e}$  be the corresponding half edge vectors for the same boundary conditions. Then, at any internal vertex $U$,
  $$
 {\hat z}_{U,e} =  (-1)^{\eta(U)} z_{U,e}. 
  $$
\end{corollary}

The proof is straightforward.

\begin{corollary}
For a face  $\Omega$  of the graph $\mathcal G$ let $\epsilon(\Omega)$ be the total contribution of the geometric signature at the edges $e=(U,V)$ bounding the face $\Omega$:
\begin{equation}\label{eq:eps_tot}
\epsilon(\Omega) = \sum_{e\in\partial\Omega} \epsilon_{e}.
\end{equation}
Then $\epsilon(\Omega)\mod(2)$ is invariant with respect to the changes of orientation, gauge ray direction and of the position of the internal vertices.
\end{corollary}
\begin{proof}
In Theorem~\ref{thm:z_orient_gauge} we prove that there exists a gauge transformation of the signature for any change of the orientation, of the gauge ray direction and of the position of the internal vertices. Therefore all these changes preserve the equivalence class of the initial geometric signature, and they do not affect $\epsilon(\Omega)$. Indeed, inserting (\ref{eq:zeta_gauge}) into (\ref{eq:eps_tot}), the gauge function $\eta(U)$ contributes twice at each internal vertex $U$ in $\partial \Omega$.
\end{proof}

\section{Completeness of the geometric linear relations}\label{sec:comb}

From now on we restrict ourselves to PBDTP graphs (see Definition~\ref{def:PBDTP}). We also assume that the half-edge vectors are elements on $\R^n$, and the construction of the Grassmannian point is as in Theorem~\ref{lemma:lam_rela}. Then, we show that, if a signature has the following properties: for any collection of positive weights:
\begin{enumerate}
\item Lam's system has full rank;
\item The image at the boundary vertices is a point of the totally non-negative part of the Grassmannian,  
\end{enumerate}
then this signature is gauge equivalent to the geometric one.
Therefore, as a consequence of Theorem~\ref{lemma:lam_rela} , the boundary map associated to the present construction coincides with Postnikov's boundary measurement map.

\smallskip

The following Lemma is the key ingredient in the proof of completeness. Indeed,
by definition, for all plabic graphs the sum of equivalent signatures has the same parity along each directed path from the boundary to the boundary of the graph. In Lemma \ref{lemma:equiv_on_path}  we show that also the opposite is true under the hypothesis that the graph in a PBDTP graph. 

\begin{lemma}\textbf{Parity of equivalent signatures on simple paths and cycles}\label{lemma:equiv_on_path}
Let $\mathcal G$ be a PBDTP graph representing an irreducible positroid cell $\S$ with a fixed orientation  $\mathcal O$. Let $\epsilon^{(1)}_{U,V}$, $\epsilon^{(2)}_{U,V}$ be two signatures on $(\mathcal G,\mathcal O)$. Then, $\epsilon^{(2)}_{U,V}$ is equivalent to $\epsilon^{(1)}_{U,V}$ if and only if, for any path $P$ from a boundary source to a boundary sink and for any closed cycle $C$, the sum of the signatures of the edges on $P$, respectively on $C$ has the same parity with respect to both signatures.
\end{lemma}

\begin{proof}
In one direction the proof is trivial and follows directly from the definition of equivalence. 

Let us now suppose that two signatures $\epsilon^{(1)}$, $\epsilon^{(2)}$ have the same parity on any edge loop-erased path $\mathcal P$ from a boundary source to a boundary sink and on any closed cycle $\mathcal Q$. Then the parities also coincide on any directed path from a boundary source to a boundary sink. Let us construct a gauge transformation $\eta(U)$ transforming  $\epsilon^{(2)}$ to $\epsilon^{(1)}$ step by step using the following procedure:
\begin{enumerate}
\item Define at the initial step $\epsilon^{(3)}(e)=\epsilon^{(2)}(e)$ at all edges of the graph;  
\item Next, select a loop-erased path from a boundary source to a boundary sink, and compare the signatures at each edge starting from the edge at the boundary source. If the edge $e=(U,V)$ is not the final edge of the path, then  
\begin{enumerate}
\item If $\epsilon^{(3)}(e)=\epsilon^{(1)}(e)$, we mark $e$ and assign $\eta(V)=0$ and proceed to the next edge;
\item If $\epsilon^{(3)}(e)\ne\epsilon^{(1)}(e)$, we mark $e$, assign $\eta(V)=1$, invert $\epsilon^{(3)}$ at all edges at $V$ and proceed to the next edge; 
\end{enumerate}
If $e$ is the final edge of such path and the signatures coincide at all its other edges, then necessarily they also coincide at $e$;
\item Next for any other path (not necessary loop-erased) from a boundary source to a boundary sink, containing at least one unmarked edge, we proceed as follows:
\begin{enumerate}
\item If the edge $e=(U,V)$ is marked, we have already assigned  $\eta(U)$ and $\eta(V)$, and we proceed to the next edge. In particular, $e$ is always marked if $V$ is white and at least one of the outgoing edges at $V$ is marked.
\item If  $e=(U,V)$ is not marked and all outgoing edges at $V$ are not marked, we proceed as in Step~1.
\item Let  $e=(U,V)$ be unmarked, $V$ be black and the outgoing edge at $V$ be marked (therefore the signature $\eta(V)$ is already assigned). Since $\eta(U)$  was assigned at a previous step, we have two possibilities:
\begin{enumerate}
\item  The outgoing edge at $V$ was marked using one of the previous paths. In such case there exists a path from a boundary source to a boundary sink such that $e$ is the only unmarked edge. Since the initial signatures have the same parity and the transformation procedure preserves the parity, $\epsilon^{(1)}$ at $e$ coincide with the current $\epsilon^{(3)}$ at $e$. Therefore we mark $e$ and proceed to the next edge.  
\item The outgoing edge at $V$ was marked at some previous step while going along this path. Then these two edges at $V$ belong to a cycle, and $e$ is the only unmarked edge of this cycle. Then the parity of signatures $\epsilon^{(1)}$,  $\epsilon^{(3)}$ guarantees that they coincide at $e$, we mark $e$ and proceed to the next edge.
\end{enumerate}
\end{enumerate}
\end{enumerate}
\end{proof}

\begin{remark}
If some edge $e$ does not belong to any path from boundary to boundary, it never participates in the boundary measurement map, and we have an extra gauge freedom in assigning its signature. Therefore the PBDTP condition in Lemma~\ref{lemma:equiv_on_path} is essential.
\end{remark}

In the following Theorem we verify that for any signature not of geometric type there exists a collection of positive weights such that either the image is not totally non-negative or the linear system has not full rank. Therefore, a signature on  $\mathcal G$ guarantees the total non-negativity property for arbitrary positive weights if and only if it is geometric; moreover, in such case the image is exactly $\S$.

As a consequence the edge vectors represent a natural parametrization for the amalgamation procedure introduced in \cite{FG1} (see also \cite{AGP2,Kap,MS}) in the total non--negative setting. 

\begin{theorem}\textbf{Completeness of the geometric signatures}\label{theo:complete}

 Let $({\mathcal G},{\mathcal O}(I))$ be a perfectly oriented PBDTP graph representing the irreducible positroid cell $\S$. Let $\epsilon_{U,V}$ be a signature on ${\mathcal G}$. Then
\begin{enumerate}
\item Lam's system of relations in Definition~\ref{def:lam} possesses a unique solution $z_{U,e}$ in $\R^n$ on the signed network $({\mathcal G},{\mathcal O}(I),w_{U,V},\epsilon_{U,V})$
for almost any collection $w_{U,V}$ of real edge weights.
\item Let us assign the canonical basis vectors at the boundary sinks and let $A$ be the $k\times n$ matrix whose $r$-th row $A[r]=E_{i_r} -z_{b_{i_r}, e_{i_r}}$, where $E_{i_r}$ is the $i_r$-th vector of the canonical basis and $z_{b_{i_r}, e_{i_r}}$ is the half-edge vector at the boundary source $b_{i_r}$, $r\in[k]$. 

If, moreover, the matrix $A$ is totally non-negative for almost every collection of positive weights, then 
\begin{enumerate}
\item The signature $\epsilon_{U,V}$ is equivalent to the geometric signature;
\item The system of relations possesses a unique solution for every choice of positive weights;
\item $A$ coincides with Postnikov's boundary measurement matrix for the base $I$.
\end{enumerate}
\end{enumerate}
\end{theorem}
\begin{proof}

\textbf{Step 1.} Let $\varepsilon_e$ be the geometric signature. From Theorem~\ref{lemma:lam_rela} we know that the matrix representing Lam's system of relations for such signature has non zero--determinant for any choice of positive weights. The determinant is a polynomial in the weights, therefore it remains different from zero for the geometric signature and almost all real weights.  

Lam's system of relations is not affected  by the following transformation of signature and weights:
\begin{equation}
  \label{eq:sig_trans}  
  \begin{split}
   \epsilon_e & \rightarrow \varepsilon_e \\  
   w_e & \rightarrow \tilde w_e = (-1)^{\epsilon_e -\varepsilon_e} w_e.
\end{split}
\end{equation}
Therefore Lam's system of relations on the signed network $({\mathcal G},{\mathcal O}(I),w_{U,V},\epsilon_{U,V})$ is equivalent to Lam's system of relations on $({\mathcal G},{\mathcal O}(I),{\tilde w_e},\varepsilon_e)$. This completes the proof of the first part.

\textbf{Step 2.} For the geometric signature $\varepsilon_e$ and the collection of weights ${\tilde w_e}$, the matrix $A$ can be directly computed using Talaska formula (\ref{eq:tal_formula_source}):
\begin{equation}
    \label{eq:talaska1}
A_{ij} = (-1)^{\sigma_{ij}} \frac{\displaystyle\sum\limits_{F\in {\mathcal F}_{ij}(\mathcal G)}  {\tilde w}(F) }{\displaystyle\sum\limits_{C\in {\mathcal C}(\mathcal G)}  {\tilde w}(C)},
\end{equation}
where $\sigma_{ij}$ is the number of sources between $b_i$ and $b_j$. We observe that transformation (\ref{eq:sig_trans}) does not change Lam's equations and their solution, therefore (\ref{eq:talaska1}) can be rewritten as:
\begin{equation}
    \label{eq:talaska1.1}
A_{ij} = (-1)^{\sigma_{ij}} \frac{{\displaystyle\sum\limits_{F\in {\mathcal F}_{ij}(\mathcal G)}}\big(-1\big)^{\sum\limits_{e\in F}\epsilon_e -\varepsilon_e  }  w(F) }{{\displaystyle\sum\limits_{C\in {\mathcal C}(\mathcal G) }} \big(-1\big)^{\sum\limits_{e\in C}\epsilon_e -\varepsilon_e  } w(C)},
\end{equation}

\textbf{Step 3.} Let us now assume that the matrix $A$ is totally non-negative for for almost every collection of positive weights $w_e$, and let us prove that the total signatures $\epsilon(P)$, $\epsilon(C)$) have the same parity of the total geometric signatures  $\varepsilon(P)$, $\varepsilon(C)$), for any simple path $P$ from the boundary to the boundary and for any conservative flow $C$ with one connected component: 
\begin{equation}
  \label{eq:tot_sign1}
  \begin{split}
    &\epsilon(P) = \varepsilon(P) \mod(2), \ \ \ \ \ \ \  \epsilon(C) = \varepsilon(C)  \mod(2), \ \ \mbox{i.e.}\\
    &\sum\limits_{e\in P}\epsilon_e -\varepsilon_e = 0   \mod(2), \ \ \ \ \ \ \ \sum\limits_{e\in C}\epsilon_e -\varepsilon_e = 0   \mod(2).
    \end{split}
\end{equation}

\begin{figure}
  \centering{\includegraphics[width=0.45\textwidth]{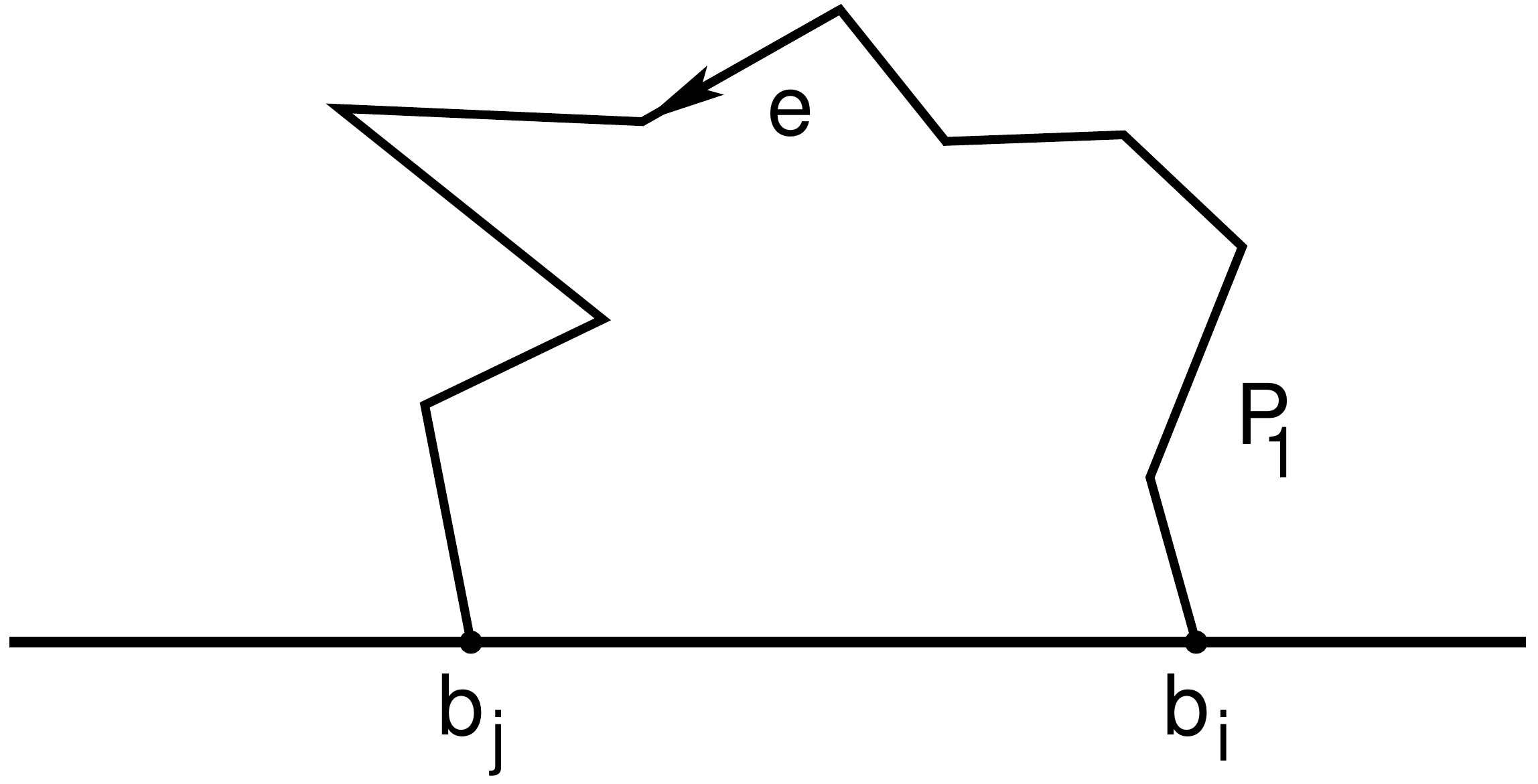}
    \includegraphics[width=0.45\textwidth]{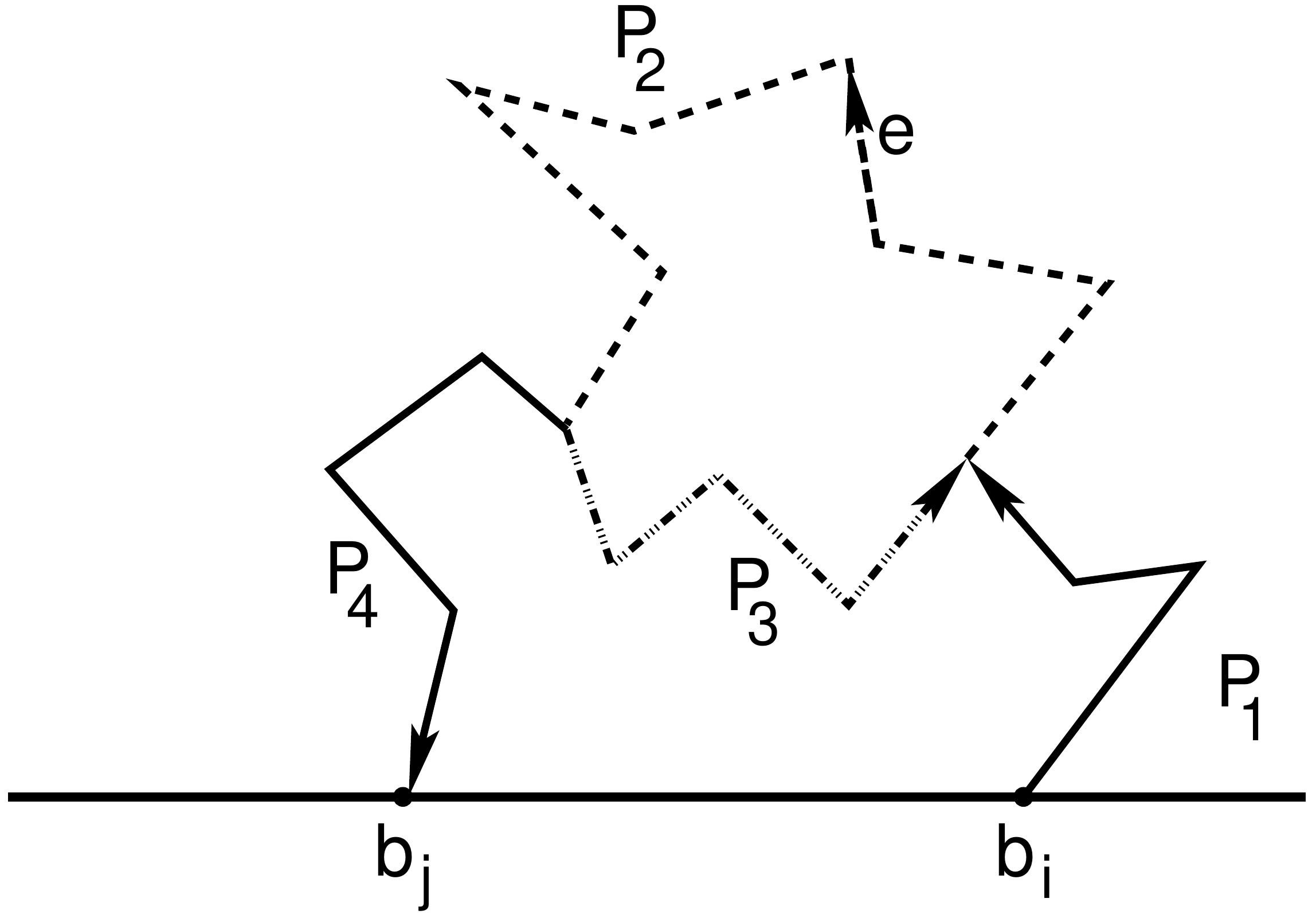}
      \caption{\small{\sl On the left: a simple path $P_1$ from the boundary source $b_i$ to the boundary sink $b_j$ passing through the edge $e$. On the right: a one-component conservative flow $C=P_2+P_3$ containing $e$ and a simple path $P=P_1+P_2+P_4$ from the boundary source $b_i$ to the boundary sink $b_j$ passing through the edge $e$. 
					}}\label{fig:path}}
\end{figure}

In the following we use notations consistent with Figure~\ref{fig:path}.

To prove the first identity let $P_1$ be a simple path from the boundary source $b_i$ to the boundary sink $b_j$. Consider all collections of positive weights $w_e$ with the following property: all the weights on the edges belonging to $P_1$ are of order 1 and all other weights are of order $\delta\ll 1$. Then equation (\ref{eq:talaska1.1}) takes the form
\begin{equation}
    \label{eq:talaska2}
A_{ij} = (-1)^{\sigma_{ij}} \frac{ \big(-1\big)^{\sum\limits_{e\in P_1}\epsilon_e -\varepsilon_e  } w(P_1) + O(\delta)} {1+ O(\delta)}.
\end{equation}
The total non-negativity condition implies that the sign of $A_{ij}$ is the same as the sign of  $A_{ij}$ for the geometric signature, therefore (\ref{eq:tot_sign1}) holds true for all simple paths from the boundary to the boundary.

To prove the second identity consider a one-component conservative flow $C=P_2+P_3$. By hypothesis there exists a simple path $P=P_1+P_2+P_4$ from the boundary source $b_i$ to the boundary sink $b_j$ passing through a given edge $e\in C$. Consider all collections of positive edge weights with the following property: all the weights on the edges belonging to $P_1,P_2,P_3,P_4$ are of order 1 and all other weights are of order $\delta\ll 1$. Then equation (\ref{eq:talaska1.1}) takes the form
\begin{equation}
    \label{eq:talaska2.1}
    A_{ij} = (-1)^{\sigma_{ij}} \frac{ \big(-1\big)^{\sum\limits_{e\in P_1+P_2+P_4}\epsilon_e -\varepsilon_e  } w(P) + O(\delta)}
    {1+ \big(-1\big)^{\sum\limits_{e\in P_2+P_3}\epsilon_e -\varepsilon_e  } w(C) +  O(\delta)}=
    (-1)^{\sigma_{ij}} \frac{w(P) + O(\delta)} {1+ \big(-1\big)^{\sum\limits_{e\in P_2+P_3}\epsilon_e -\varepsilon_e  } w(C) +  O(\delta)},
  \end{equation}
  where we use the first identity. The total non-negativity condition implies that the sign of the denominator has to be positive. Since the weight $w(C)$ can be any positive number, we conclude that (\ref{eq:tot_sign1}) holds true also in this case.
  
\textbf{Step 4.} Lemma~\ref{lemma:equiv_on_path} implies that the signatures $\epsilon_e$ and  $\varepsilon_e$ are equivalent.
\end{proof}

\section{Total signatures of the faces of graphs}\label{sec:kasteleyn}

In this Section we compute the total geometric signature of the faces of the oriented PBDTP graph $\mathcal G$.

\begin{remark}
In the following we orient the boundary of the disc clockwise.
\end{remark}

\begin{definition}\textbf{Topological winding of faces}
Let $\Omega$ be a face. Let us define a topological winding $\mbox{wind}_{\Omega}(v_i)$ at both the internal and boundary vertices $v_i$ in $\partial\Omega$. Let us label the edges bounding $\Omega$ (including those belonging to the intersection of $\Omega$ with the boundary of the disc) in increasing order counterclockwise: $e_1,e_2,\ldots,e_l,e_{l+1}=e_1$. Let $v_i$ be the vertex in $\partial\Omega$ common to the pair of consecutive edges  $(e_i,e_{i+1})$. Then
\begin{equation}
   \label{eq:twind1} 
   \mbox{wind}_{\Omega}(v_i)=\left\{
     \begin{array}{ll}
       \mbox{wind}(e_i,e_{i+1}), &  \mbox{if both of them are internal and both are}\\
                                & \mbox{either oriented clockwise or counterclockwise,} \\
       \mbox{wind}(e_i,e_{i+1}) =0,  & \mbox{if} \ \ v_i \ \ \mbox{is a boundary sink,} \\
       \mbox{wind}(e_i,e_{i+1}),  & \mbox{if} \ \ v_i \ \ \mbox{is a boundary source,}\\
       \mbox{wind}(e_i,f)+  \mbox{wind}(f, e_{i+1}), &  \mbox{if both of them are internal, one is oriented }\\
                                & \mbox{clockwise, and the other counterclockwise, and}\\
                                & f \ \ \mbox{is the third edge at} \ \ v_i,
 \end{array}  
 \right.   
\end{equation}
(see Figure~\ref{fig:twind1}).
\begin{figure}
  \centering{\includegraphics[width=0.45\textwidth]{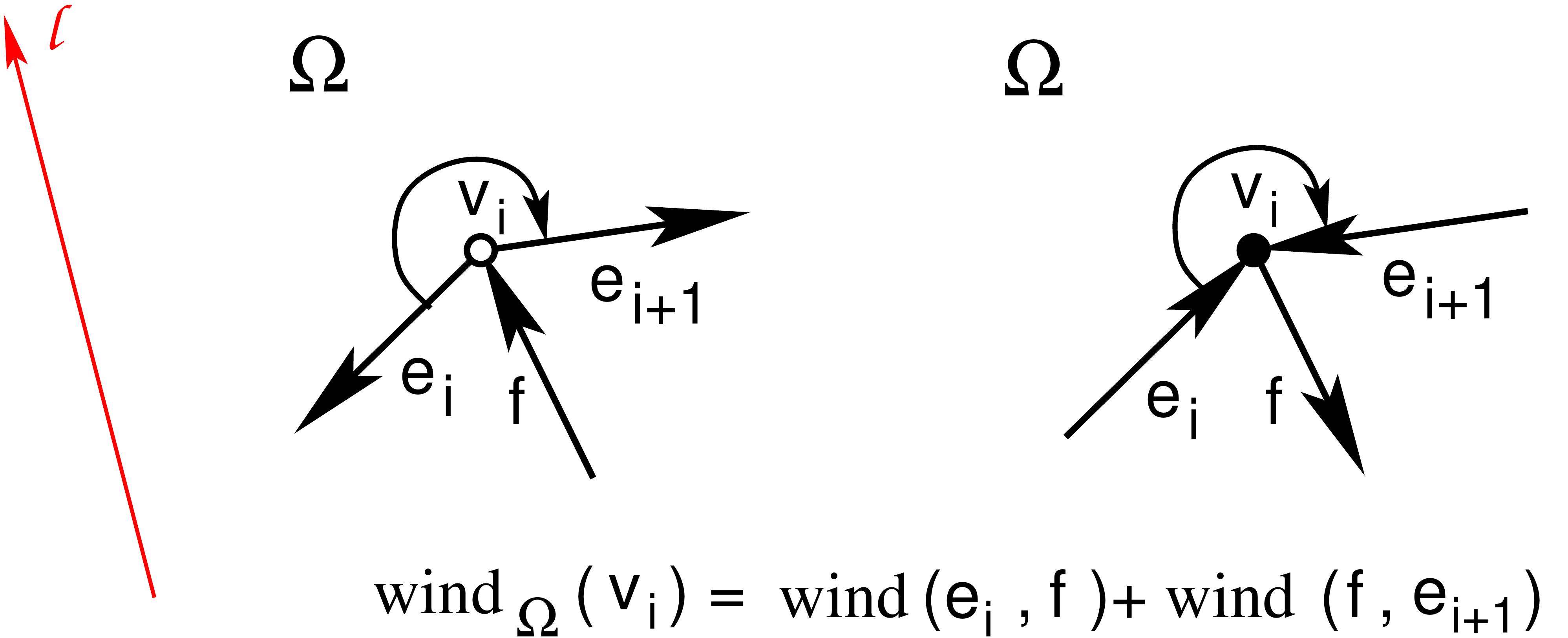}\hfill\includegraphics[width=0.45\textwidth]{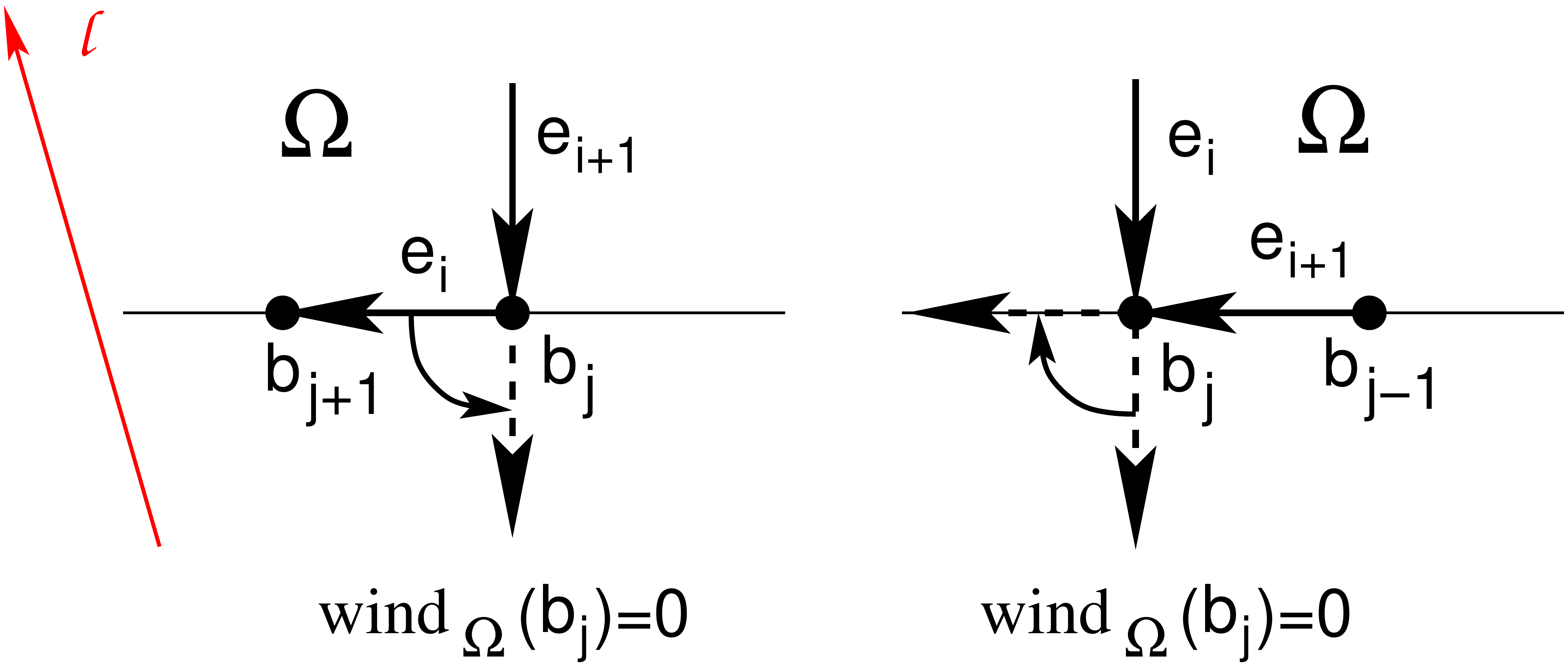}\\
    \bigskip
    \includegraphics[width=0.45\textwidth]{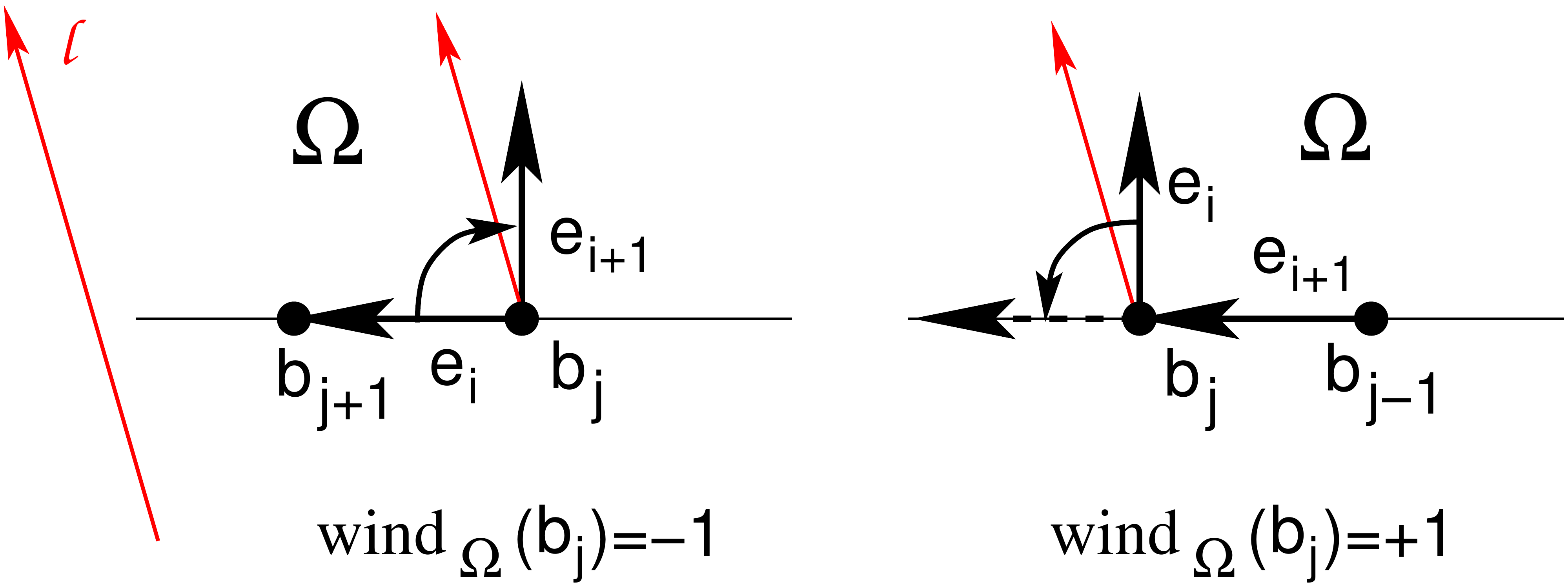}\hfill\includegraphics[width=0.45\textwidth]{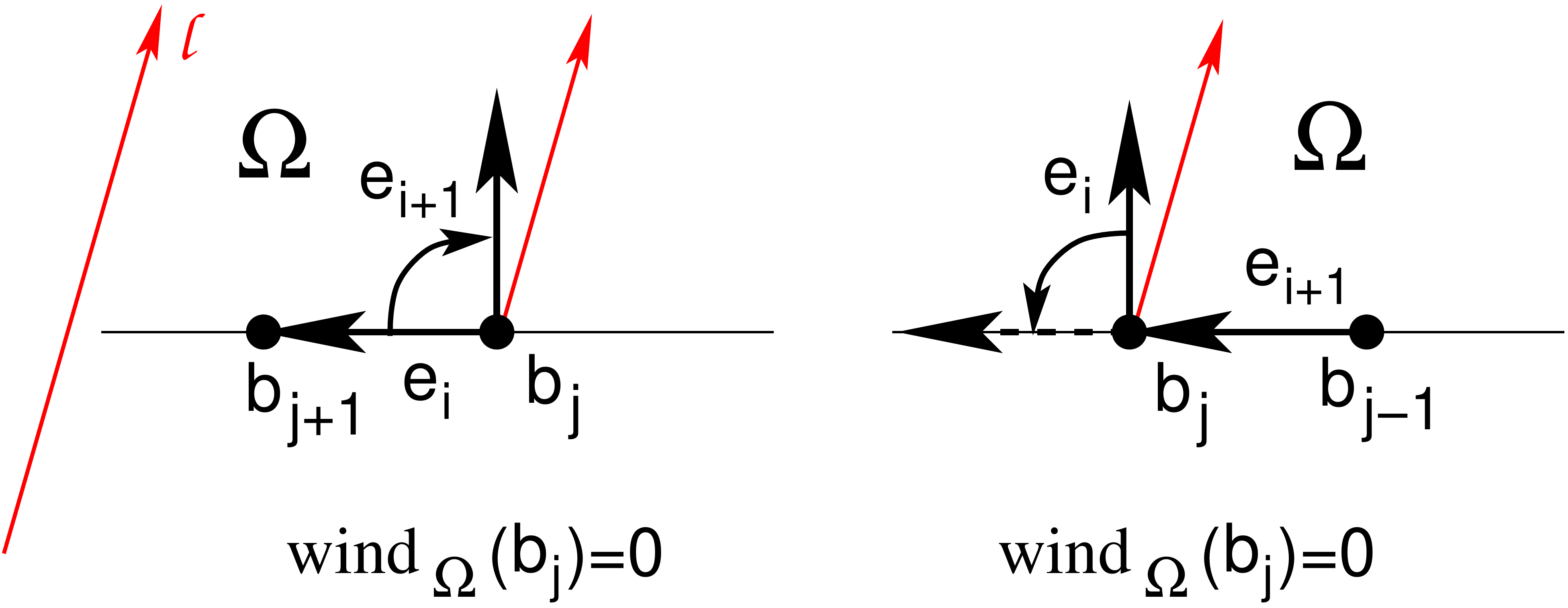}
      \caption{\small{\sl The topological winding $w_{\Omega}$ at pairs of edges at internal vertices with change of orientation (top left), at the boundary sinks (top right) and at the boundary sources (bottom).
        }}\label{fig:twind1}}
\end{figure}

Next we define the total winding of a face $\Omega$:
\begin{equation}
   \label{eq:twind2} 
\mbox{wind}(\Omega) = \sum\limits_{i=1}^{l} \mbox{wind}_{\Omega}(v_i).
\end{equation}  
If $\Omega$ is a boundary face (see Definition~\ref{def:faces}), we also define
\begin{equation}
  \label{eq:twind3}
  \begin{split}
   \mbox{wind}_{\mbox{\scriptsize{i}}}(\Omega) &= \sum\limits_{\mbox{internal vertices}\ v_i\in\partial\Omega } \mbox{wind}_{\Omega}(v_i),\\
   \mbox{wind}_{\mbox{\scriptsize{b}}}(\Omega) &=\sum\limits_{\mbox{boundary vertices}\ v_i\in\partial\Omega } \mbox{wind}_{\Omega}(v_i) 
   \end{split}
\end{equation}  
Of course, for boundary faces $\Omega$ we have
\begin{equation}
   \label{eq:twind4} 
\mbox{wind}(\Omega) = \mbox{wind}_{\mbox{\scriptsize{i}}}(\Omega) + \mbox{wind}_{\mbox{\scriptsize{b}}}(\Omega).
\end{equation}  
\end{definition}

\begin{notation}
  Let us introduce the following notations:
\begin{enumerate}
\item  $\epsilon(\Omega)$ is the face signature defined in (\ref{eq:eps_tot});
\item  $\mbox{int}(\Omega)=\sum\limits_{e\in\partial\Omega}\mbox{int}(e)$. Let us remark that for all internal faces $\mbox{int}(\Omega)=0$.
\item  $n_{\mbox{\scriptsize{white}}}(\Omega)$ denotes the total number of white vertices in $\partial \Omega$;
\item  $n_{\mbox{\scriptsize{source}}}(\Omega)$ is the number of boundary sources belonging to $\Omega$;
\item  $\mbox{cd}_{\mbox{\scriptsize{white}}}(\Omega)$ is the number of internal white vertices $v_i$ such that $e_i$, $e_{i+1}$ have opposite orientation;
\item  $\mbox{cd}_{\mbox{\scriptsize{source}}}(\Omega)$ is the number of boundary source vertices $v_i$ such that $e_i$, $e_{i+1}$ have opposite orientation.
 \end{enumerate} 
\end{notation}

We need the following simple topological Lemma:
\begin{lemma}\textbf{The total winding of a face}. The following formula holds true:
  \begin{equation}
   \label{eq:tot_wind} 
  \mbox{wind}(\Omega) =
  \left \{
    \begin{array}{ll}
      1 -  \mbox{cd}_{\mbox{\scriptsize{white}}}(\Omega) & \mbox{if} \ \ \Omega \ \ \mbox{is an internal face}\\
      1 -  \mbox{cd}_{\mbox{\scriptsize{white}}}(\Omega) - \mbox{cd}_{\mbox{\scriptsize{source}}}(\Omega)  & \mbox{if} \ \ \Omega \ \ \mbox{is a finite boundary face}\\
      -  \mbox{cd}_{\mbox{\scriptsize{white}}}(\Omega) - \mbox{cd}_{\mbox{\scriptsize{source}}}(\Omega)  & \mbox{if} \ \ \Omega \ \ \mbox {is the infinite boundary face.}
      \end{array}
  \right.
\end{equation}
\end{lemma}  
\begin{proof}
If all the edges at the boundary of $\Omega$ are oriented counterclockwise (or clockwise), the statement is well-known: one starts from $e_1$, continuously assigns angles to each $e_i$ and counts the increment of angle $\delta\phi$ after returning to $e_1$. Then $\mbox{wind}(\Omega)=\delta\phi/2\pi$. Any change of direction along the boundary adds $-\pi$ to $\delta\phi$. Since changes of directions occur in pairs for both internal and boundary faces, it is enough to count the number of changes from counterclockwise to clockwise and multiply it by 2.  
\end{proof}
  
\begin{theorem}\textbf{The total signature of a face}\label{theo:sign_face}
Let $({\mathcal G},{\mathcal O}(I))$ be a PBDTP graph in the disc representing a positroid cell $\S\subset \GTNN$, and let $\epsilon_{U,V}$ be its geometric signature. Then
\begin{equation}\label{eq:sign_face}
\epsilon(\Omega) \; = \;  
\left\{ \begin{array}{ll}
n_{\mbox{\scriptsize{white}}}(\Omega) \; + \; 1 \quad \mod 2, & \quad \mbox{if } \Omega \mbox{ is a finite face}; \\
\\
n_{\mbox{\scriptsize{white}}}(\Omega) \; + \; k \quad \mod 2, & \quad \mbox{if } \Omega \mbox{ is the infinite face}.
\end{array}
\right.
\end{equation}
\end{theorem}

\begin{remark}
It is easy to check that Theorem~\ref{theo:sign_face} is true for the examples discussed in the previous sections (see Figures \ref{fig:glick}, \ref{fig:lediag}, \ref{fig:acycl_2},  \ref{fig:acycl_3},  \ref{fig:acycl_4}).
\end{remark}  

\begin{proof}
From Definition~\ref{def:geometric-signature} we have
\begin{equation}\label{eq:face_ind1}
\epsilon(\Omega) = \mbox{wind}_{\mbox{\scriptsize{i}}(\Omega)} +n_{\mbox{\scriptsize{white}}}(\Omega)+\mbox{cd}_{\mbox{\scriptsize{white}}}(\Omega) + n_{\mbox{\scriptsize{source}}}(\Omega) +\mbox{int}(\Omega) \qquad \mod 2.
\end{equation}

Let us introduce the following notations:
\begin{enumerate}
\item $\#\,{\raisebox{-2.mm}{\includegraphics[height=8mm]{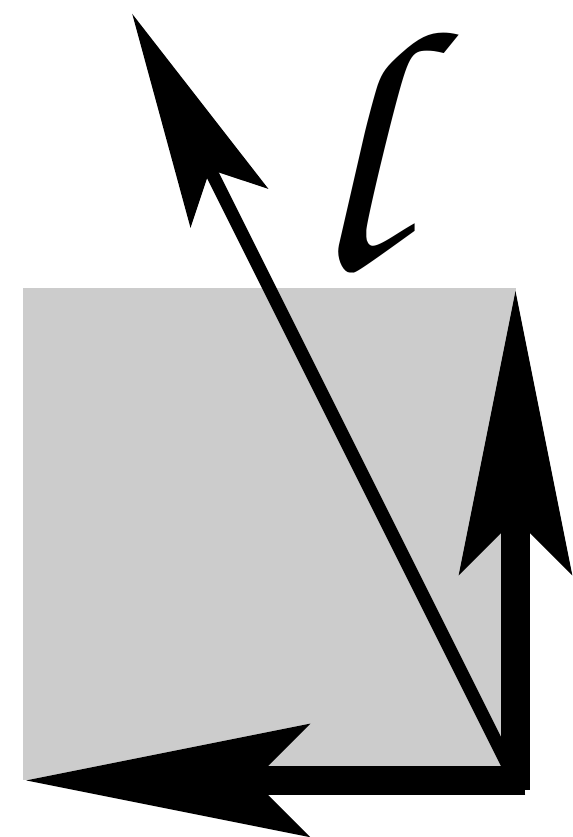}}}$ denotes the number of sources at the boundary of $\Omega$ such that $\Omega$ lies to the left of the boundary source and the gauge ray direction points inside the face; 
\item  $\#\,{\raisebox{-2.mm}{\includegraphics[height=8mm]{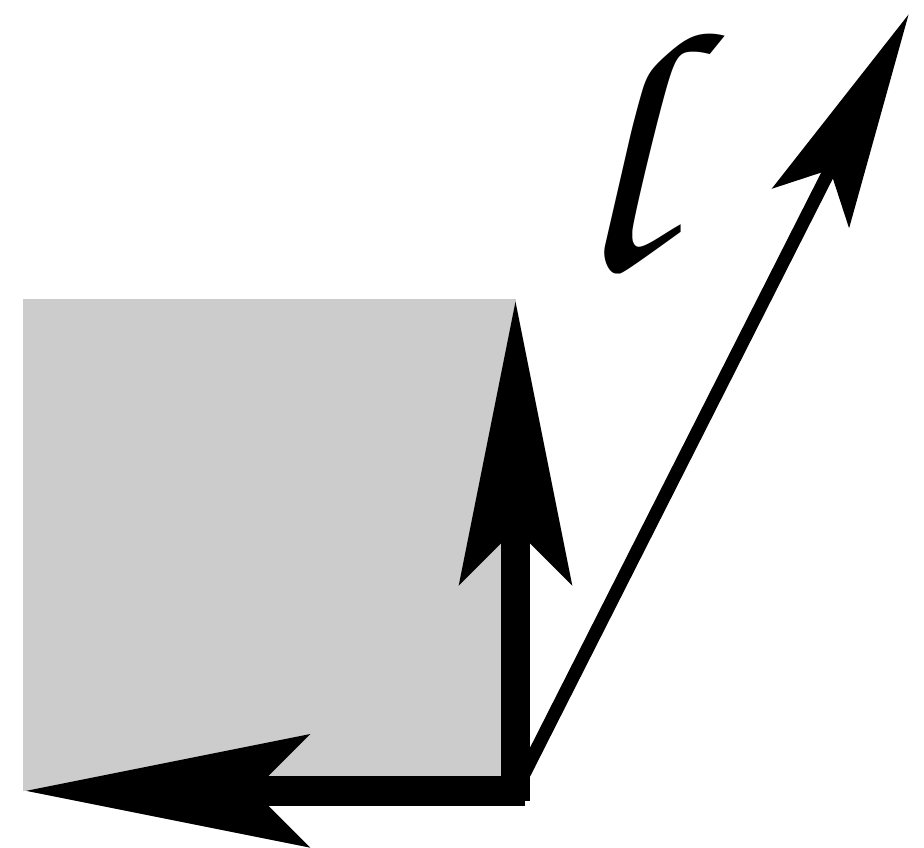}}}$ denotes the number of sources at the boundary of $\Omega$ such that $\Omega$ lies to the left of the boundary source and the gauge ray direction points outside the face; 
\item $\#\,{\raisebox{-2.mm}{\includegraphics[height=8mm]{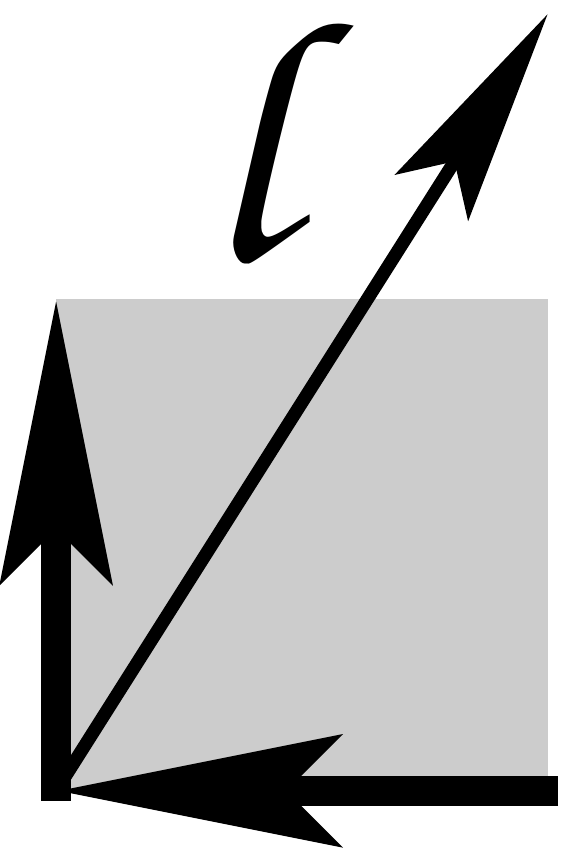}}}$ denotes the number of sources at the boundary of $\Omega$ such that $\Omega$ lies to the right of the boundary source and the gauge ray direction points inside the face; 
\item  $\#\,{\raisebox{-2.mm}{\includegraphics[height=8mm]{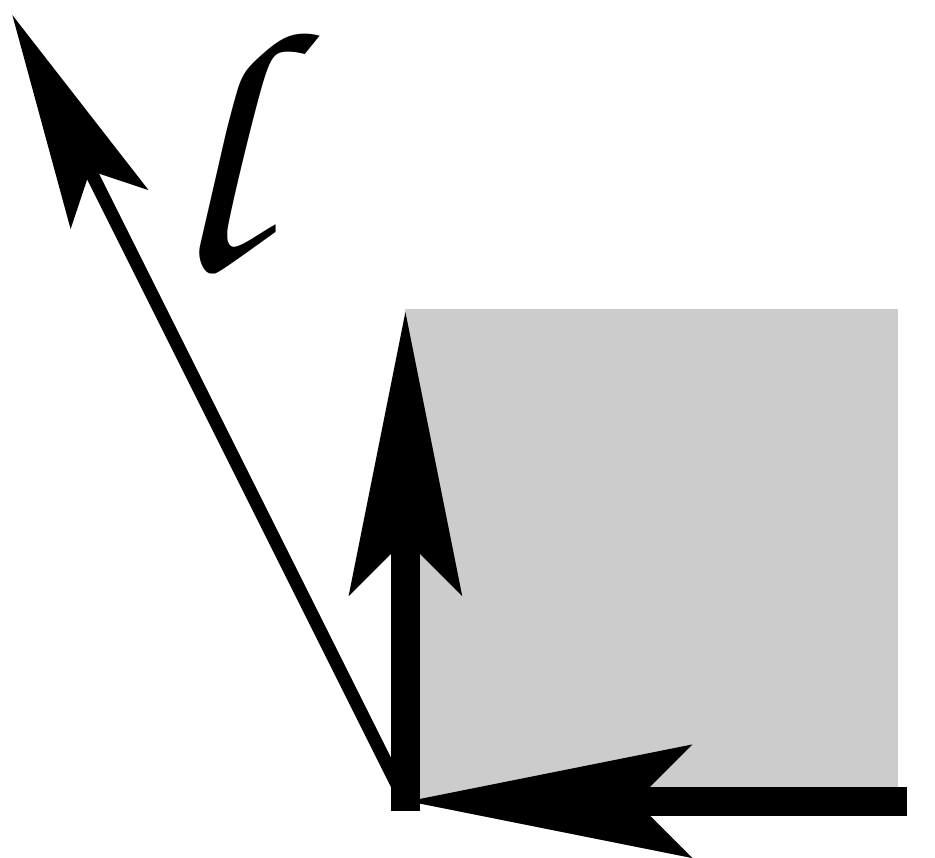}}}$ denotes the number of sources at the boundary of $\Omega$ such that $\Omega$ lies to the right of the boundary source and the gauge ray direction points outside the face;   
\end{enumerate}
Since 
\begin{equation}\label{eq:face_ind2}
\#\,{\raisebox{-2.mm}{\includegraphics[height=8mm]{vert1.pdf}}}+ \#\,{\raisebox{-2.mm}{\includegraphics[height=8mm]{vert2.pdf}}} = \mbox{cd}_{\mbox{\scriptsize{source}}}(\Omega)
\end{equation}
we have
\begin{equation}\label{eq:face_ind3}
  \mbox{wind}(\Omega) = \left\{
    \begin{array}{ll}
      1 + \mbox{cd}_{\mbox{\scriptsize{white}}}(\Omega) +\#\,{\raisebox{-2.mm}{\includegraphics[height=8mm]{vert1.pdf}}}+\#\,{\raisebox{-2.mm}{\includegraphics[height=8mm]{vert2.pdf}}} \mod 2,& \mbox{if} \ \ \Omega \ \ \mbox{is a finite face },\\
      \mbox{cd}_{\mbox{\scriptsize{white}}}(\Omega) +\#\,{\raisebox{-2.mm}{\includegraphics[height=8mm]{vert1.pdf}}}+\#\,{\raisebox{-2.mm}{\includegraphics[height=8mm]{vert2.pdf}}} \mod 2,& \mbox{if} \ \ \Omega \ \ \mbox{is the infinite face}.
\end{array}    
\right.   
\end{equation}
Since
\begin{equation}\label{eq:face_ind4}
\mbox{wind}_{\mbox{\scriptsize{b}}}(\Omega) = \#\,{\raisebox{-2.mm}{\includegraphics[height=8mm]{vert1.pdf}}}+\#\,{\raisebox{-2.mm}{\includegraphics[height=8mm]{vert3.pdf}}}\bigskip \mod 2,
\end{equation}
we have
\begin{equation}\label{eq:face_ind5}
  \mbox{wind}_{\mbox{\scriptsize{i}}}(\Omega) = \left\{
    \begin{array}{ll}
      1+  \mbox{cd}_{\mbox{\scriptsize{white}}}(\Omega) + \#\,{\raisebox{-2.mm}{\includegraphics[height=8mm]{vert2.pdf}}}+\#\,{\raisebox{-2.mm}{\includegraphics[height=8mm]{vert3.pdf}}} \mod 2, & \mbox{if} \ \ \Omega \ \ \mbox{is a finite face }\\
      \mbox{cd}_{\mbox{\scriptsize{white}}}(\Omega) + \#\,{\raisebox{-2.mm}{\includegraphics[height=8mm]{vert2.pdf}}}+\#\,{\raisebox{-2.mm}{\includegraphics[height=8mm]{vert3.pdf}}} \mod 2,  & \mbox{if} \ \ \Omega \ \ \mbox{is the infinite face}.
     \end{array}    
\right. 
\end{equation}
\begin{equation}\label{eq:face_ind6}
   \mbox{int}(\Omega) = \left\{
     \begin{array}{ll}
       0 \mod 2,  & \mbox{if} \ \ \Omega \ \ \mbox{is an internal }\\
       \#\,{\raisebox{-2.mm}{\includegraphics[height=8mm]{vert1.pdf}}}+\#\,{\raisebox{-2.mm}{\includegraphics[height=8mm]{vert4.pdf}}}\bigskip \mod 2,\ & \mbox{if} \ \ \Omega \ \ \mbox{is a finite boundary face }\\
       k + \#\,{\raisebox{-2.mm}{\includegraphics[height=8mm]{vert1.pdf}}}+\#\,{\raisebox{-2.mm}{\includegraphics[height=8mm]{vert4.pdf}}}\bigskip \mod 2,\ & \mbox{if} \ \ \Omega \ \ \mbox{is the infinite face}.
\end{array}    
\right. 
\end{equation}
Therefore 
\begin{equation}\label{eq:face_ind7}
  n_{\mbox{\scriptsize{source}}}(\Omega)+ \mbox{int}(\Omega) = \left\{
     \begin{array}{ll}
       \#\,{\raisebox{-2.mm}{\includegraphics[height=8mm]{vert2.pdf}}}+\#\,{\raisebox{-2.mm}{\includegraphics[height=8mm]{vert3.pdf}}}\bigskip \mod 2,\ & \mbox{if} \ \ \Omega \ \ \mbox{is a finite boundary face }\\
       k + \#\,{\raisebox{-2.mm}{\includegraphics[height=8mm]{vert2.pdf}}}+\#\,{\raisebox{-2.mm}{\includegraphics[height=8mm]{vert3.pdf}}}\bigskip \mod 2,\ & \mbox{if} \ \ \Omega \ \ \mbox{is the infinite face}.
\end{array}    
\right. 
\end{equation}

Inserting (\ref{eq:face_ind7}) into (\ref{eq:face_ind5}) we get
 \begin{equation}\label{eq:face_ind8}
   \mbox{wind}_{\mbox{\scriptsize{i}}}(\Omega) = \left\{
     \begin{array}{ll}
       1+  \mbox{cd}_{\mbox{\scriptsize{white}}}(\Omega) + n_{\mbox{\scriptsize{source}}}(\Omega) +\mbox{int}(\Omega) \bigskip \mod 2  & \mbox{if} \ \ \Omega \ \ \mbox{is a finite face }\\
        k+  \mbox{cd}_{\mbox{\scriptsize{white}}}(\Omega) + n_{\mbox{\scriptsize{source}}}(\Omega) +\mbox{int}(\Omega) \bigskip \mod 2 & \mbox{if} \ \ \Omega \ \ \mbox{is the infinite face}.
\end{array}    
\right. 
\end{equation}
Inserting (\ref{eq:face_ind8}) into (\ref{eq:face_ind1}) we finally get
\begin{equation}\label{eq:face_ind9}
  \epsilon(\Omega) =
  \left\{
     \begin{array}{l}
       1+  2\, \mbox{cd}_{\mbox{\scriptsize{white}}}(\Omega) + 2\, n_{\mbox{\scriptsize{source}}}(\Omega) + 2\, \mbox{int}(\Omega) + n_{\mbox{\scriptsize{white}}}(\Omega) =1 + n_{\mbox{\scriptsize{white}}}(\Omega) \mod 2,\\
       k+  2\, \mbox{cd}_{\mbox{\scriptsize{white}}}(\Omega) + 2\, n_{\mbox{\scriptsize{source}}}(\Omega) + 2\, \mbox{int}(\Omega) + n_{\mbox{\scriptsize{white}}}(\Omega) =k + n_{\mbox{\scriptsize{white}}}(\Omega) \mod 2,
 \end{array}    
\right.       
\end{equation}
respectively for a finite face $\Omega$  and for the infinite one.
\end{proof}

\begin{remark}\textbf{Connection between $\epsilon(\Omega)$ and  Kasteleyn theorem}
In  \cite{A3}, a Kasteleyn's sign matrix is associated to each reduced bipartite graph in the disc, with entries equal to  $(-1)^{\epsilon(e_{ij})}$ if the edge $e_{ij}$ joins the black vertex $b_i$ and the white vertex $w_j$, and $0$ otherwise; and it is proven that its maximal minors count the dimer configurations on the graph sharing the same boundary conditions. Therefore, the geometric signature realizes Speyer's variant \cite{Sp} of the  classical Kasteleyn's theorem \cite{Kas1,Kas2} for such class of graphs.
\end{remark}

\begin{conjecture}
In the present setting of PBDTP graphs, we can define a sign matrix using the geometric signature and we claim that it explicitly realizes the variant \cite{Sp} of classical Kasteleyn theorem for planar bicoloured graphs.
\end{conjecture}

\subsection{Effect of moves and reductions on signatures}\label{sec:moves_reduc}

In \cite{Pos} Postnikov classified the set of local transformations - moves and reductions - on planar bicoloured networks in the disc which leave invariant the boundary measurement map. Two networks in the disc connected by a sequence of such moves and reductions represent the same point in $\GTNN$. Since we are interested in the class of PBDTP graphs, we shall describe the effect of
\begin{enumerate}
\item (M1) the square move (Figure \ref{fig:squaremove}),
\item (M2) the unicoloured edge contraction/uncontraction (Figure \ref{fig:flipmove}),
\item (M3) the middle vertex insertion/removal (Figure \ref{fig:middle}),
\item (R1) the parallel edge reduction (Figure \ref{fig:parall_red_poles}),
\end{enumerate}
on the geometric signature,  using the characterization of geometric signatures obtained in the previous Sections.

Let $({\mathcal N}, \mathcal O, \mathfrak l)$ be the initial oriented network, and $({\tilde {\mathcal N}}, {\tilde{\mathcal O}}, \mathfrak l)$ be the oriented network obtained from it by applying one move (M1)--(M3) or one reduction (R1).

\smallskip

{\bf (M1) The square move}
If a network has a square formed by four trivalent vertices
whose colours alternate as one goes around the square, then one can switch the colours of these
four vertices and transform the weights of adjacent faces as shown in Figure \ref{fig:squaremove}.
The relation between the face weights before and after the square move is \cite{Pos}
${\tilde f}_5 = (f_5)^{-1}$, ${\tilde f}_1 = f_1/(1+1/f_5)$, ${\tilde f}_2 = f_2 (1+f_5)$, ${\tilde f}_3 = f_3 (1+f_5)$, ${\tilde f}_4 = f_4/(1+1/f_5)$,
so that the relation between the edge weights with the orientation in Figure \ref{fig:squaremove} is
${\tilde \alpha}_1 = \frac{\alpha_3\alpha_4}{{\tilde\alpha}_2}$, 
${\tilde \alpha}_2 = \alpha_2 + \alpha_1\alpha_3\alpha_4$, ${\tilde \alpha}_3 = \alpha_2\alpha_3/{\tilde \alpha}_2$,  ${\tilde \alpha}_4 = \alpha_1\alpha_3/{\tilde \alpha}_2$.

\begin{figure}
\centering{\includegraphics[width=0.45\textwidth]{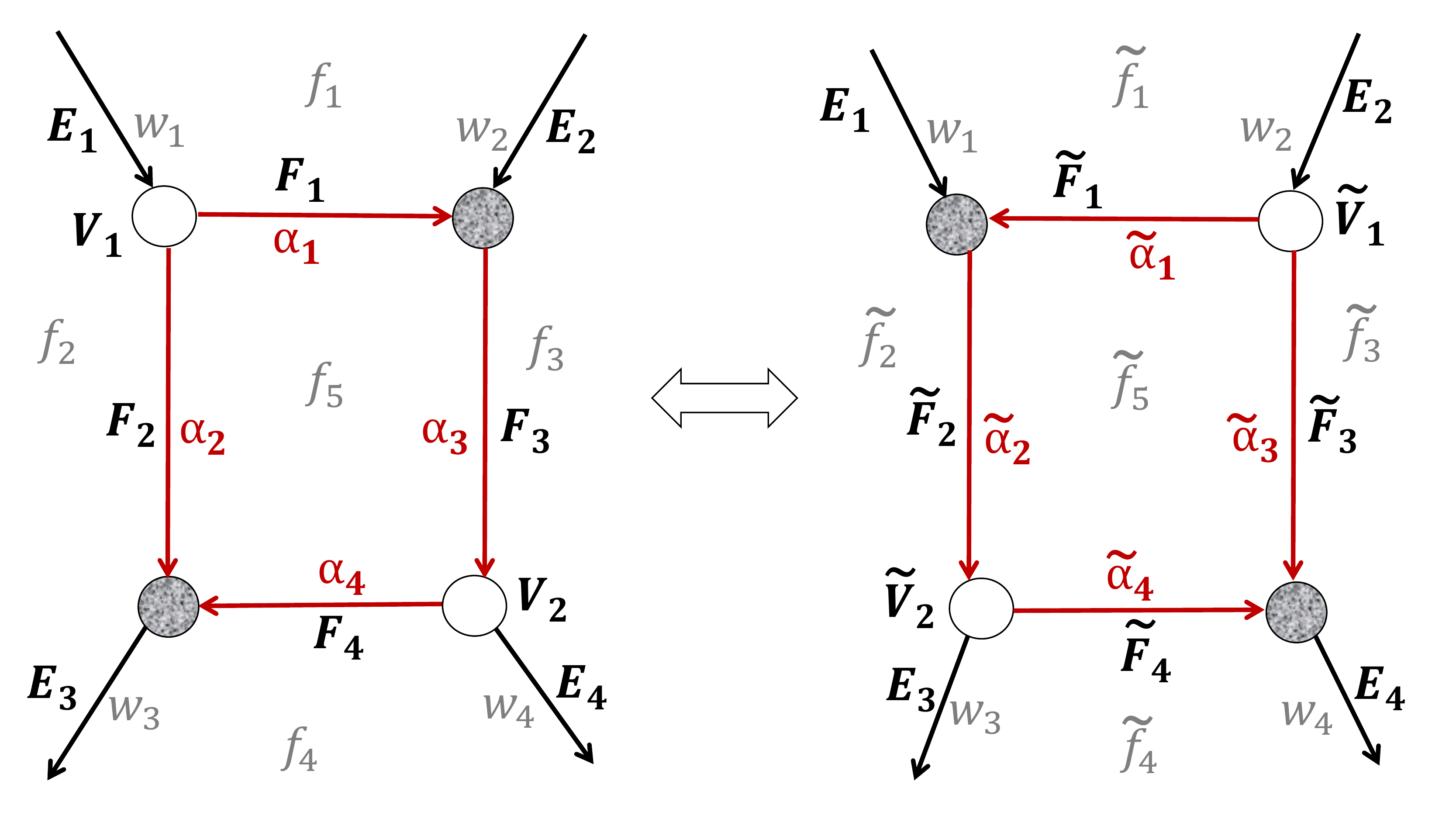}}
\caption{\small{\sl The effect of the square move.}\label{fig:squaremove}}
\end{figure}

\begin{lemma} If a signature $\epsilon(e)$ is geometric for the initial network  $({\mathcal N}, \mathcal O, \mathfrak l)$, then it remains geometric after the square move.
\end{lemma}
The statement follows immediately from Theorems~\ref{theo:sign_face} and \ref{theo:complete}.

\smallskip

{\bf (M2) The unicoloured edge contraction/uncontraction}
The unicoloured edge contraction/un\-con\-traction consists in the elimination/addition of an internal vertex of equal colour and of a unit edge, and it leaves invariant the face weights and the boundary measurement map \cite{Pos}. 
The contraction/uncontraction of an unicoloured internal edge combined with the trivalency condition is equivalent to a flip of the unicoloured vertices involved in the move (see Figure \ref{fig:flipmove}). We consider only pure flip moves, i.e.
all vertices keep the same positions before and after the move. 

\begin{figure}
  \centering{\includegraphics[width=0.6\textwidth]{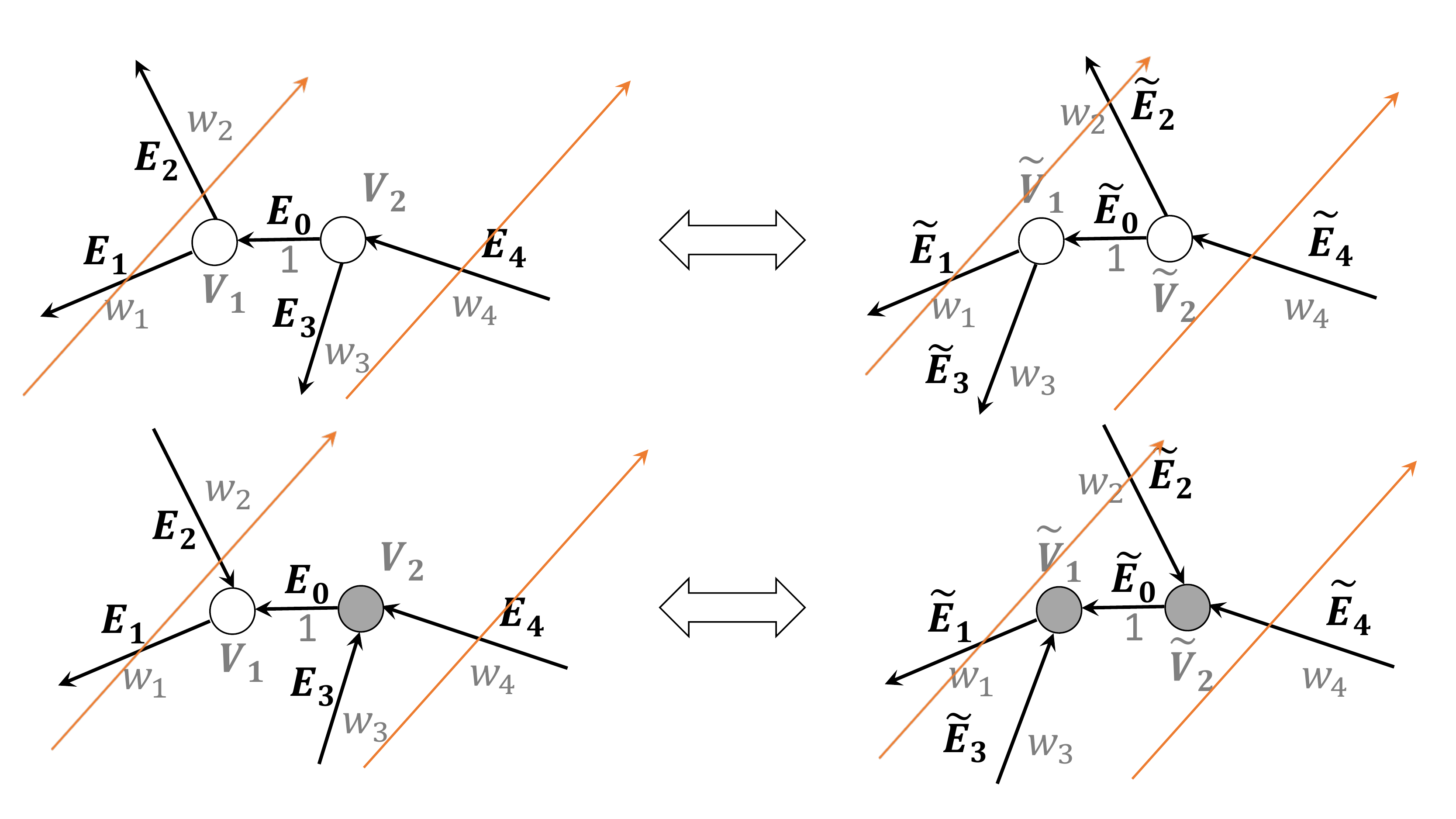}}
  \caption{\small{\sl The insertion/removal of an unicoloured internal vertex is equivalent to a flip move of the unicoloured vertices.}\label{fig:flipmove}}
\end{figure}

\begin{lemma}
Let the initial signature be geometric, and such that the signature of the internal unicoloured edge is equal to 1 if both vertices are white, respectively to 0 if both vertices are black. Then, the signature remains geometric after the transformation.
\end{lemma}
Again, the statement follows immediately from Theorems~\ref{theo:sign_face} and \ref{theo:complete}.

We remark that the above condition on the initial signature is not restrictive, since we can always apply a gauge transformation at one vertex to fulfill it.  

\smallskip

{\bf (M3) The middle edge insertion/removal}
The middle edge insertion/removal consists in the addition/elimination of bivalent vertices (see Figure \ref{fig:middle}) without changing the face configuration.
\begin{figure}
  \centering
	{\includegraphics[width=0.49\textwidth]{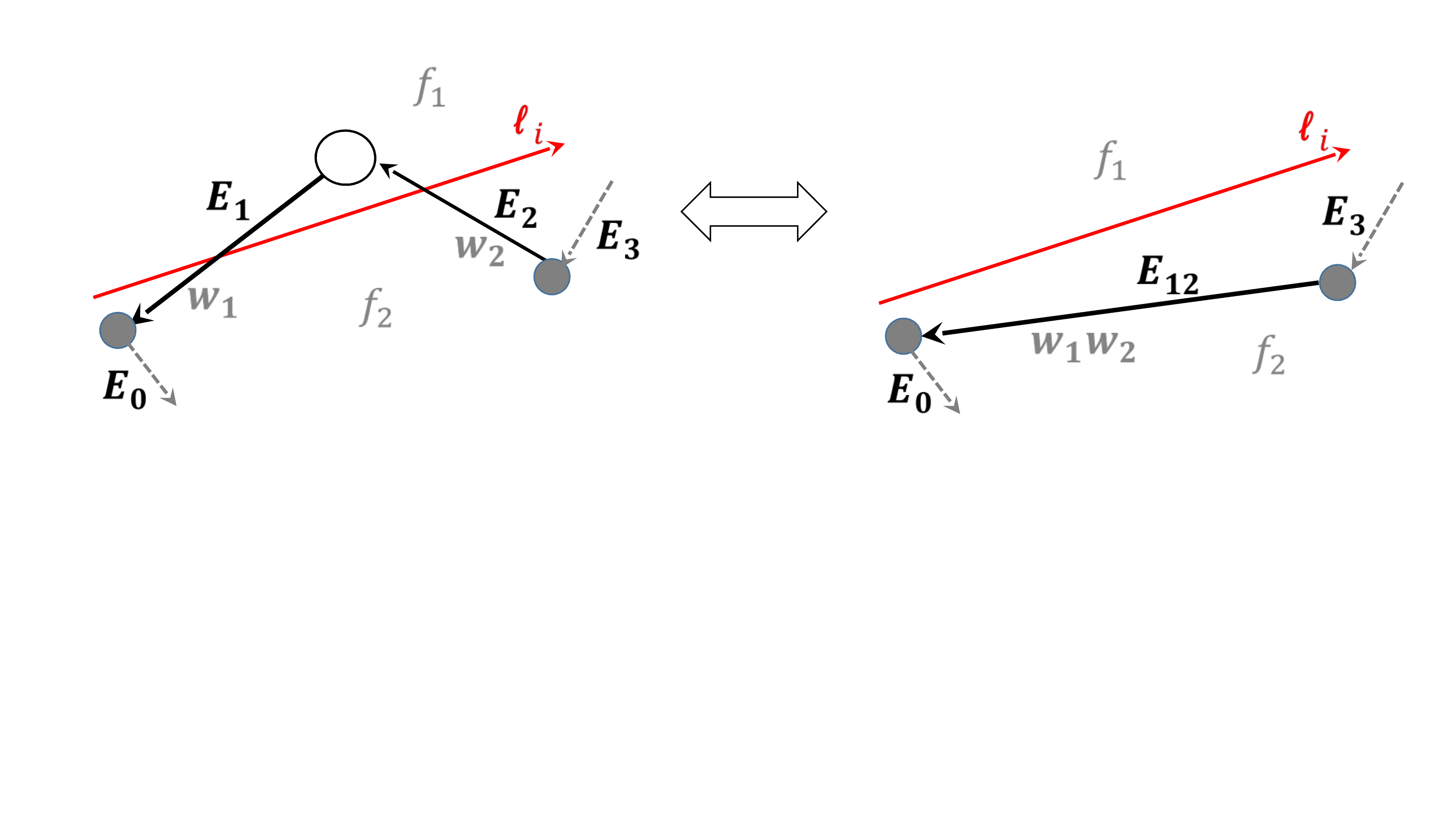}}
	\vspace{-2.1 truecm}
  \caption{\small{\sl The middle edge insertion/removal at a white vertex.}\label{fig:middle}}
\end{figure}

\begin{lemma}
Let the initial signature be geometric, and define the transformed one to be equal to the initial one at all edges common to both networks, and to satisfy the following identity at $e_1$, $e_2$, $e_{12}$: 
\begin{equation}
  \epsilon(e_1)+\epsilon(e_2)+\epsilon(e_{12}) =\left\{\begin{array}{ll} 0  & \mbox{if the additional vertex is black;} \\
                   1  & \mbox{if the additional vertex is white.} \end{array} \right.
 \end{equation}
 Then the transformed signature is geometric as well.            
 \end{lemma}
Again, the statement follows immediately from Theorems~\ref{theo:sign_face} and \ref{theo:complete}.

\smallskip

{\bf (R1) The parallel edge reduction}
The parallel edge reduction consists in the removal of two trivalent vertices of different colour connected by a pair of parallel edges (see Figure \ref{fig:parall_red_poles}). If the parallel edge separates two distinct faces, the relation of the face weights before and after the reduction is 
${\tilde f}_1 = \frac{f_1}{1+(f_0)^{-1}}$, ${\tilde f}_2 = f_2 (1+f_0)$,
otherwise ${\tilde f}_1 = {\tilde f}_2 = f_1 f_0$ \cite{Pos}. 
\begin{figure}
  \centering{\includegraphics[width=0.55\textwidth]{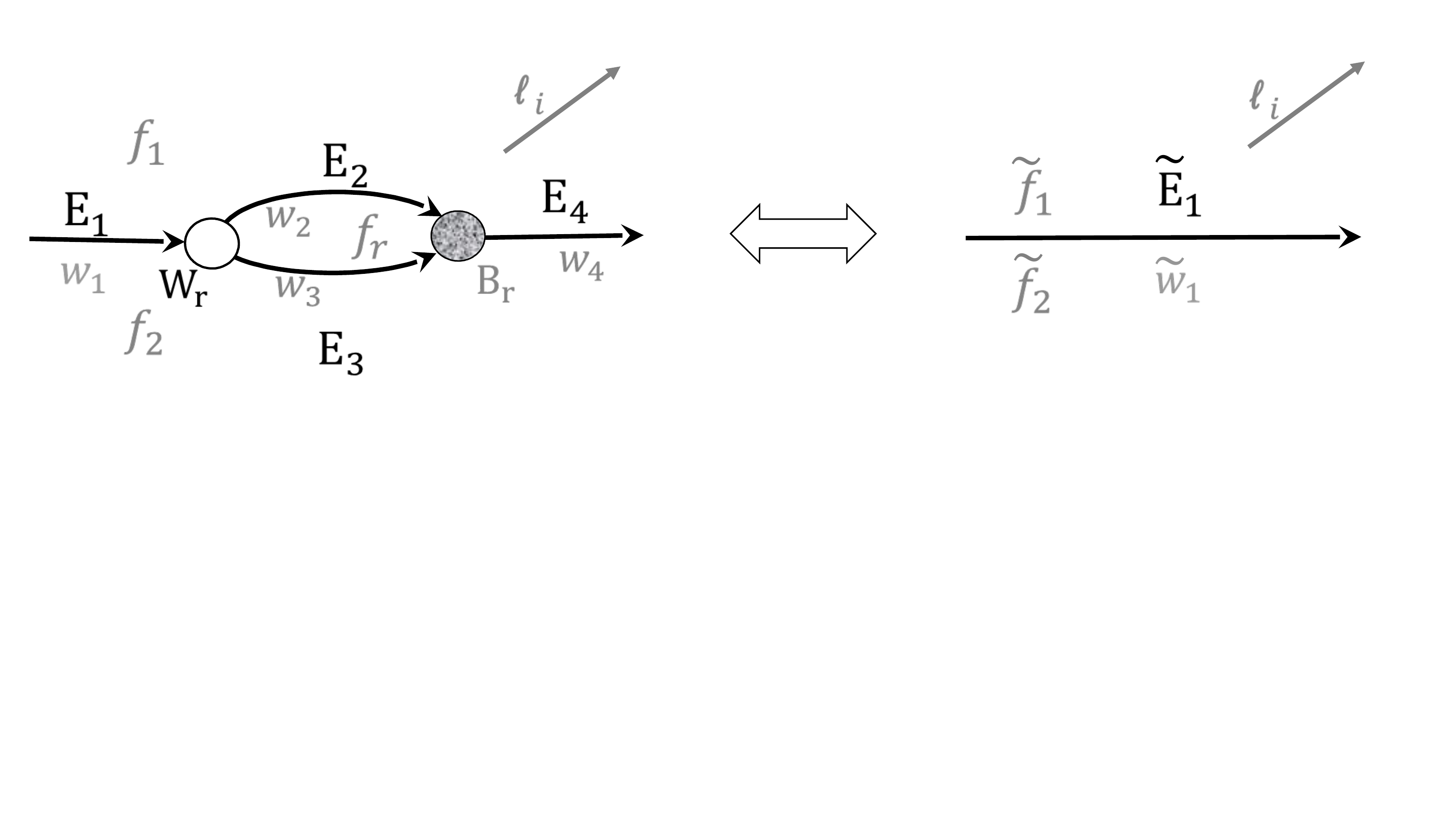}}
	\vspace{-3. truecm}
  \caption{\small{\sl The parallel edge reduction.}\label{fig:parall_red_poles}}
\end{figure}
If the original signature is geometric, necessarily it takes the same values at both parallel edges $\epsilon(e_2)=\epsilon(e_3)$. 
\begin{lemma}
Let the initial signature be geometric, and let us impose that the transformed signature remains unchanged at all edges common to both networks, whereas at $e_1$, $e_2$, $e_3$, $e_4$, $\tilde e_1$ it fulfils
\begin{equation}
\epsilon(e_1)+\epsilon(e_2)+\epsilon(e_4)+\epsilon(\tilde e_1) = 1.
\end{equation}  
Then the transformed signature is geometric as well.            
 \end{lemma}
Again, the statement follows immediately from Theorems~\ref{theo:sign_face} and \ref{theo:complete}.

\subsection{Amalgamation of positroid cells and geometric signatures}\label{sec:amalg}

In \cite{FG1} Fock and Goncharov introduced the amalgamation of cluster varieties, which has turned out to be relevant both for the construction of integrable systems on Poisson cluster varieties \cite{KG} and for the computation of the scattering amplitudes on on-shell diagrams in the $N=4$ SYM theory \cite{AGP2}. The amalgamation of positroid varieties admits a very simple representation in terms of simple operations on the corresponding plabic graphs. In our setting the vertices are those of the graph and the frozen ones are those at the boundary. Since the planarity property is essential, amalgamation is represented by compositions of the following elementary operations:
\begin{enumerate}
\item The disjoint union of a pair of planar graphs (see Figure~\ref{fig:amalg1}) which corresponds to the direct sum of the corresponding Grassmannians via a map $Gr^{\mbox{\tiny{TNN}}}(k_1,n_1)\times Gr^{\mbox{\tiny{TNN}}}(k_2,n_2)\rightarrow Gr^{\mbox{\tiny{TNN}}}(k_1+k_2,n_1+n_2)$;
\item The defrosting of a pair of consecutive boundary vertices (see Figure~\ref{fig:amalg2})  which corresponds to a projection map $Gr^{\mbox{\tiny{TNN}}}(k,n)\rightarrow Gr^{\mbox{\tiny{TNN}}}(k-1,n-2)$. 
\end{enumerate}

\begin{figure}
  \centering{\includegraphics[width=0.45\textwidth]{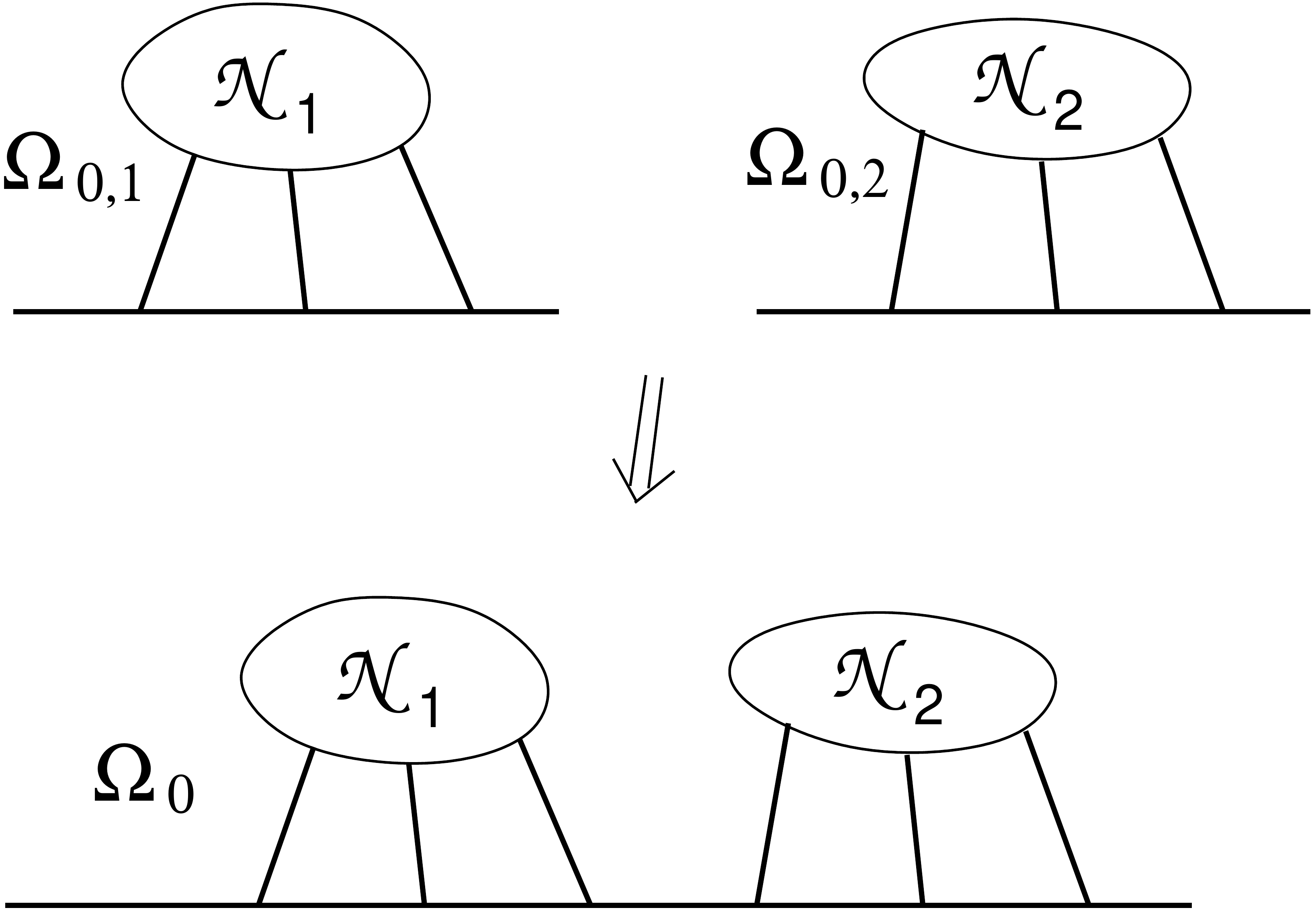}
	\hfill
	\includegraphics[width=0.45\textwidth]{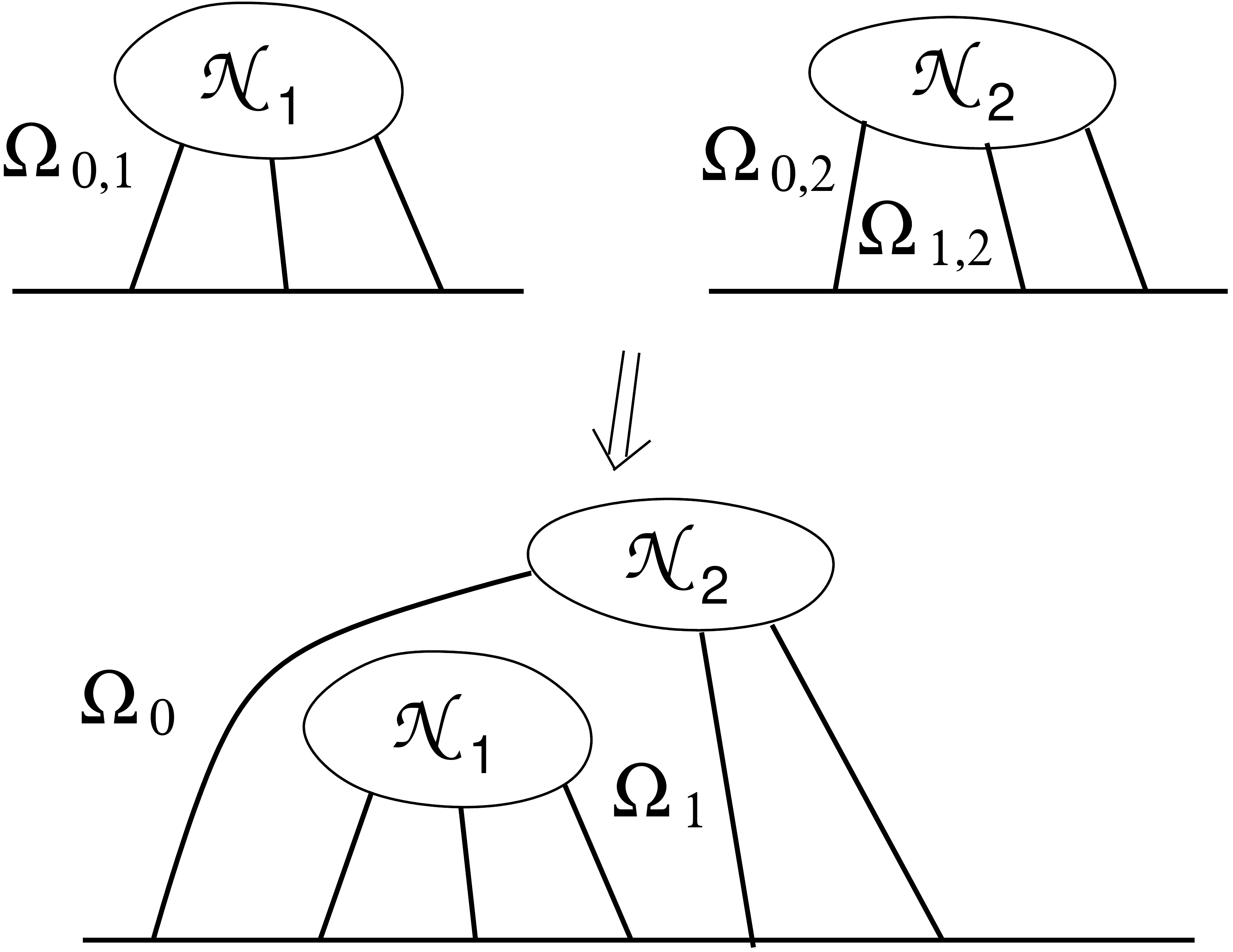}}
  \caption{\small{\sl The two possible ways of constructing disjoint unions of two given networks in the disc preserving the planarity property.}\label{fig:amalg1}}
\end{figure}

Let the initial points be $[A_i]\in Gr^{\mbox{\tiny{TNN}}}(k_i,n_i)$, $i=1,2$. Then we have exactly two ways to perform the disjoint union preserving the total non-negativity property:
\begin{enumerate}
\item All boundary vertices of one network precede all boundary vertices of the second one (see Figure~\ref{fig:amalg1} [left]).  In this case the resulting infinite face $\Omega_0$ is the union of the infinite faces $\Omega_{0,1}$, $\Omega_{0,2}$ of the initial networks and all finite faces are not modified;
\item All boundary vertices of one network are located between two consecutive boundary vertices of the other one (see Figure~\ref{fig:amalg1} [right]). Let us denote $\Omega_{1,2}$,  $\Omega_{1}$  respectively the finite face containing this pair of boundary vertices in the initial and final networks. In this case the infinite face $\Omega_0$ coincides with $\Omega_{0,2}$, the infinite face of the ``external'' network (${\mathcal N}_2$ in Figure~\ref{fig:amalg1} [right]), whereas $\Omega_1$ is built out of  $\Omega_{1,2}$ and of $\Omega_{0,1}$, the infinite face of the  ``internal'' network  (${\mathcal N}_1$ in Figure~\ref{fig:amalg1} [right]). All other faces are not modified.
\end{enumerate}

\begin{figure}
  \centering{\includegraphics[width=0.47\textwidth]{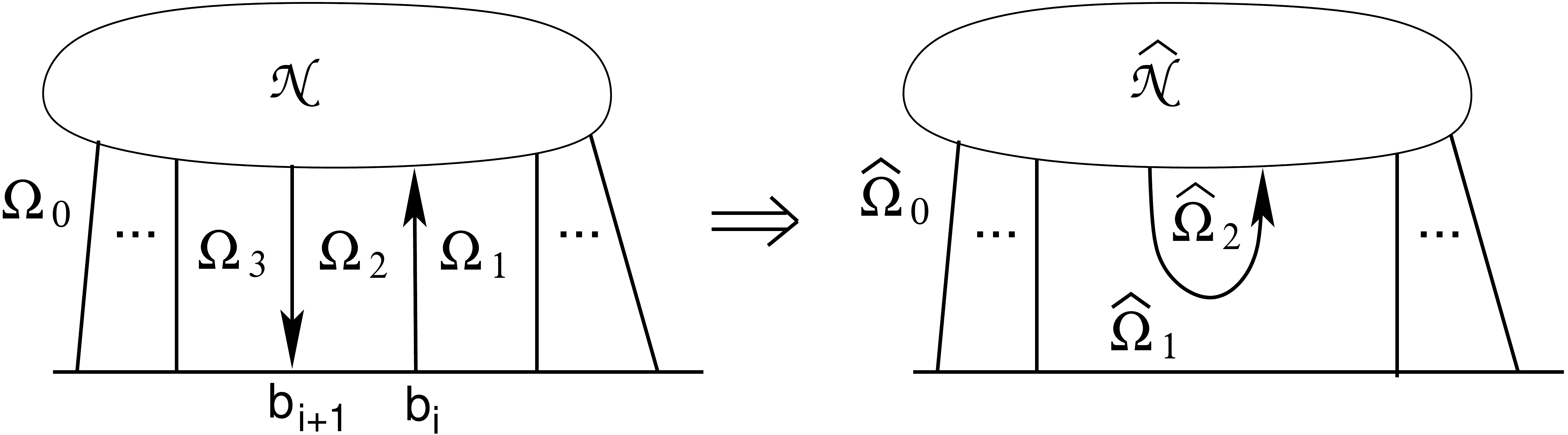} \hfill \includegraphics[width=0.47\textwidth]{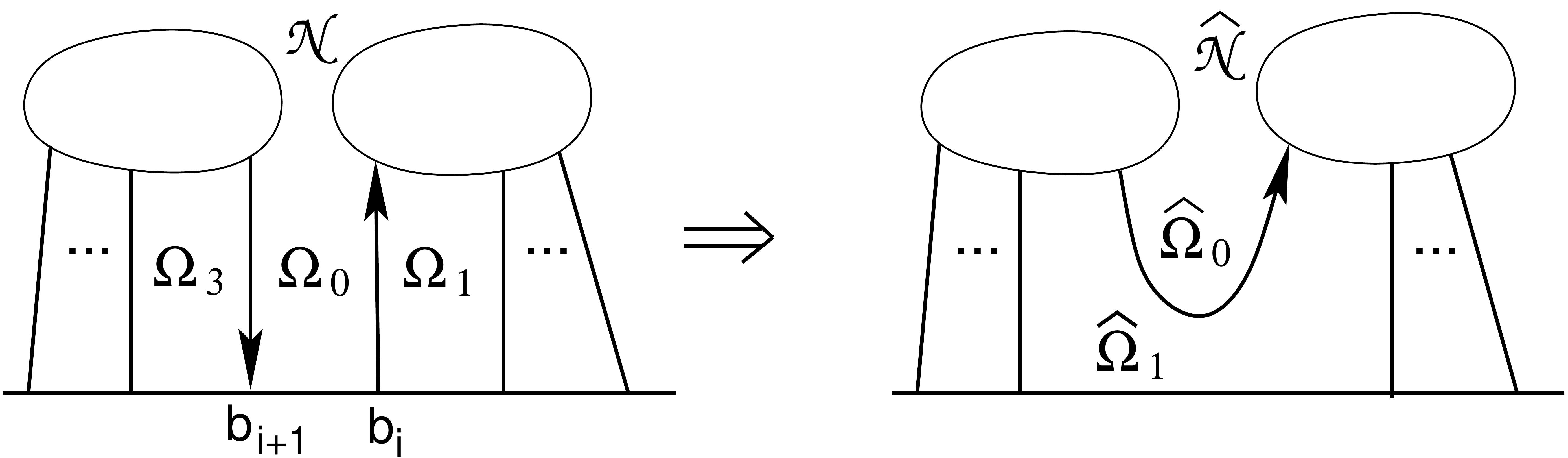}}
  
  \vspace{5mm}
   \includegraphics[width=0.47\textwidth]{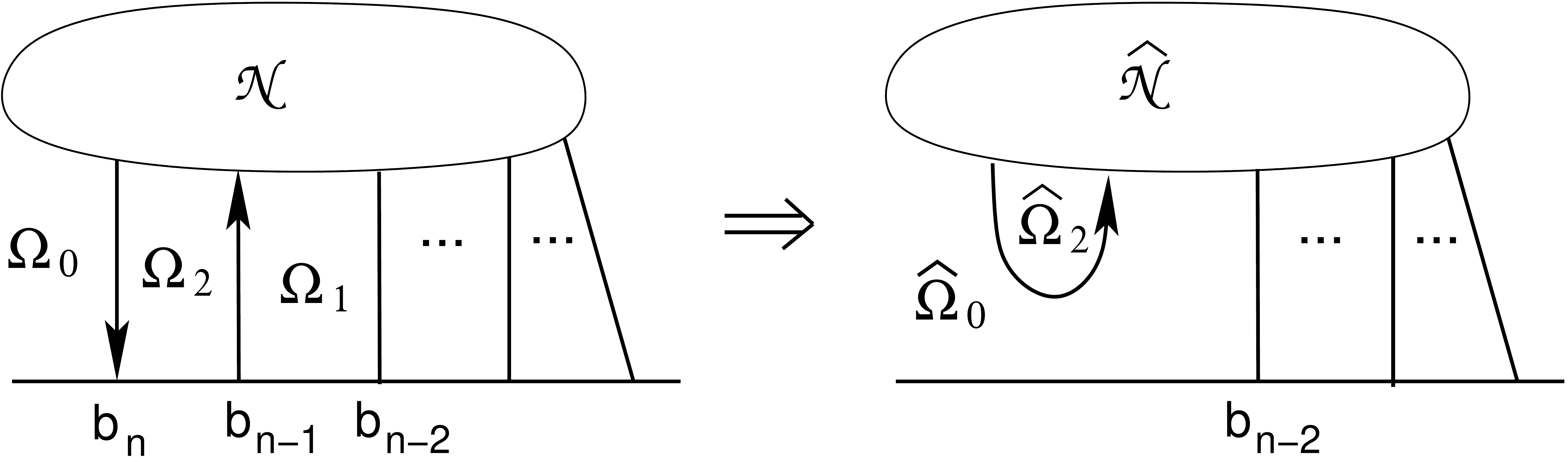}
  \caption{\small{\sl All the admissible ways of defrosting a graph preserve the planarity property.}\label{fig:amalg2}}
\end{figure}

Since we work only with planar directed graphs and we assume that any internal edge belongs to at least one path starting and ending at the boundary of the disc, in our setting defrosting consists in the elimination of two consecutive boundary vertices, one of which is a source  and the other one is a sink, and in gluing the resulting directed half-edges (see Figure~\ref{fig:amalg2}). Defrosting transforms the face $\Omega_2$ into the face $\hat\Omega_2$, and the faces $\Omega_1$, $\Omega_3$ are merged into the face $\hat\Omega_1$.

It is easy to check that any planar graph in the disc considered in our text can be obtained starting from several copies of Le-graphs associated with the small positive Grassmannians $Gr^{TP}(1,3)$, $Gr^{TP}(2,3)$ and $Gr^{TP}(1,2)$ in such a way that at any step planarity is preserved and any edge of the resulting graph belongs at least to one directed path starting and ending at the boundary of the disc. 

Let us now explain the effect of amalgamation on the face geometric signature. At this aim we use Thereom~\ref{theo:sign_face} to compute the edge signatures of the faces of the amalgamated networks in terms of those of the initial networks. As in the previous Section let $\mathcal N$ be a plabic network representing a point in $\GTNN$, and for any given face $\Omega$, let the indices $\epsilon(\Omega)$ and $n_{\mbox{\scriptsize{white}}}(\Omega)$ respectively denote the edge signature and the number of white vertices $\partial \Omega$. Then the proof of the following Lemmas follows from Formula (\ref{eq:sign_face}).

\begin{lemma}\textbf{Edge signature of the direct sum}\label{lem:dir_sum}
Let ${\mathcal N}_i$ be plabic networks representing points in $Gr^{\mbox{\tiny{TNN}}}(k_i,n_i)$, $i=1,2$ and $\mathcal N$ be their disjoint union representing a point in $Gr^{\mbox{\tiny{TNN}}}(k_1+k_2,n_1+n_2)$ with notations as in Figure \ref{fig:amalg1}. Then, the edge signature behaves as follows: 
\begin{enumerate}
\item If all boundary vertices of $\mathcal N_2$ precede all boundary vertices of $\mathcal N_1$ (Figure~\ref{fig:amalg1} [left]),
\begin{equation}\label{eq:inf_case1}
\epsilon(\Omega_0) \; = \;  \epsilon(\Omega_{0,1}) + \epsilon(\Omega_{0,2})  \quad \mod 2,
\end{equation}
and is unchanged in all other faces;
\item If all boundary vertices of ${\mathcal N}_1$ are located between two consecutive boundary vertices of ${\mathcal N}_2$ (Figure~\ref{fig:amalg1} [right]),
\begin{equation}\label{eq:inf_case2}
\epsilon(\Omega_0) \; = \;  \epsilon(\Omega_{0,2}) + k_1  \quad \mod 2,\qquad\qquad
\epsilon(\Omega_1) \; = \;  \epsilon(\Omega_{0,1}) + \epsilon(\Omega_{1,2}) + k_1, \quad \mod 2,
\end{equation}
and is unchanged in all other faces.
\end{enumerate}
\end{lemma}

We now describe the effect of defrosting on the edge signatures. We remark that if both faces $\Omega_1$, $\Omega_3$ are finite, then $\hat\Omega_1$ is also a finite face, otherwise it is the infinite face. Similarly, $\hat\Omega_2$ is the infinite face if and only if $\Omega_2$ is the infinite face.

\begin{lemma}\textbf{Effect of defrosting on edge signatures}\label{lem:defrost}
  Let $\mathcal N$ be a plabic network representing a point in $Gr^{\mbox{\tiny{TNN}}}(k,n)$,  and $\hat{\mathcal N}$ be the defrosted network representing a point in $Gr^{\mbox{\tiny{TNN}}}(k-1,n-2)$ with notations as in Figure \ref{fig:amalg2}. Then, the edge signature behaves as follows:

\begin{enumerate}
\item If $b_i\ne b_1,b_{n-1}$ and $\Omega_2$ is not the infinite face, then 
  \begin{equation}\label{eq:inf_case3}
 \begin{split}   
   \epsilon(\hat\Omega_1) \; &= \;  \epsilon(\Omega_{1}) + \epsilon(\Omega_{3}) +1  \quad \mod 2,\\
   \epsilon(\hat\Omega_2) \; &= \;  \epsilon(\Omega_{2}),   \quad \mod 2,\\
   \epsilon(\hat\Omega_0) \; &= \;  \epsilon(\Omega_{0})  +1   \quad \mod 2,\\
  \end{split} 
\end{equation}
\item If $b_i\ne b_1,b_{n-1}$ and $\Omega_2$ is the infinite face $\Omega_2=\Omega_0$ , then 
  \begin{equation}\label{eq:inf_case3_bis}
 \begin{split}   
   \epsilon(\hat\Omega_1) \; &= \;  \epsilon(\Omega_{1}) + \epsilon(\Omega_{3}) +1  \quad \mod 2,\\
   \epsilon(\hat\Omega_2) \; &= \;  \epsilon(\Omega_{2}) +1 ,   \quad \mod 2,\\
  \end{split} 
\end{equation}

\item If $b_i=b_{n-1}$ then 
  \begin{equation}\label{eq:inf_case4}
 \begin{split}   
   \epsilon(\hat\Omega_0) \; &= \;  \epsilon(\Omega_{0}) +  \epsilon(\Omega_{1})  \quad \mod 2,\\
   \epsilon(\hat\Omega_2) \; &= \;  \epsilon(\Omega_{2}),   \quad \mod 2.
 \end{split}
\end{equation}                                                             
\item The case  $b_i=b_{1}$ is similar to the previous one.
\end{enumerate}
In all other faces the signature is unchanged.
\end{lemma}

\section{Appendix: Proof of Theorem~\ref{thm:z_orient_gauge} (invariance of the geometric signature)}\label{app:invariance}

\begin{remark} All the identities in this Section are meant $\mod 2$.\end{remark}

The following relations between the geometric indices and the cyclic order are proven in \cite{AG6}.

\begin{lemma} \cite{AG6}
  \label{lemma:equiv_rel}
  Let $V$ belong to $\mathcal P_0$ or $\mathcal Q_0$, $e_1,e_2,f$ be the edges at $V$, where $e_1$, $e_2$ belong to  $\mathcal P_0$ or $\mathcal Q_0$,  and $e_1$ (respectively $e_2$) is an incoming (respectively outgoing) edge in the initial configuration. 
\begin{enumerate}
\item If $V$ is black, then:
  \begin{align}
 \label{eq:black1}   
\widehat{\mbox{int}}(f) + \mbox{int}(f) + \gamma(f) +\gamma_1(e_1) &= [e_1,-e_2,f] \quad
                                                                         (\!\!\!\!\!\!\mod 2),\\
 \label{eq:black2}      
 \mbox{wind}(e_1,e_2) + \mbox{wind}(f,e_2) +  \mbox{wind}(f,-e_1) +\gamma_2 (e_1) &=  [e_1,-e_2,f] \quad
(\!\!\!\!\!\!\mod 2). 
\end{align}
\item If $V$ is white, then:
  \begin{align}
   \label{eq:white1}    
\gamma(f) +\gamma_1(e_1) &= [e_1,-e_2,-f] \quad
                               (\!\!\!\!\!\!\mod 2),\\
  \label{eq:white2}     
  \mbox{wind}(e_1,f) + \mbox{wind}(-e_2,f) + \mbox{wind}(-e_2,-e_1) +\gamma_2(e_1)&= 1-[e_1,-e_2,-f] \quad
(\!\!\!\!\!\!\mod 2). 
\end{align}
\end{enumerate}
\end{lemma}

Finally we prove Theorem \ref{thm:z_orient_gauge} step by step:

\begin{enumerate}
\item Let us prove that every change of orientation corresponds to a gauge transformation as in (\ref{eq:zeta_orient}):
\begin{enumerate}  
\item If $U$ and $V$ do not belong to $\mathcal Q$ (respectively $\mathcal P$), then the winding numbers at $U$ and $V$ remain the same, and the intersection number of the edge $(U,V)$ changes if and only if $U$ and $V$ lie in regions with different sign marks.  If $V$ is either a boundary source or a boundary sink, by construction the latter is always in a $+$ region. Therefore, in this case (\ref{eq:zeta_orient}) holds true.
\begin{figure}
  \centering
	{\includegraphics[width=0.9\textwidth]{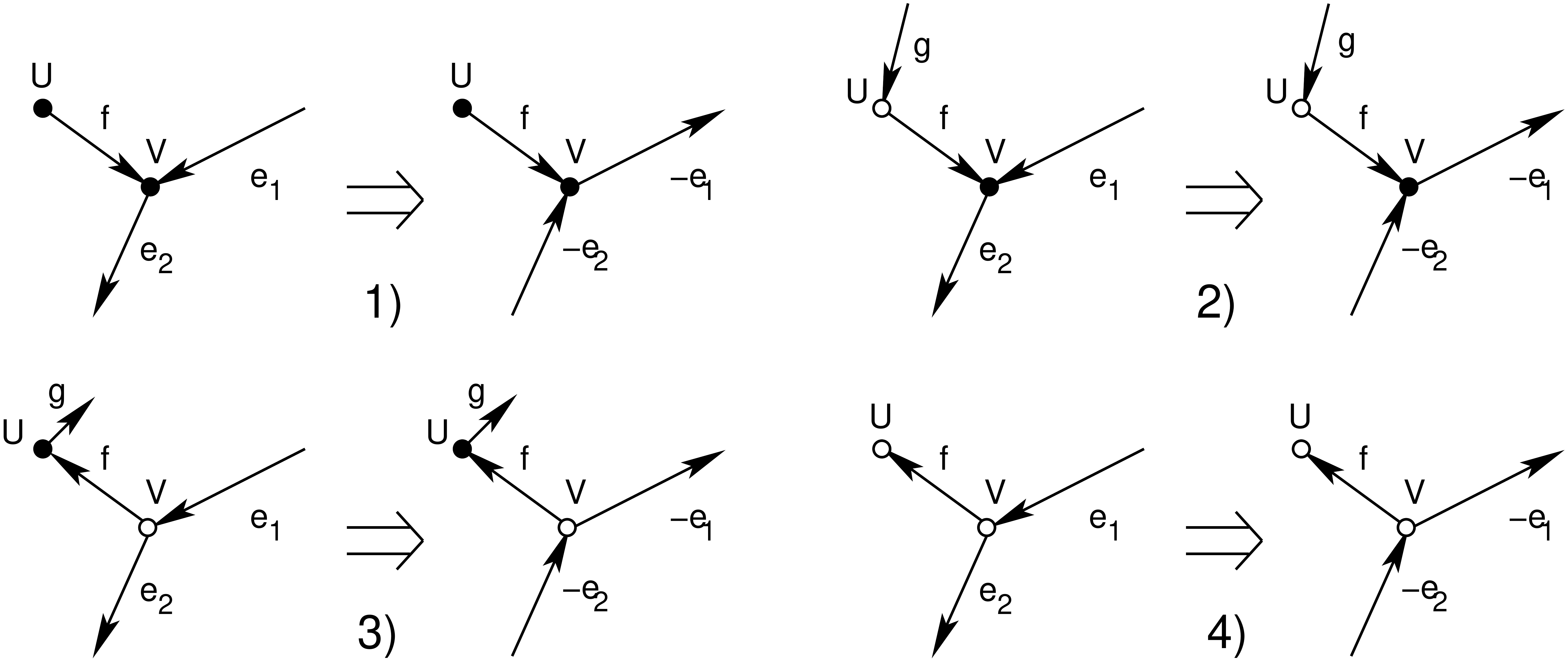}}
  \caption{\small{\sl }\label{fig:gauge1}}
\end{figure}
\item If $U$ does not belong to $\mathcal Q$ (respectively $\mathcal P$) and $V$ belongs to it, equations (\ref{eq:black1}),
  (\ref{eq:white1}) are equivalent to:
\begin{equation}\label{eq:eta-outside}  
  \eta(U) = \mbox{int}(f)+ \widehat{\mbox{int}}(f) + \gamma_1(e_1) +
  \left\{
    \begin{array}{ll} [e_1,-e_2,f] & \mbox{if} \ \ V \ \ \mbox{is black}, \\   
                       {[ e_1,-e_2,-f]}  & \mbox{if} \ \ V \ \ \mbox{is white}. 
    \end{array}
  \right.
\end{equation}
We also have:
\begin{equation}\label{eq:wind-opposite}  
\mbox{wind}(e_1,e_2) + \mbox{wind}(-e_2,-e_1) = \gamma_2(e_1) + \gamma_2(e_2).
\end{equation}

We have four possible cases  (see Figure~\ref{fig:gauge1}). 
\begin{enumerate}
\item If $V$ is black (Cases 1 and 2), then
$$
\epsilon(f) + \hat\epsilon(f) = \mbox{int}(f) +  \widehat{\mbox{int}}(f) + \mbox{wind}(f,e_2)+ \mbox{wind}(f,-e_1).
$$  
Using (\ref{eq:black2}) we obtain:
$$
\epsilon(f) + \hat\epsilon(f) = \mbox{int}(f) +  \widehat{\mbox{int}}(f) + \mbox{wind}(e_1,e_2)+ \gamma_2(e_1) + [e_1,-e_2,f]=
$$  
$$
= \big[\mbox{int}(f) +  \widehat{\mbox{int}}(f) +  \gamma_1(e_1)  + [e_1,-e_2,f] \big] + \big[ \gamma_1(e_1) + \gamma_2(e_1)  + \mbox{wind}(e_1,e_2)  \big] = \eta(U) + \eta(V);
$$
\item If $V$ is white (Cases 3 and 4), then
$$
\epsilon(f) + \hat\epsilon(f) = \mbox{int}(f) +  \widehat{\mbox{int}}(f) + \mbox{wind}(e_1,f)+ \mbox{wind}(-e_2,f).
$$  
Using (\ref{eq:white2}) we obtain:
$$
\epsilon(f) + \hat\epsilon(f) = \mbox{int}(f) +  \widehat{\mbox{int}}(f) +1 + \mbox{wind}(-e_2,-e_1)+ \gamma_2(e_1) + [e_1,-e_2,-f].
$$
Using (\ref{eq:wind-opposite}) we get  
$$
\epsilon(f) + \hat\epsilon(f) = \mbox{int}(f) +  \widehat{\mbox{int}}(f) + \mbox{wind}(e_1,e_2)+ 1+ \gamma_2(e_2) + [e_1,-e_2,-f]=
$$
$$
= \big[\mbox{int}(f) +  \widehat{\mbox{int}}(f) +  \gamma_1(e_1)  + [e_1,-e_2,-f] \big] + \big[1+ \gamma_1(e_1) + \gamma_2(e_1)  + \mbox{wind}(e_1,e_2)  \big] = \eta(U) + \eta(V).
$$
\end{enumerate}
\begin{figure}
  \centering
	{\includegraphics[width=0.9\textwidth]{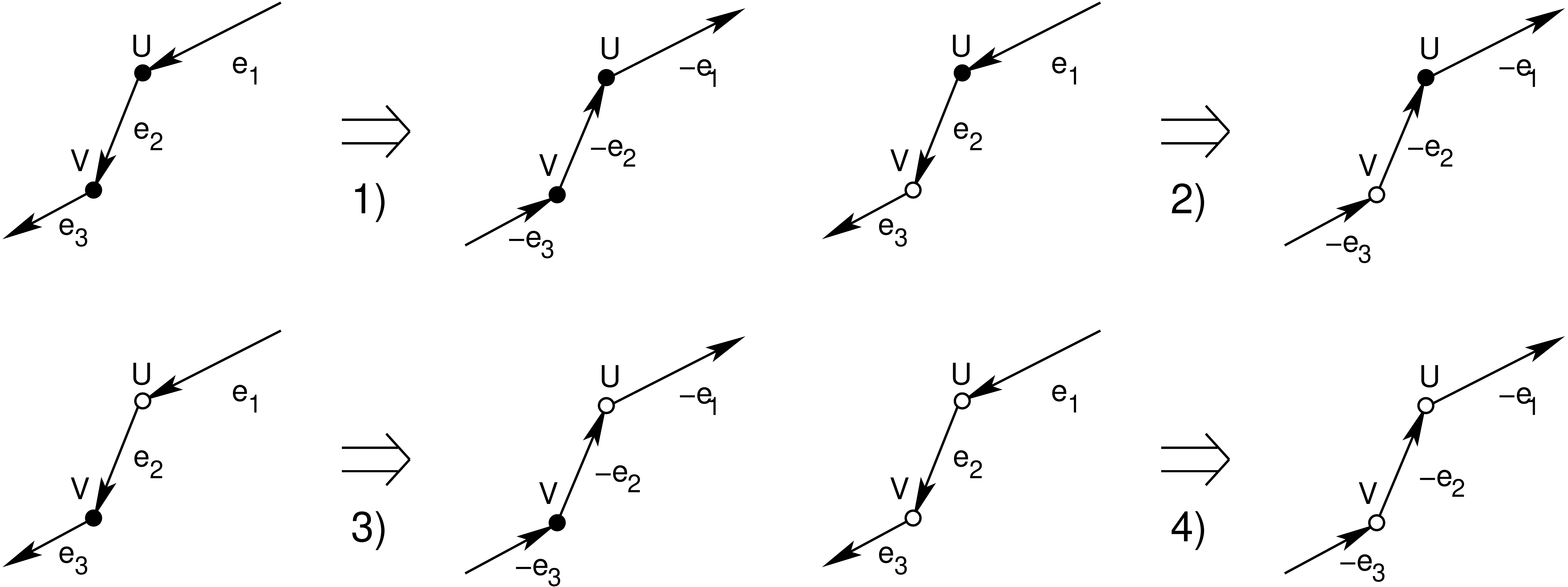}}
  \caption{\small{\sl }\label{fig:gauge2}}
\end{figure}
\item  If both $U$ and $V$ belong to $\mathcal Q$ ($\mathcal P$ respectively) and none of them is a boundary vertex, then we have four possible cases  (see Figure~\ref{fig:gauge2}). It is easy to check that in all cases
\begin{equation}\label{eq:int-diff}  
 \mbox{int}(e_2) +  \widehat{\mbox{int}}(e_2) = \gamma_1(e_1) + \gamma_1(e_2).
\end{equation}
\begin{enumerate}
\item If both $U$ and $V$ are black (Case 1) , then
$$
\epsilon(e_2) + \hat\epsilon(e_2) =   \mbox{int}(e_2) +  \widehat{\mbox{int}}(e_2) +  \mbox{wind}(e_2,e_3) +  \mbox{wind}(-e_1,-e_2). 
$$
Using (\ref{eq:wind-opposite}) and (\ref{eq:int-diff}) we obtain
$$
\epsilon(e_2) + \hat\epsilon(e_2) =   \gamma_1(e_1) + \gamma_1(e_2) +  \mbox{wind}(e_2,e_3) +  \mbox{wind}(e_1,e_2) + \gamma_2(e_2) + \gamma_2(e_1)= 
$$
$$
= [\mbox{wind}(e_1,e_2) +  \gamma_1(e_1) + \gamma_2(e_1) ] + [  \mbox{wind}(e_2,e_3) + \gamma_1(e_2) + \gamma_2(e_2) ]= \eta(U)+ \eta(V);
$$
\item If $U$ is black and $V$ is white (Case 2) , then
$$
\epsilon(e_2) + \hat\epsilon(e_2) =   \mbox{int}(e_2) +  \widehat{\mbox{int}}(e_2) +  \mbox{wind}(-e_1,-e_2) +  \mbox{wind}(-e_2,-e_3)+1. 
$$
Using (\ref{eq:wind-opposite}) and (\ref{eq:int-diff}) we obtain
$$
\epsilon(e_2) + \hat\epsilon(e_2) =   \gamma_1(e_1) + \gamma_1(e_2) +  \mbox{wind}(e_1,e_2) + \mbox{wind}(e_2,e_3) + \gamma_2(e_1) + \gamma_2(e_3) +1 = 
$$
$$
=[ \gamma_1(e_1) + \gamma_2(e_1) + \mbox{wind}(e_1,e_2) ] + [ \gamma_1(e_2) + \gamma_2(e_3) +  \mbox{wind}(e_2,e_3) +1] =  \eta(U)+ \eta(V);
$$
\item If $U$ is white and $V$ is black (Case 3) , then
$$
\epsilon(e_2) + \hat\epsilon(e_2) =   \mbox{int}(e_2) +  \widehat{\mbox{int}}(e_2) +  \mbox{wind}(e_1,e_2) +  \mbox{wind}(e_2,e_3)+1. 
$$
Using (\ref{eq:int-diff}) we obtain
$$
\epsilon(e_2) + \hat\epsilon(e_2) =   \gamma_1(e_1) + \gamma_1(e_2) +  \mbox{wind}(e_1,e_2) + \mbox{wind}(e_2,e_3) +1 = 
$$
$$
=[ \gamma_1(e_1) + \gamma_2(e_2) + \mbox{wind}(e_1,e_2) +1 ] + [ \gamma_1(e_2) + \gamma_2(e_2) +  \mbox{wind}(e_2,e_3) ] =  \eta(U)+ \eta(V);
$$
\item If both $U$ and $V$ are white (Case 4), then
$$
\epsilon(e_2) + \hat\epsilon(e_2) =   \mbox{int}(e_2) +  \widehat{\mbox{int}}(e_2) +  \mbox{wind}(e_1,e_2) +  \mbox{wind}(-e_2,-e_3). 
$$
Using (\ref{eq:wind-opposite}) and (\ref{eq:int-diff}) we obtain
$$
\epsilon(e_2) + \hat\epsilon(e_2) =   \gamma_1(e_1) + \gamma_1(e_2) +  \mbox{wind}(e_1,e_2) +  \mbox{wind}(e_2,e_3) + \gamma_2(e_2) + \gamma_2(e_3)= 
$$
$$
= [\mbox{wind}(e_1,e_2) +  \gamma_1(e_1) + \gamma_2(e_2) + 1 ] + [  \mbox{wind}(e_2,e_3) + \gamma_1(e_2) + \gamma_2(e_3)+ 1 ]= \eta(U)+ \eta(V).
$$
\end{enumerate}
\item  If $U$ belongs to $\mathcal P$, and $V=b$ is a boundary vertex on it, then we have two possible cases  (see Figure~\ref{fig:gauge3}). 
\begin{figure}
  \centering
	{\includegraphics[width=0.9\textwidth]{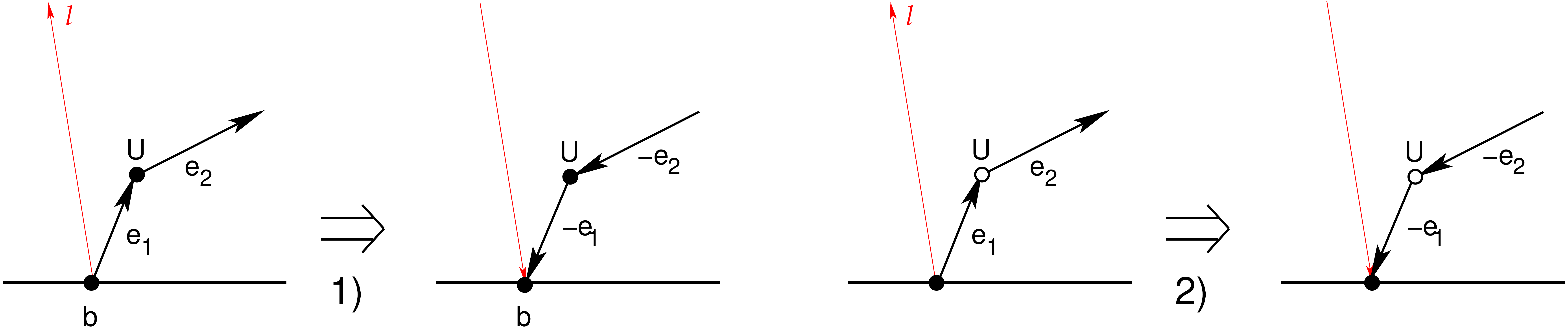}}
  \caption{\small{\sl }\label{fig:gauge3}}
\end{figure}
In both cases we have 
\begin{equation}\label{eq:int-diff-boundary}
  \mbox{int}(e_1) +  \widehat{\mbox{int}}(e_1) = \gamma_1(e_1) + \gamma_2(e_1)+1. 
\end{equation}
\begin{enumerate}
\item If $U$ is black, then
  $$
  \epsilon(e_1) + \hat\epsilon(e_1) = 1 + \mbox{int}(e_1) + \mbox{wind}(e_1,e_2)+ \widehat{\mbox{int}}(e_1)=
  $$
  $$
  =  \gamma_1(e_1) + \gamma_2(e_1) +  \mbox{wind}(e_1,e_2)= \eta(U); 
  $$
\item  If $U$ is white, then
  $$
  \epsilon(e_1) + \hat\epsilon(e_1) = 1 + \mbox{int}(e_1) + 1+ \widehat{\mbox{int}}(e_1) +  \mbox{wind}(-e_2,-e_1) =
  $$
  $$
  = 1 + \gamma_1(e_1) + \gamma_2(e_1) + \mbox{wind}(e_1,e_2) +\gamma_2(e_1) +  \gamma_2(e_2)  = \eta(U).   
  $$
\end{enumerate}
\end{enumerate}  
The proof of the first statement is complete.

\item To prove the invariance with respect to a change of gauge direction, we assume that the graph is generic, and we continuously rotate the gauge ray direction from $\mathfrak l$ to $\hat{ \mathfrak l}$. Then $\epsilon_{U,V}$ and $\eta(U)$ become functions of the rotation parameter. Let us check that (\ref{eq:equiv_sign}) is true for all values of the rotation parameter.

Let $e=(U,V)$ be an internal edge (for boundary edges the proof goes through with obvious modifications). The left-hand side and the right-hand side of  (\ref{eq:equiv_sign}) may change if:
\begin{enumerate}
\item A gauge ray crosses either $U$ or $V$; 
\item The gauge ray direction becomes parallel to one of the edges at $U$ or $V$.
\end{enumerate}
\begin{figure}
  \centering
	{\includegraphics[width=0.8\textwidth]{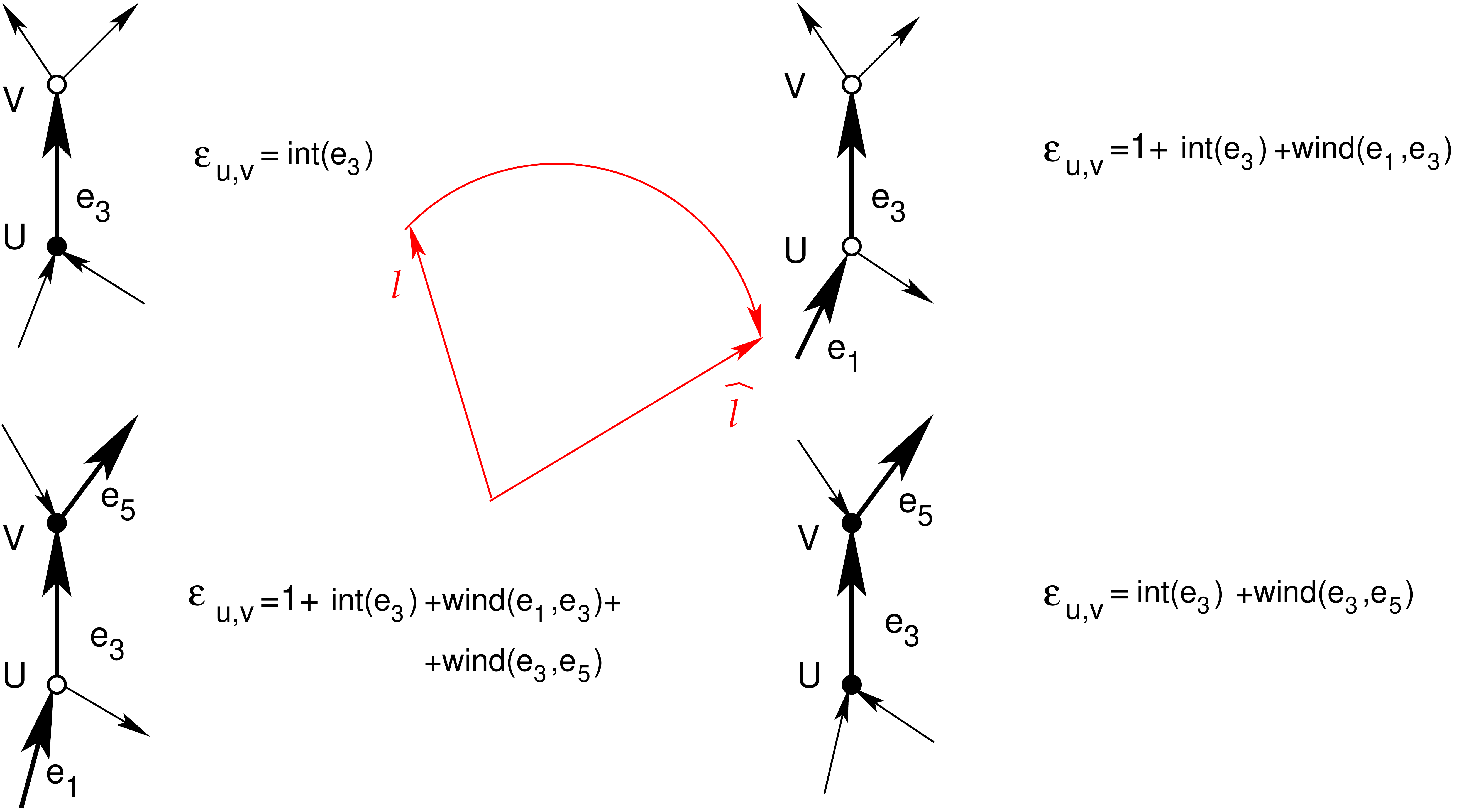}}
  \caption{\small{\sl }\label{fig:gauge4}}
\end{figure}
The assumption that the graph is generic means that we have at most one of such events for each value of the rotation parameter. Let us check that after each such an event (\ref{eq:equiv_sign}) remains true.
\begin{enumerate}
\item If a gauge ray crosses $U$, then
$$
 \epsilon_{U,V} \rightarrow \epsilon_{U,V} + 1, \ \ \eta(U) \rightarrow \eta(U) +1, \ \
 \eta(V) \rightarrow \eta(V); 
$$
\item If a gauge ray crosses $V$, then
$$
\epsilon_{U,V} \rightarrow \epsilon_{U,V} + 1, \ \ \eta(U) \rightarrow \eta(U), \ \
\eta(V) \rightarrow \eta(V)+1;  
$$
\item Let us check what happens if the gauge ray direction is parallel (not antiparallel) to one of the edges at $U$ or $V$. In Figure~\ref{fig:gauge4} the edges appearing in the formulas for the signature or gauge are drawn thick.

If the gauge ray direction becomes parallel to one of the edges drawn thin, neither $\epsilon_{U,V}$ nor $\eta(U)$, $\eta(U)$, change. Therefore it is sufficient to check what happens if the gauge ray direction is parallel to one of the thick edges.
\begin{enumerate}
\item If $U$ is black, $V$ is white, and $l$ passes the direction of $e_3$, then we have the following transformation:
$$
\epsilon_{U,V} \rightarrow \epsilon_{U,V},\ \
\eta(U) \rightarrow \eta(U)+1, \ \ \eta(V) \rightarrow \eta(V)+1; 
$$
 \item If $U$ and $V$ are white, and $l$ passes the direction of $e_1$, then we have the following transformation:
$$
\epsilon_{U,V} \rightarrow \epsilon_{U,V}+1,\ \
\eta(U) \rightarrow \eta(U)+1, \ \ \eta(V) \rightarrow \eta(V);  
$$
 \item If $U$ and $V$ are white, and $l$ passes the direction of $e_3$, then we have the following transformation:
$$
\epsilon_{U,V} \rightarrow \epsilon_{U,V}+1,\ \
\eta(U) \rightarrow \eta(U), \ \ \eta(V) \rightarrow \eta(V)+1;  
$$
\item If $U$ is white, $V$ is black, and $l$ passes the direction of $e_1$, then we have the following transformation:
$$
\epsilon_{U,V} \rightarrow \epsilon_{U,V}+1, \ \
\eta(U) \rightarrow \eta(U)+1, \ \ \eta(V) \rightarrow \eta(V);
$$
\item If $U$ is white, $V$ is black, and $l$ passes the direction of $e_3$, then we have the following transformation:
$$
\epsilon_{U,V} \rightarrow \epsilon_{U,V}, \ \
\eta(U) \rightarrow \eta(U), \ \ \eta(V) \rightarrow \eta(V);
$$
\item If $U$ is white, $V$ is black, and $l$ passes the direction of $e_3$, then we have the following transformation:
$$
\epsilon_{U,V} \rightarrow \epsilon_{U,V}+1, \ \
\eta(U) \rightarrow \eta(U), \ \ \eta(V) \rightarrow \eta(V)+1;
$$
\item If $U$, $V$ are black, and $l$ passes the direction of $e_3$, then we have the following transformation:
$$
\epsilon_{U,V} \rightarrow \epsilon_{U,V}+1, \ \
\eta(U) \rightarrow \eta(U)+1, \ \ \eta(V) \rightarrow \eta(V);
$$
\item If $U$, $V$ are black, and $l$ passes the direction of $e_5$, then we have the following transformation:
$$
\epsilon_{U,V} \rightarrow \epsilon_{U,V}+1, \ \
\eta(U) \rightarrow \eta(U), \ \ \eta(V) \rightarrow \eta(V)+1.
$$

\end{enumerate}
\end{enumerate}
The proof of the second statement is completed.

\item Let us prove the invariance of the signature with respect to a change of position of an internal vertex.

Again we assume that our configuration is generic, we continuously move
the vertex $U$, so that the position $U$ is a smooth function of a parameter $t$, and for each $t$ at most one of the following events may occur:
\begin{enumerate}
\item $U$ crosses one of the gauge rays, and there is no edge at $U$ parallel to the gauge ray direction;
\item One of the edges at $U$ is parallel to the gauge direction.  
\end{enumerate}

In the first case $\eta(U)$ changes by 1, all edges connected to $U$ change the intersection number by 1, and all windings remain unchanged. Therefore the signatures at these edges change by 1, and the statement holds.   

To check the Theorem in the second case we consider all possible configurations.

Assume that $U$ is not connected by an edge to the boundary.
\begin{enumerate}
\item Let $U$ be white as in Figure~\ref{fig:move_white}.
 \begin{figure}
  \centering
	{\includegraphics[width=0.49\textwidth]{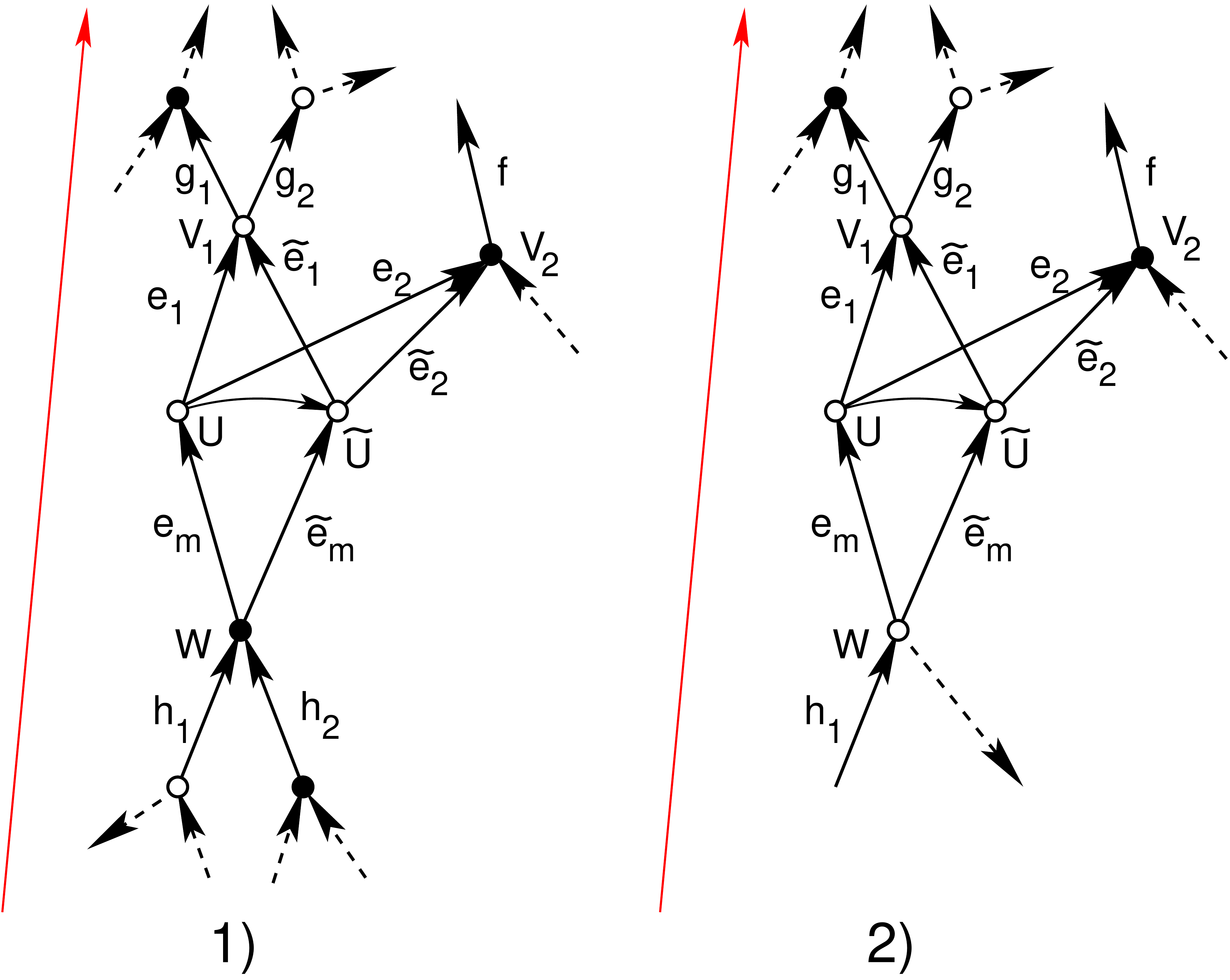}}
  \caption{\small{\sl If one of the edges at $U$ becomes parallel to $\mathfrak l$, some windings may change. We mark the edges participating in these windings with continuous lines and all other edges with dashed lines.}\label{fig:move_white}}
\end{figure} 
Here $W$ is the second end of the unique incoming edge $e_m$, $V_i$ denote the ends of the outgoing edges $e_i$ at $U$. A simple calculation shows that the statement is correct since:
\begin{align}\label{eq:move_white}
  \epsilon(f) -\tilde\epsilon(f) &= 0,\nonumber\\
  \epsilon(g_i) -\tilde\epsilon(g_i) &= \mbox{par}(e_1), \nonumber \\
  \epsilon(e_i) -\tilde\epsilon(e_i) &= \left\{
                                       \begin{array}{ll} \mbox{par}(e_1)+ \mbox{par}(e_m) &
                                                  \mbox{ if } V_i \mbox{  is white},  \\
                                         \mbox{par}(e_m) & \mbox{ if } V_i \mbox{  is black}, 
                                       \end{array}\right. \nonumber \\
  \epsilon(e_m) -\tilde\epsilon(e_m) &= \left\{
                                       \begin{array}{ll} \mbox{par}(e_m) &
                                                  \mbox{ if } W \mbox{  is white},  \\
                                         0 & \mbox{ if } W \mbox{  is black}, 
                                       \end{array}\right. \nonumber \\
   \epsilon(h_i) -\tilde\epsilon(h_i) &= \left\{
                                       \begin{array}{ll} 0 &
                                                  \mbox{ if } W \mbox{  is white},  \\
                                         \mbox{par}(e_m) & \mbox{ if } W  \mbox{  is black}; 
                                       \end{array}\right. \nonumber \\
\end{align}
\item Let $U$ be black as in Figure~\ref{fig:move_black}.
 \begin{figure}
  \centering
	{\includegraphics[width=0.49\textwidth]{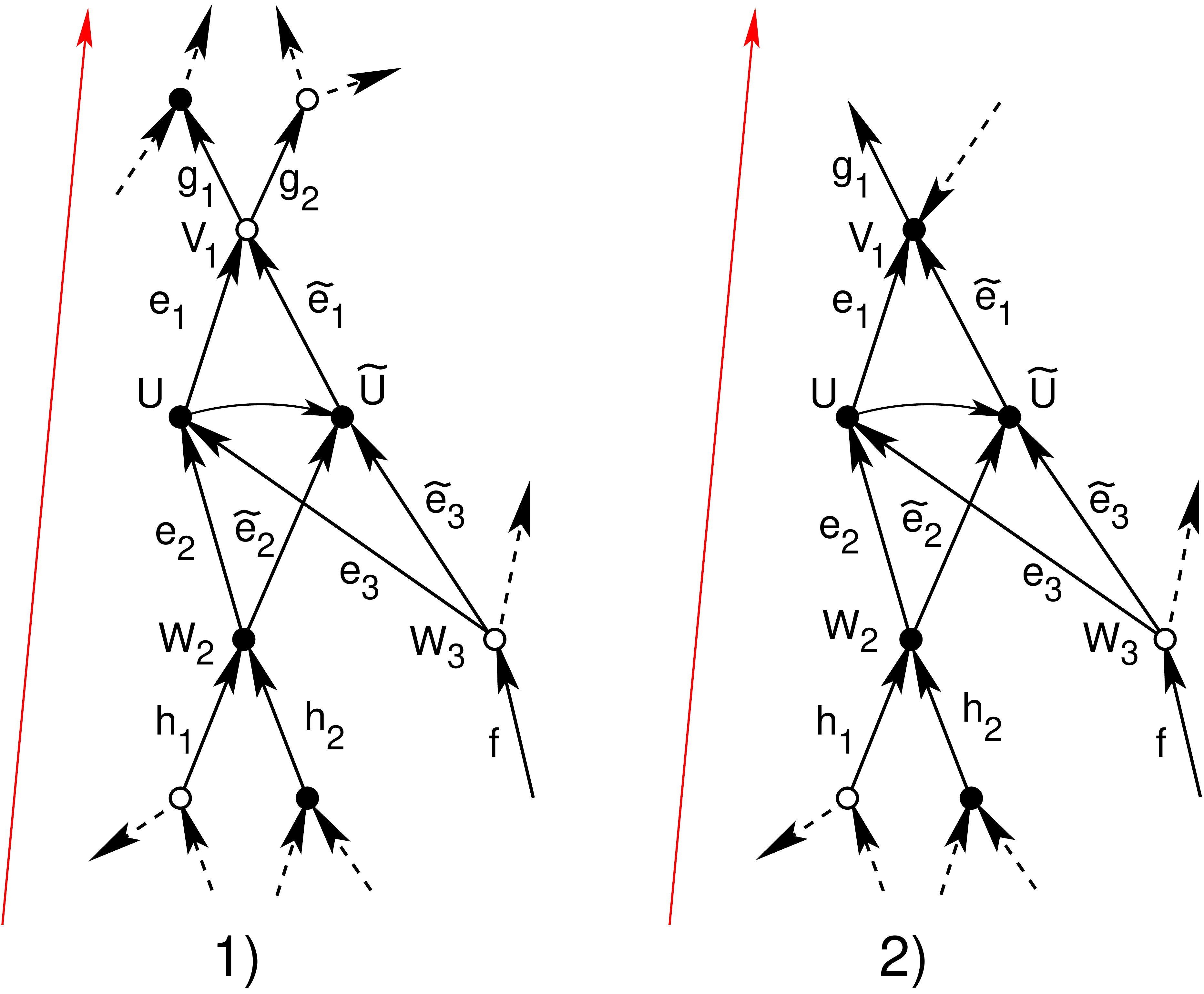}}
  \caption{\small{\sl If one of the edges at $U$ becomes parallel to $\mathfrak l$, some windings may change. We mark the edges participating in these windings with continuous lines and all other edges with dashed lines.}\label{fig:move_black}}
\end{figure} 
Again, a simple calculation shows that the statement is correct since:
\begin{align}\label{eq:move_black}
  \epsilon(f) -\tilde\epsilon(f) &= 0,\nonumber\\
  \epsilon(g_i) -\tilde\epsilon(g_i) &= \left\{
                                       \begin{array}{ll} \mbox{par}(e_1)  &
                                                  \mbox{ if } V_1 \mbox{  is white},  \\
                                         0 & \mbox{ if } V_1 \mbox{  is black}, 
                                       \end{array}\right. \nonumber \\
  \epsilon(e_1) -\tilde\epsilon(e_1) &=  \left\{
                                       \begin{array}{ll} 0 & 
                                                  \mbox{ if } V_1 \mbox{  is white},  \\
                                          \mbox{par}(e_1)  & \mbox{ if } V_1 \mbox{  is black}, 
                                       \end{array}\right. \nonumber \\
  \epsilon(e_j) -\tilde\epsilon(e_j) &= \left\{
                                       \begin{array}{ll} \mbox{par}(e_1) &
                                                  \mbox{ if } W_j \mbox{  is white},  \\
                                         \mbox{par}(e_1) +  \mbox{par}(e_j) & \mbox{ if } W_j \mbox{  is black}, 
                                       \end{array}\right. \nonumber \\
   \epsilon(h_j) -\tilde\epsilon(h_j) &= \left\{
                                       \begin{array}{ll} 0 &
                                                  \mbox{ if } W \mbox{  is white},  \\
                                         \mbox{par}(e_j) & \mbox{ if } W_j  \mbox{  is black}; 
                                       \end{array}\right. \nonumber \\
\end{align}
\end{enumerate}

If $V$ is a boundary sink, then $\mbox{par}(U,V)=0$, and formulas  (\ref{eq:move_white}),  (\ref{eq:move_black}) are valid.

If $W=b$ is a boundary source, and $U$ does not pass the gauge ray starting at $b$, then $\mbox{par}(b,U)=0$, and formulas  (\ref{eq:move_white}),  (\ref{eq:move_black}) are valid. If $U$ passes the gauge ray starting at $b$, then contemporary $\mbox{cr}(U)$ and $\mbox{par}(b,U)$ change by 1. 
\begin{figure}
  \centering
	{\includegraphics[width=0.9\textwidth]{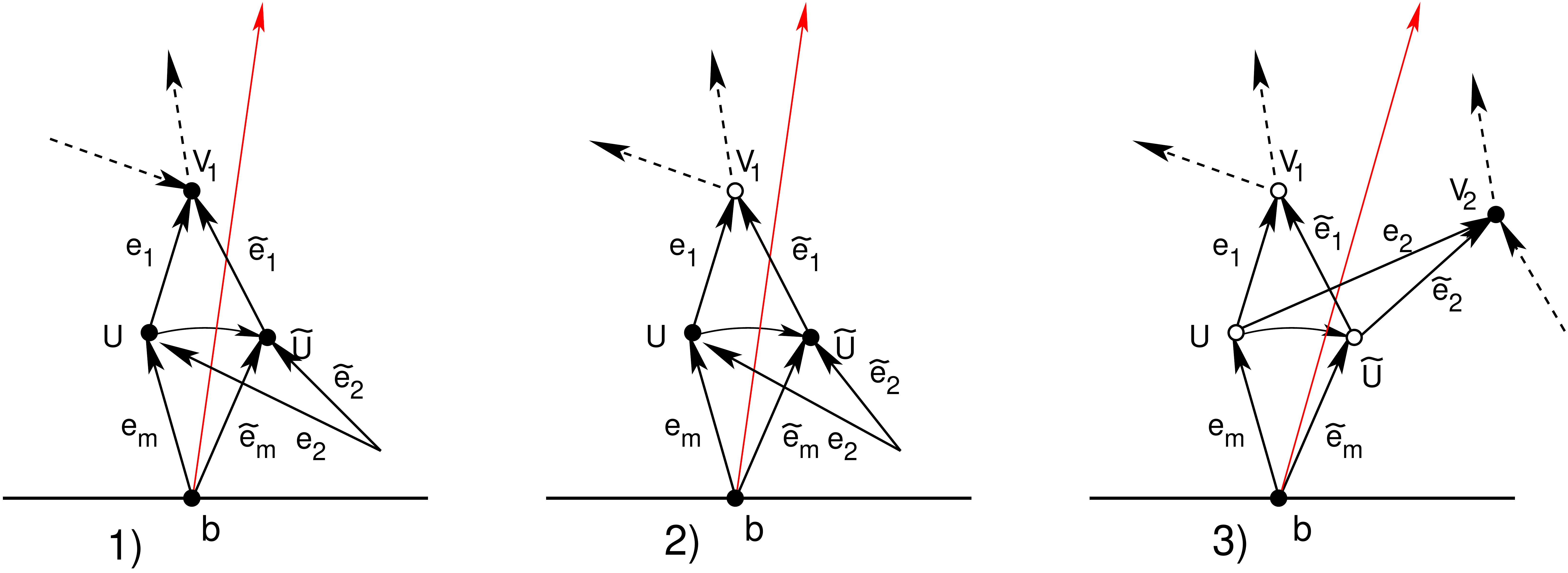}}
  \caption{\small{\sl If one of the edges at $U$ becomes parallel to $\mathfrak l$, some windings may change. We mark the edges participating in these windings with continuous lines and all other edges with dashed lines.}\label{fig:move_at_source}}
\end{figure}
In this case all other edges at $U$ change intersection index by 1, and winding numbers as in  formulas  (\ref{eq:move_white}),  (\ref{eq:move_black}). Therefore, in the notations of Figure~\ref{fig:move_at_source}:
\begin{equation}\label{eq:move_at_source}
  \epsilon(e_j) -\tilde\epsilon(e_j) = \left\{
                                       \begin{array}{ll} 1 & \mbox{ if } U \mbox{  is black},\\ 
                                                         0 & \mbox{ if } U \mbox{  is white}, 
                                       \end{array}\right. \ \ \mbox{for all  } j.
\end{equation}
The proof is complete.
\end{enumerate}

\section*{Acknowledgments} 
The authors would like to express their gratitude to T. Lam for pointing our attention to Reference \cite{AGPR}.


\begin{thebibliography}{0}

\bibitem{A3} Abenda, S. ``Kasteleyn theorem, geometric signatures and KP-II divisors on planar bipartite networks in the disk.'', \textit{Math. Phys. Anal. Geom.} 24, Art. \#35, (2021): 64 pp.
 
\bibitem{AG1} Abenda, S., and P.G. Grinevich,``Rational degenerations of $M$-curves, totally positive Grassmannians and KP--solitons.'' \textit{Comm. Math. Phys.} \textbf{361}, no. 3 (2018): 1029--1081.

\bibitem{AG2} Abenda, S., and P.G. Grinevich, ``Real soliton lattices of the Kadomtsev-Petviashvili II equation and desingularization of spectral curves corresponding to $Gr^{\mbox{\tiny TP}}(2,4)$.''  \textit{Proc. Steklov Inst. Math.} \textbf{302}, no. 1 (2018): 1--15. 

\bibitem{AG3} Abenda, S., and P.G. Grinevich, ``Reducible $M$-curves for Le-networks in the totally-nonnegative Grassmannian and KP--II multiline solitons.'' \textit{Selecta Math. (N.S.)} \textbf{25}, no. 3 (2019) 25:43. https://doi.org/10.1007/s00029-019-0488-5

\bibitem{AG5} Abenda, S. and P.G. Grinevich, {\em Real regular KP divisors on M-curves and totally non-negative Grassmannians}, arXiv:2002.04865.

\bibitem{AG6} Abenda, S., and P.G. Grinevich, {\em A generalization of Talaska formula for edge vectors on plabic networks in the disk}, arXiv:2108.03229.

\bibitem{AGPR} Affolter, N., M. Glick, P. Pylyavskyy, and S. Ramassamy, ``Vector--relation configurations and plabic graphs'', \textit{S\'em. Lothar. Combin.} \textbf{84B} (2020), Art. \#91, 12 pp. 


\bibitem{AGP2} Arkani--Hamed, N., J.L. Bourjaily, F. Cachazo, A.B. Goncharov, A. Postnikov, and J. Trnka, \textit{Grassmannian geometry of scattering amplitudes.} Cambridge University Press, Cambridge, 2016.

\bibitem{ADM} Atiyah, M., M. Dunajski, and L.J. Mason, ``Twistor theory at fifty: from contour integrals to twistor strings.'' \textit{Proc. A}  473 (2017):  20170530, 33 pp.

\bibitem{BK} Biondini, G., and Yu. Kodama, ``On a family of solutions of the Kadomtsev--Petviashvili equation which also satisfy the Toda
lattice hierarchy.'', \textit{J. Phys. A: Math. Gen.} \textbf{36} (2003) 10519--10536.
  
\bibitem{BPPP} Boiti, M., F. Pempinelli, A.K. Pogrebkov, B. Prinari, ``Towards an inverse scattering theory for non-decaying potentials of the heat equation'', \textit{Inverse Problems} {\bf 17} (2001) 937--957.

\bibitem{BFGW} Bourjaily, J.L., S. Franco, D. Galloni, and C. Wen, ``Stratifying on--shell cluster varieties: the geometry of non--planar on--shell diagrams.'' \textit{J. High Energy Phys.} (2016), no. 10, 003, front matter+30 pp.

\bibitem{BG} Buchstaber, V., and A. Glutsyuk, \textit{Total positivity, Grassmannian and modified Bessel functions.} in: 
Functional Analysis and Geometry: Selim Grigorievich Krein Centennial. Editors: P. Kuchment, E. Semenov,  Contemporary mathematics, Vol. 733, Providence: AMS, 2019, pp. 97--107.

\bibitem{CK} Chakravarty, S., and Y. Kodama, ``Soliton solutions of the KP equation and application to shallow water waves.'' \textit{Stud. Appl. Math.} 123 (2009): 83--151.

\bibitem{CW} Corteel, S., and L.K. Williams, ``Tableaux combinatorics for the asymmetric exclusion process.'' \textit{Adv.
Appl. Math.} 39, no. 3 (2007): 293--310.

\bibitem{DN} Dubrovin, B. A., and S.M. Natanzon, ``Real theta-function solutions of the Kadomtsev-Petviashvili equation.'' \textit{Izv. Akad. Nauk SSSR Ser. Mat.} 52 (1988): 267--286.

\bibitem{FG1} Fock, V.V., and A. B. Goncharov, ``Cluster ${\mathcal X}$--Varieties, Amalgamation and
Poisson-Lie Groups.'', in Algebraic Geometry and Number Theory, dedicated to Drinfeld's 50th birthday, pp. 27--68, \textit{Progr. Math.} 253, Birkhauser, Boston, 2006.

\bibitem{Fom} Fomin, S., ``Loop--erased walks and total positivity.'' \textit{Trans. Amer. Math. Soc.} 353, no. 9 (2001): 3563--3583.

\bibitem{FPS} Fomin, S., P. Pylyavskyy, E. Shustin and D. Thurston , \textit{Morsifications and mutations.} arXiv:1711.10598 (2017).

\bibitem{FZ1} Fomin, S., and A. Zelevinsky, ``Double Bruhat cells and total positivity.'' \textit{J. Amer. Math. Soc.} 12 (1999): 335--380.

\bibitem{FZ2} Fomin S., and A. Zelevinsky, ``Cluster algebras I: foundations.'' \textit{J. Amer. Math. Soc.} 15 (2002): 497--529.

\bibitem{GK} Gantmacher, F.R., and M.G. Krein, ``Sur les matrices oscillatoires.'' \textit{C.R. Acad. Sci. Paris}
201 (1935): 577--579.

\bibitem{GK2} Gantmacher, F.R., and M.G. Krein, \textit{Oscillation Matrices and Kernels and Small Vibrations of Mechanical Systems.} (Russian), Gostekhizdat, Moscow-
Leningrad, (1941), second edition (1950), English edition from AMS Chelsea Publ. (2002).

\bibitem{GSV1} Gekhtman, M.,  M. Shapiro, and A. Vainshtein, ``Poisson geometry of directed networks in a
disk'', \textit{Selecta Math. (N.S.)} 15 (2009): 61--103

\bibitem{GSV}  Gekhtman, M., M. Shapiro, and A. Vainshtein, \textit{Cluster algebras and Poisson geometry.} Mathematical Surveys and Monographs, 167. American Mathematical 
Society, Providence, RI, (2010), xvi+246 pp.

\bibitem{GSV2} Gekhtman M., M. Shapiro, and A. Vainshtein, ``Poisson Geometry of Directed Networks
in an Annulus.'' \textit{J. Eur. Math. Soc.} 14 (2012): 541--570.

\bibitem{GGMS} Gel'fand, I.M., R.M. Goresky, R.D. MacPherson, and V.V. Serganova, ``Combinatorial geometries, convex polyhedra, and Schubert cells.'' \textit{Adv. Math.}  63, no. 3 (1987): 301--316.

\bibitem{GS} Gel'fand, I.M., and V.V. Serganova, ``Combinatorial geometries and torus strata on homogeneous compact manifolds.'' \textit{Russian Math. Surveys} 42, no. 2 (1987): 133--168.

\bibitem{KG} Goncharov, A.B., and R. Kenyon, ``Dimers and cluster integrable systems.'' \textit{Ann. Sci. \'Ec. Norm. Sup\'er.} (4)  46, no. 5 (2013): 747--813.

\bibitem{Kap} Kaplan, J., ``Unraveling ${\mathcal L}_{n;k}$ Grassmannian Kinematics.'' \textit{J. High Energy Phys.} 2010, no. 3, 025, (2010) 34 pp. 


\bibitem{Kas1} Kasteleyn, P.W., ``The statistics of dimers on a lattice.I. The number of dimer arrangements on a quadratics lattice'', \textit{Physica} {\bf 27} (1961), 1209-1225.

\bibitem{Kas2} Kasteleyn, P. {\em Graph theory and crystal physics, in Graph Theory and Theoretical Physics}, Ed. F. Harary, Academic Press, London (1967) pp. 43-110.
  
\bibitem{KO}  Kenyon, R., and A. Okounkov, ``Planar dimers and Harnack curves.'' \textit{Duke Math. J.} 131, no. 3 (2006): 499--524.

\bibitem{KW1} Kodama, Y. and L.K. Williams, ``The Deodhar decomposition of the Grassmannian and the regularity of KP solitons.'' \textit{Adv. Math.} 244 (2013): 979--1032.

\bibitem{KW2} Kodama, Y. and L.K. Williams, ``KP solitons and total positivity for the Grassmannian.'' \textit{Invent. Math.} 198 (2014) 637--699.

\bibitem{Kr4} Krichever, I.M.,  ``Spectral theory of two-dimensional periodic operators and its applications'', \textit{Russian Math. Surveys}, 44, no. 8 (1989): 146--225.

\bibitem{Lam1} Lam, T., ``Dimers, webs, and positroids.'', \textit{J. Lond. Math. Soc.} (2) 92, no. 3 (2015): 633--656.

\bibitem{Lam2} Lam, T., \textit{Totally nonnegative Grassmannian and Grassmann polytopes.}, Current developments in mathematics 2014, 51--152, Int. Press, Somerville, MA, 2016.

\bibitem{Lus1} Lusztig, G.,  ``Total positivity in reductive groups.'' \textit{Lie Theory and Geometry: in honor
of B. Kostant}, Progress in Mathematics 123, Birkh\"auser, Boston, 1994, 531--568.

\bibitem{Lus2} Lusztig, G., ``Total positivity in partial flag manifolds.'' \textit{Represent. Theory} 2 (1998),
70--78.

\bibitem{Mal} Malanyuk, T.M., ``A class of exact solutions of the Kadomtsev–Petviashvili equation.'' \textit{Russian Math. Surveys}, {\bf 46}:3 (1991), 225--227.

\bibitem{MR} Marsh, R. J. and K. Rietsch, ``Parametrizations of flag varieties.'' \textit{Represent. Theory} 8 (2004): 212--242.

\bibitem{MS} Mason, L., and  D. Skinner, ``Dual Superconformal Invariance, Momentum Twistors
and Grassmannians.'' \textit{J. High Energy Phys.} 2009, no. 11, 045, (2009), 39 pp.

\bibitem{OPS} Oh, S., A. Postnikov, and D.E. Speyer, ``Weak separation and plabic graphs.'' \textit{Proc. Lond. Math. Soc.} (3) 110, no. 3 (2015): 3, 721--754.

\bibitem{Pos} Postnikov, A., \textit{Total positivity, Grassmannians, and networks.}, arXiv:math/0609764 [math.CO].

\bibitem{Pos2} Postnikov, A., ``Positive Grassmannian and polyhedral subdivisions'', \textit{Proc. Internat. Congress Math. (ICM 2018), Rio de Janeiro,} {\bf 3}, 3167--3196.

\bibitem{PSW} Postnikov, A., D. Speyer, L. Williams, ``Matching polytopes, toric geometry, and the totally non-negative Grassmannian'', \textit{J. Algebraic Combin.} \textbf{30}, no. 2 (2009): 173--191.
  
\bibitem{Sch} Schoenberg, I., ``\"Uber variationsvermindende lineare Transformationen.'' \textit{Math. Z.} 32, (1930): 321--328.

\bibitem{Sc} Scott J.S., ``Grassmannians and cluster algebras.'' \textit{Proc. London Math. Soc.} 92 (2006): 345--380.

\bibitem{Sp} Speyer, D.E. ``Variations on a theme of Kasteleyn, with application to the totally nonnegative Grassmannian.'' \textit{Electron. J. Combin.} 23, no. 2 (2016) Paper 2.24, 7 pp.

\bibitem{Tal2} Talaska, K.,  ``A Formula for Pl\"ucker Coordinates Associated with a Planar Network.'' \textit{Int. Math. Res. Not. IMRN} 2008, (2008),  Article ID rnn081, 19 pages.

\bibitem{TW} Talaska, K., and L. Williams, ``Network parametrizations for the Grassmannian.'' \textit{Algebra Number Theory} 7, no. 9 (2013): 2275--2311.
 

\end{thebibliography}
\end{document}